\documentclass{amsart}
\usepackage{amsmath,amssymb,graphicx,setspace,verbatim}
\input xy
\xyoption{all}
\newtheorem{prop}{Proposition}[section]
\newtheorem{coro}[prop]{Corollary}
\newtheorem{thm}[prop]{Theorem}

\newcommand{\Alt}{\mathrm{A}} 
\newcommand{\Sym}{\mathrm{S}}

\begin{document}
\title{Highest rank of a polytope for $A_n$}
\author[P. J. Cameron]{Peter J. Cameron}
\address{Peter J. Cameron,
School of Mathematics and Statistics, 
University of St Andrews,
North Haugh,
St Andrews, Fife KY16 9SS,
UK
}
\email{pjc20@st-andrews.ac.uk}

\author[M. E. Fernandes]{Maria Elisa Fernandes}
\address{Maria Elisa Fernandes,
Center for Research and Development in Mathematics and Applications, Department of Mathematics, University of Aveiro, Portugal
}
\email{maria.elisa@ua.pt}

\author[D. Leemans]{Dimitri Leemans}
\address{Dimitri Leemans, Department of Mathematics, University of Auckland, Private Bag 92019, Auckland 1142, New Zealand
}
\email{d.leemans@auckland.ac.nz}

\author[M. Mixer]{Mark Mixer}
\address{Mark Mixer, Department of Applied Mathematics, Wentworth Institute of Technology, Boston, MA 02115, USA
}
\email{mixerm@wit.edu}

\date{}
\maketitle

\begin{abstract}
We prove that the highest rank of a string C-group constructed from an alternating group $\Alt_n$ is 0 if $n=3, 4, 6, 7, 8$; 3 if $n=5$; 4 if $n=9$; 5 if $n=10$; 6 if $n=11$; and $\lfloor\frac{n-1}{2}\rfloor$ if $n\geq 12$.
This solves
a conjecture made by the last three authors in 2012. 
\end{abstract}

\section{Introduction}
Given a group $G$ and a set of involutions $S:=\{\rho_0, \ldots, \rho_{r-1}\}$ which generate $G$, such that 
\[ \forall i,j \mathrm{\;with\;} |i-j|>1, \rho_i \mathrm{\;and\;} \rho_j \mathrm{\;commute\;(the\;\emph{string\;property}}),
\]
we call the pair $(G,S)$ a {\em string group generated by involutions} (or {\em sggi} for short).
We denote by $\Gamma_I$ the group generated by $\{\rho_i:i\in I\}$ for $I\subseteq\{0,\ldots,r-1\}$.
The pair $(G,S)$ satisfies the {\em intersection property} if for every $I,J \subseteq\{0,\ldots,r-1\}$,
$\Gamma_I\cap\Gamma_J=\Gamma_{I\cap J}$.
A sggi $\Gamma:= (G,S)$ that satisfies the intersection property is called a {\em string C-group} of {\em rank $|S|$}.
If $\Gamma:=(G,S)$ is a string C-group, we sometimes will abuse language and talk about the group $\Gamma$ and denote the {\em rank} of $G$ as the largest size of a set $S$ of involutions such that $\Gamma:=(G,S)$ is a string C-group.
For $i \in \{0, \ldots, r-1\}$, we denote by $\Gamma_i$ the group generated by all the elements of $S$ except $\rho_i$.

It is known that string C-groups are automorphism
groups of abstract regular polytopes and that, given an abstract regular polytope and a base flag of the polytope, one can construct a string C-group whose group $G$ is the automorphism group of the polytope~\cite[Section 2E]{ARP}. Hence the study of string C-groups has interests not only in group theory, but also in geometry.

Classifications of string C-groups from almost simple groups started with experimental work of Leemans and Vauthier~\cite{LVatlas} (see also~\cite{Halg, HHalg, LMalg, CLM2012} for more experimental results) and quickly led to the determination of the rank of a string C-group of Suzuki type~\cite{Leemans:2006}. A series of results then followed for the almost simple groups with socle $\mathrm{PSL}(2,q)$~\cite{ls07,ls09,DiJuTho}, groups $\mathrm{PSL}(3,q)$ and $\mathrm{PGL}(3,q)$~\cite{Brooksbank:2010}, groups $\mathrm{PSL}(4,q)$~\cite{Brooksbank:2015}, small Ree groups~\cite{Leemans:2015} and finally, symmetric groups~\cite{fl,Corr} and alternating groups~\cite{flm,flm2}. In particular, only the last two families gave rise to string C-groups of arbitrary large rank. It was proved in \cite{transitive} that the maximal rank of a string C-group for transitive subgroups of $S_n$ that are neither $\Alt_n$ nor $\Sym_n$ is $\frac{n}{2}+1$.
 A symmetric group $\Sym_n$ is known to have rank $n-1$~\cite{fl} and an alternating group $\Alt_n$ with $n\geq 12$ is known to have rank at least $\lfloor\frac{n-1}{2}\rfloor$ when $n\geq 12$~\cite{flm2}.  It is conjectured in~\cite{flm2} that this is the highest possible rank for a string C-group of alternating type.
In this paper, we prove this conjecture. Our main result is as follows.

\begin{thm}\label{maintheorem}
The rank of $\Alt_n$ is $0$ if $n=3, 4, 6, 7, 8$; $3$ if $n=5$; $4$ if $n=9$; $5$ if $n=10$; $6$ if $n=11$ and $\lfloor\frac{n-1}{2}\rfloor$ if $n\geq 12$.
\end{thm}

The cases where $n\leq 11$ had already been dealt with the use of {\sc Magma}~\cite{flm}.
In this paper, we show that if $\Gamma:=(\Alt_n,S)$ is a string C-group and $n\geq 12$, then $|S| \leq \lfloor\frac{n-1}{2}\rfloor$.
In some parts of this proof we use induction over $n$.
The proof is divided into three parts. In Sections~\ref{s:prim} and
\ref{s:imprim}, we deal with the case where some subgroup $\Gamma_i$ is
primitive or transitive imprimitive, respectively, and our main tool here is permutation group theory. In the remainder of the
    paper we have to deal with the case where all $\Gamma_i$'s are intransitive.
Our main tool for this case is the use of \emph{fracture graphs}; but these
are also used elsewhere, so we give a brief introduction here.

Let $\Gamma=\langle\rho_0,\,\ldots,\,\rho_{r-1}\rangle$ be an sggi acting as a permutation group on a set $\{1,\,\ldots,\,n\}$.
We define the {\it permutation representation graph} $\mathcal{G}$ as the $r$-edge-labeled multigraph with $n$ vertices and with a single $i$-edge $\{a,\,b\}$ whenever $a\rho_i=b$ with $a\neq b$.  
Suppose we have a sggi $\Gamma$ which is a transitive subgroup of
the symmetric group $S_n$, such that every subgroup $\Gamma_i$ is intransitive.
Then, for each $i$, the permutation $\rho_i$ has a cycle whose points lie in
different $\Gamma_i$-orbits. Choosing one such cycle for each $i$, and
regarding them as the edges of a graph on the vertex set $\{1,\ldots,n\}$,
we obtain a \emph{fracture graph} for $\Gamma$. The fracture graph is of
course not unique, and indeed much of our proof involves showing how to
replace a fracture graph by a more convenient one.

If $\Gamma$ is contained in the group of even permutations (as 
in our main theorem), then each permutation $\rho_i$ has at least two cycles.
If it happens that for each $i$ we can find two pairs of points in
different $\Gamma_i$-orbits, then taking an $i$-edge between each of these pairs of points we obtain a
\emph{$2$-fracture graph}. Section~\ref{s:2frac} handles the case where a
$2$-fracture graph exists. Section~\ref{s:no2frac} the case where it
does not, and we then use fracture graphs instead.

\section{$\Gamma_i$ is primitive for some $i$}
\label{s:prim}

Now we embark on the proof of the main theorem.  In this section we prove the
theorem in the case where some $\Gamma_i$ is primitive.

Given a string C-group $\Phi:=(G,T)$ with $T:= \{\rho_0, \ldots, \rho_{d-1}\}$, the {\em diagram} of $\Phi$ is a graph with $d$ vertices   and an edge between vertices $i$ and $j$ whenever $\rho_i\rho_j$ is not an involution. Moreover, the edge $\{i,j\}$ is then labelled with the order of  $\rho_i\rho_j$. 
Observe that, by the string property, the diagram of $\Phi$ is a union of paths.  
We say that a set $S\subseteq T$ is {\em connected} provided the labels of the generators of $S$ form an interval. 

Let us first state a Theorem due to Mar\'oti that will be useful in the proof of the next proposition and also later on.
\begin{thm}[Mar\'oti~\cite{Maroti}]\label{marotiThm}
Let $G$ be a primitive group of degree~$n$ which is not $S_n$ or $A_n$. Then
one of the following possibilities occurs:
\begin{enumerate}
\item For some integers $m,k,l$, we have $n={m\choose k}^l$, and $G$ is a
subgroup of $S_m\wr S_l$, where $S_m$ is acting on $k$-subsets of
$\{1,\ldots,m\}$;
\item $G$ is $M_{11}$, $M_{12}$, $M_{23}$ or $M_{24}$ in its natural
$4$-transitive action;
\item $\displaystyle{|G|\le n\cdot\prod_{i=0}^{\lfloor \log_2n\rfloor-1}
(n-2^i)}$.
\end{enumerate}
\end{thm}

\begin{prop}\label{prim}
Let $n\geq 12$. If $\Phi:=(G,T)$ is a string C-group of rank $d$ with $G<A_n$ and $G$ primitive, then $d\leq (n-3)/2$.
\end{prop}

Indeed, in this case $d$ is asymptotically much smaller than $n/2$.

\begin{proof}
We use the methods of~\cite{transitive}.

Suppose first that the diagram of $\Phi$ is not connected. Then the 
primitive group $G$ is the direct product of two proper subgroups,
each of which is necessarily simple and acts regularly; so $|G|=n^2$, and $n\ge60$. But
clearly $|G|\ge2^{d}$; so $d\le 2\log_2n<(n-3)/2$ for $n\ge60$.

So we may suppose that the diagram of $\Phi$ is connected. Now we combine Conder's
lower bound $2^{2d-1}$ for the order of a string C-group of rank $d$
\cite{Conder} with well-known upper bounds for the order of primitive groups,
such as Mar\'oti's (see Theorem~\ref{marotiThm}). We deal with the three cases of Mar\'oti's Theorem.
Case (b) is handled by computer. In case (a), since we are only interested in
an upper bound for $|G|$, we can assume that $G$ is maximal in $A_n$, so that
either $G$ is $S_m$ acting on $k$-sets, or $G=S_m\wr S_l$ with $l>1$. In the first
subcase, $d\leq m-1$, while $n={m\choose k}\ge m(m-1)/2$, hence  $d\leq \frac{n-3}{2}$ for $n\geq 12$. In the second subcase, we can use the main result of \cite{transitive} to conclude that, $d\leq \frac{ml}{2}+1$ while $n={m\choose k}^l$ and $n\geq 12$; again this gives $d\leq \frac{n-3}{2}$.
In these cases, $G$ is embeddable in a smaller symmetric group, of degree $m$
in the first case, or $ml$ in the second.
Finally, in case (c), we have 
\[2^{2d-1}\leq |G|\le n^{1+\log_2n},\] 

If we assume $d\geq \frac{n-2}{2}$, we get

\[2^{n-3}\leq 2^{2d-1}\leq |G|\le n^{1+\log_2n},\] 

thus

\[n\leq (\log_2n)(\log_2n+1)+3\]
which 
gives a contradiction for $n\geq 34$.

For $n\leq 33$, we give in Table~\ref{primBound} the list of primitive groups of degree $n$ such that their order is $\geq 2^{2\lfloor\frac{n}{2}\rfloor-3}$, following numbering of Sims's list~\cite{BL96}.
When {\sc Magma} is mentioned in the references column, it means we computed all string C-groups representations of the corresponding group using {\sc Magma} and the bound is sharp.

\begin{table}
\begin{tabular}{|c|c|c|c|c|}
\hline
Degree&Number&$G$&Max rank&Reference\\
\hline
12&1&$PSL(2,11)$&4&\cite{LVatlas,ls07}\\
&2&$PGL(2,11)$&3&\cite{LVatlas,ls09}\\
&3&$M_{11}$&0&\cite{LVatlas}\\
&4&$M_{12}$&4&\cite{LVatlas}\\
\hline
13&7&$PSL(3,3)$&0&\cite{LVatlas,Brooksbank:2010}\\
\hline
14&2&$PGL(2,13)$&3&\cite{LVatlas,ls09}\\
\hline
15&3&$A_7$&0&\cite{LVatlas}\\
&4&$PSL(4,2)$&0&{\sc Magma}\\
\hline
16&18&$2^4:S_6$&5&{\sc Magma}\\
&19&$2^4:A_7$&0&{\sc Magma}\\
&20&$2^4:PSL(4,2)$&0&{\sc Magma}\\
\hline
17&8&$P\Gamma L(2,16)$&0&\cite{LVatlas}\\
\hline
22&2&$M_{22}:2$&4&\cite{LVatlas}\\
\hline
23&5&$M_{23}$&0&\cite{HHalg}\\
\hline
24&3&$M_{24}$&5&\cite{HHalg}\\
\hline
\end{tabular}
\caption{Primitive groups $G$ of degree $\leq 33$ with $|G| \geq  2^{2\lfloor\frac{n}{2}\rfloor-3}$.}\label{primBound}
\end{table}

So $d\le(n-3)/2$ in all cases.
\end{proof}

We remark that Mar\'oti's bound uses the Classification of Finite Simple Groups.
The use of such heavy machinery could in principle be avoided by using the
slightly weaker bounds proved by `elementary' means by Babai and
Pyber~\cite{Babai,Pyber}; however, this would require examining of many more
`small' cases, some of which are too large for practical computation.

The previous proposition gives the following corollary that finishes a case for our main theorem, when some $\Gamma_i$ is primitive.
\begin{coro}
Let $n\geq 12$ and let $\Gamma:=(A_n,S)$ be a string C-group of rank $r$. If $\Gamma_i$ is primitive for some $i$ then $r\leq \frac{n-1}{2}$.
\end{coro}
\begin{proof}
If $\Gamma_i$ is primitive for some $i$, then $\Gamma_i < A_n$ and satisfies the hypotheses of Proposition~\ref{prim}. Hence the rank $r-1$ of $\Gamma_i$ is bounded by $\frac{n-3}{2}$. 
\end{proof}

\section{$\Gamma_i$ is transitive imprimitive for some $i$}
\label{s:imprim}

In this section, we prove the main theorem in the case where $\Gamma_i$ is
transitive but imprimitive for some $i$.

Let $\Gamma\cong A_n$ with $n\geq 12$ and $\Gamma_i$ be transitive imprimitive for some $i\in\{0,\ldots, r-1\}$.
Let $k$ and $m$ be such that $\Gamma_i$ is embedded into $S_k\wr S_m$. We assume that the blocks of imprimitivity are maximal (so $\Gamma_i$ acts primitively
on the set of blocks), but do not require that $k$ is as big as possible.

Consider the following sets of generators of $\Gamma_i$:
\begin{itemize}\itemsep0pt
\item $L$ an independent generating set for the block action;
\item $C$ the set of generators that commute with all elements of $L$;
\item $R$ the remaining generators.
\end{itemize}
Let us first recall some important results found in \cite{transitive}. We have that $|L|\leq m-1$ and  $|C|\leq k-1$. The group $\langle L\rangle$ is primitive on the set of blocks, and $L$ has at most two connected components. When $L$ has two components, $r-1\leq2\log_2m+(k-1)+4$ and $m\geq 60$, thus in that case we have $r\leq \frac{n}{2}-1$. So, in what follows, we assume that $L$ is connected and generates a primitive group on the set of blocks.

\begin{prop}\label{R<=1}
If $m\neq 2$ then $\{\rho_i\}\cup L$ must be connected and $|R|\leq 1$.
\end{prop}

\begin{proof}
As $\Gamma$ is primitive, $\rho_i$ must break the imprimitivity of $\Gamma_i$, thus it must swap at least one pair of points in different blocks.
On the other hand  $\langle L\rangle$ is primitive, thus $\rho_i$ cannot commute with every element of $L$.  Hence $\{\rho_i\}\cup L$ must be connected and $|R|\leq 1$.
\end{proof}

\subsection{The case $k,m>2$}

\begin{prop}
If $k>2$ and $m>2$, then $r\le\lfloor(n-1)/2\rfloor$.
\end{prop}

\begin{proof}
As observed at the beginning of the section,  $|L|\leq m-1$, $|C|\leq k-1$. By Proposition~\ref{R<=1}, $|R|\leq 1$, hence $r-1\leq (m-1)+(k-1)+1$. When $n=12$ the bound that we get for the rank is $7$, but using  {\sc Magma}~\cite{BCP97} we found out that there are no polytopes of ranks $6$ or $7$ for $A_{12}$. So we may assume that $n>12$. 

If $r>\lfloor(n-1)/2\rfloor$, then
\[\frac{n-1}{2}=\frac{km-1}{2}\leq r\leq k+m,\]
so $(k-2)(m-2)\le5$. The solutions with $km>12$ are $(k,m)=(3,5)$, $(3,6)$,
$(3,7)$, $(4,4)$, $(5,3)$, $(6,3)$, $(7,3)$.

Now we consider these cases. If $(k,m)=(3,7)$ or $(7,3)$, then $r\le10=(21-1)/2$, as required. If $(k,m)=(3,5)$, $(3,6)$, $(4,4)$, $(5,3)$ or
$(6,3)$, then we have $r\le\lfloor(n-1)/2\rfloor$ unless $|L|=m-1$, $|C|=k-1$,
and $|R|=1$. So $\langle C\rangle\cong S_k$, and since $\langle C \rangle $ commutes with a group
acting primitively on the blocks, it acts in the same way on each block.

We also see that $\langle L\rangle$ acts as $S_m$ on the set of blocks, and
since it commutes with $S_k$ fixing the blocks, we have
$\langle L,C\rangle\cong S_k\times S_m$. Transpositions in $S_k$ (resp.~$S_m$)
act as products of $m$ (resp.~$k$) transpositions on the point set. So if either
$m$ or $k$ is odd, then $\Gamma$ contains an odd permutation, a contradiction.

Now suppose that $(k,m)=(4,4)$ and $r=8$. We know that $\rho_i$ commutes with a subgroup
$S_3\times S_4$ with orbits of sizes $4$ and $12$.
We know from the previous paragraph that $S_4\times S_4$
is acting on the product of two sets of size $4$. So when we descend to
$S_3\times S_4$, the orbit of size $4$ has $S_4$ acting in the usual action (and its
centraliser is trivial), while the orbit of size $12$ is the product of sets
of sizes $3$ and $4$. A permutation which commutes with it must fix the two
systems of imprimitivity, so its projection onto each factor commutes with
the corresponding symmetric group, and so is trivial.
\end{proof}

\subsection{The case $k=2$}

The estimate above gives $r\le m+2=\frac{n}{2}+2$. We have to knock three off
this bound. The group induced on the blocks is primitive. It follows that
the centraliser of $\langle L\rangle$ in the symmetric group is generated by
the involution $z$ which interchanges the points of each block. Now if $m$ is
odd, then $z$ is an odd permutation, and so $C=\emptyset$. If $m$ is even, then
$|C|\le1$, and if $z\in\langle L\rangle$ then the intersection property forces
$C=\emptyset$.

We separate the argument into three cases, according as the group $H$ induced on blocks is
$S_m$, $A_m$, or neither of these.  The cases $H=S_m$ and $H=A_m$ use similar
arguments, but differ in detail, so we have kept them separate.

\paragraph{\underline{Case $H=S_m$}}  We assume that $m\ge7$ for this proof.

 Note first that $|R\cup C|\le 1$. Indeed, if both $R$ and $C$ are nonempty, $m$ is even,
and $z$ is contained in $\langle L\cup R\rangle\cap \langle C\rangle$, a contradiction. [This is because the kernel
of the action of $\Gamma_i$ on blocks is an $S_m$-submodule of $(C_2)^m$;
the only such submodules are the trivial ones, the module $M_1$ generated
by $z$, and the module $M_2$ consisting of elements interchanging an even
number of blocks; if $m$ is even then $M_1\le M_2$, and there cannot exist
two independent submodules.]

So $r\le|L|+2$, and if either $|L|\le m-3$ or
$R\cup C=\emptyset$ and $|L|\le m-2$ then we have the required result. 
Up to duality, there are the three possibilities, either
\begin{itemize}
\item[(A)] $\rho_i=\rho_0, \, L=\{\rho_1,\ldots, \rho_{r-2}\}, \, R=\{\rho_{r-1}\}\mbox{ and }C=\emptyset,$
\item[(B)] $\rho_i=\rho_{r-2},\, L=\{\rho_0,\ldots, \rho_{r-3}\}, \, C=\{\rho_{r-1}\}\mbox{ and }R=\emptyset,$ or 
\item[(C)] $\rho_i=\rho_0,\, L=\{\rho_0,\ldots, \rho_{r-1}\}, \mbox{ and }R=C=\emptyset.$
\end{itemize}
Let $G=\langle L\rangle$. Now $G$ induces the symmetric group $S_m$ on the
set of blocks, and $L$ is an independent set of generators for $G$ as a
string group (not necessarily a string C-group!). We have $|L|\le m-1$.

Assume that $|L|=m-1$ and $m\ge7$. The elements of $L$ induce the Coxeter
generators on the set of blocks: for a certain numbering of the blocks, $\rho_j$ swaps blocks $j$ and $j+1$, for $j\in\{1,\ldots,m-1\}$. (This
is an easy deduction from the result of \cite{CC}.)
In (A) and (C), $\rho_i$ commutes with $\rho_2,\ldots,\rho_{m-1}$, and these elements
generate a group acting as $S_{m-1}$ on blocks, fixing the first block.
Since $m\ge7$, we see that $\rho_i$ must fix all the blocks numbered from
$2$ to $m$, and clearly also block $1$; so it preserves the block system.
In (B) $\rho_i$ commutes with $\rho_0,\ldots,\rho_{m-3}$, which
also acts as $S_{m-1}$ on blocks, and the same applies.  So $\Gamma$ preserves
the block system, and is imprimitive, a contradiction to the assumption that
$\Gamma$ is the alternating group. 

So we can assume that $|L|=m-2$ and $R\cup C\neq\emptyset.$

Let $K$ be the kernel of the action of $\Gamma_i$ on
the blocks, and let $K_1=K\cap G$. Then $K$ and $K_1$ are $S_m$-submodules of
the permutation module $F^m$, where $F$ is the field with two elements. The
only submodules have dimensions $0$, $1$ (spanned by the all-$1$ vector),
$m-1$ (the vectors of even weight), and $m$. Now since $\Gamma_i$ consists
of even permutations, we cannot have $K=F^m$.

We show that $K=K_1$ is impossible. If $K=K_1$, then $G=\Gamma_i$, and so
$R\cup C=\emptyset$, contrary to our assumption.

Next we show that $K_1=1$ is impossible. In this case, $L$ generates $S_m$ as
string C-group. By the main result of ~\cite{fl,Corr}, there is a unique 
possibility, up to duality. In (A) the group generated by 
$\rho_2,\ldots,\rho_{r-2}$ is $S_{m-1}$, and $\rho_0$ commutes with this
group, so $\rho_0$ preserves the block system, a contradiction. In (B) the group generated by $\rho_0,\ldots,\rho_{r-4}$ is $S_{m-1}$, and
$\rho_i$ commutes with this group.
We then get the same contradiction as before.


So we are left with the case $K=[(C_2)^m]^+$ and $K_1=\langle z\rangle$.
In this case $G$ is an extension of $C_2$ by $S_m$ and $C=\emptyset$ (as in case (A)).

The involutions $z$ and $\rho_0$ both commute with $\Gamma_{0,1}$, since
$z$ is in the centre of $\Gamma_0$. So the dihedral group $D$ they generate also
commutes with $\Gamma_{0,1}$. Moreover, $\rho_0$ and $z$ do not commute with
each other; if they did, then $\Gamma=\langle\Gamma_0,\rho_0\rangle$ would be
contained in the centralizer of $z$, contradicting the fact that $\Gamma$ is
the alternating group. So $D$ has order $2d$ with $d\geq 3$. We now separate in two
cases.

In the case where $\Gamma_{0,1}$ is transitive, the group $D$ is semiregular;
thus $\Gamma_{0,1}$ has $m/d$ blocks of imprimitivity
each of size $2d$, and is contained in $D\wr S_{m/d}$. Now since $d\ge3$, we
can replace the action of $\Gamma_{0,1}$ by one where each orbit of $D$ has
size $d$, rather than $2d$; this action is still faithful (since $D$ acts
faithfully on $d$ points if $d\ge3$). So $\Gamma_{0,1}$ is isomorphic to a
transitive imprimitive group of degree $m$. By the main result of
\cite{transitive}, we have $r-2\le\frac{m}{2}+1$ (whence $m\le 6$,
which is not so).

So suppose that $\Gamma_{0,1}$ is intransitive on blocks; then also
$\Gamma_{0,1,r-1}$ is intransitive. Let $H_1$ be the group it induces on the
blocks. Now the images of $\rho_2,\ldots,\rho_{r-2}$ form a set of $r-3$
generators for $H_1$ as an sggi (not necessarily a string C-group). If $r\ge m$,
then we conclude that
\begin{itemize}\itemsep0pt
\item  $H_1$ has at most three orbits;
\item if it has three orbits, then it acts on each as the symmetric group;
\item if it has two orbits, then it acts on one of them as the symmetric
group.
\end{itemize}
Suppose there are three orbits. Then $H_1$ commutes with the group induced by $D$
(which has at least one orbit of size $d\ge3$), the three orbits must be
isomorphic and a $D$-orbit meets each in one point. But then $H_1\le S_{m/3}$
and the number of independent generators for $H_1$ is at most $m/3-1$. So we
have $m/3-1\ge m-3$, which is impossible.

Suppose that $H_1$ has two orbits $O_1$ and $O_2$ with the action of $H_1$ on $O_1$ being that of the symmetric group $S_{|O_1|}$. We have a dihedral group $D$ commuting with the symmetric group such that each $D$-orbit meets $O_1$ in one point, as these intersections form a system of imprimitivity for $D$ on $O_1$.
Suppose that
the action on the other orbit is not faithful. Then there is a non-trivial
subgroup fixing all points in this orbit (and hence fixing all $D$-orbits)
but non-trivially on $O_1$, and so moving the intersections of $D$-orbits
with $O_1$ (since these have size $1$, only the trivial group fixes them all).
As to the size, each $D$-orbit has one point in $O_1$ and $d-1$ in $O_2$,
so $|O_2|$ is $(d-1)/2$ times the degree, that is, $(d-1)m/d$.
Thus $H_1$ has at most $(d-1)m/d-1$ independent generators. If our
inequality holds, then $(d-1)m/d-1\ge m-3$, from which we get $m\le 2d$. But
then the dihedral group has at most two orbits, and
$\Gamma_{0,1}\le D\wr C_2$. 
A group of order $2m^2$ has
largest independent set of size at most $2\log_2m+1$.
This number cannot be $m-2$ or greater for 
$m>8$; the remaining cases are resolved by a computer check.

\paragraph{\underline{Case $H=A_m$}} As before, let $L$ be an independent set of generators
for the action of $\Gamma_i$ on blocks, and let $G=\langle L\rangle$. If
$G$ is intransitive, then its orbits form a transversal for the blocks, and
so $G\cong A_m$. By the induction hypothesis, if $m\ge12$, then
$|L|\le\lfloor\frac{m-1}{2}\rfloor$, and so 
$r\le\lfloor\frac{m+3}{2}\rfloor\le m-1$, since $m\ge7$. For $m<12$, if
the bound fails, we have either $m=10$ and $|L|=5$, or $m=11$ and $|L|=6$; the
required bound is satisfied in either case.

So we may assume that $G$ is transitive.

Let $K$ be the kernel of the action of $\Gamma_i$ on blocks, and $K'=K\cap G$.
Since $\Gamma_i/K$ and $G/K'$ are both isomorphic to $A_m$, we see that 
$K\ne K'$. Moreover, $A_m$ cannot act transitively on $2m$ points, and so
$K'\ne 1$. Since $K$ and $K'$ are submodules of the $A_m$-module $(C_2)^m$,
and neither is the whole of $(C_2)^m$ (which contains odd permutations), we
must have $|K|=2^{m-1}$, $|K'|=2$. The generator of $K'$ is the involution $z$
which interchanges the points in each block. Since this is an even permutation,
$m$ must be even. Moreover $C=\emptyset$.  Thus up to duality we may assume that $i=0$.
If $R=\emptyset$ then $r\leq m-1$, hence we now assume that $R=\{\rho_{r-1}\}$.

Now $z$ is in the centre of $\Gamma_0$, and so commutes with $\Gamma_{0,1}$.
The involution $\rho_0$ also commutes with $\Gamma_{0,1}$, and by the
intersection property $\rho_0\ne z$. Let $D=\langle z,\rho_0\rangle$, a
dihedral group of order $2d$, say. Now $\rho_0$ and $z$ do not commute: for,
if they did, then $\langle\rho_0,\Gamma_0\rangle$ would be contained in the
centraliser of $z$, whereas in fact this group is $A_{2m}$. In particular,
$d\ge3$.

Suppose first that $\Gamma_{0,1}$ is transitive. Then $D$, which commutes
with a transitive group, is semiregular; and $\Gamma_{0,1}$, which commutes
with $D$, is isomorphic to a subgroup of $D\wr S_{m/d}$. Since $D$ acts
faithfully on $d$ points (as $d\ge3$), $\Gamma_{0,1}$ is isomorphic to a transitive
imprimitive group on $m$ points. By the main theorem of \cite{transitive},
we get $r-2\le\frac{m}{2}+1$, so $m\le6$.

Now suppose that $\Gamma_{0,1}$ is intransitive.

We know that $\langle\rho_1,\ldots,\rho_{r-2}\rangle\cong C_2\times A_m$,
and $\{\rho_1,\ldots,\rho_{r-2}\}$ are string C-group generators. We claim that their images, $\bar{\rho}_1,\ldots,\bar{\rho}_{r-2}$,
in $\Gamma_{0,r-1}/\langle z\rangle\cong A_m$ are independent. 

Suppose not. Note that they generate $A_m$ as an sggi. If they fail to be
independent, one of them can be expressed in terms of the others. Suppose that
it is $\bar{\rho}_h$. We cannot have $1<h<r-2$, since then these elements would
generate a commuting product of two subgroups. We cannot have $h=1$, since
we are assuming that $\Gamma_{0,1}$ is intransitive. And finally, we cannot
have $h=r-2$. For if so, then
$C_{A_{2m}}(A_m)=\langle z\rangle$; but $\rho_{r-1}$ centralises
$\langle\rho_1,\ldots,\rho_{r-3}\rangle=A_m$, so $\rho_{r-1}=z$, contradicting
the intersection property, since $z\in\langle\rho_1,\ldots,\rho_{r-2}\rangle$.

Now the images mod $z$ of $\rho_2,\ldots,\rho_{r-2}$ are independent, and
generate an intransitive subgroup of $A_m$. So $r-3\le m-3$.  If equality
holds, then this group has just two orbits; it acts on each orbit as the
symmetric group.
But this contradicts the fact that these elements belong to
$\Gamma_{0,1}$, which centralises the dihedral group $D=\langle\rho_0,z\rangle$
having at least one orbit of size greater than $2$ on the set of blocks.
\paragraph{\underline{Case $H\ne S_m,A_m$}}
In this case we prove the following result on independent sets.

\begin{prop} \label{IGSprim}
Let $G$ be a primitive group of degree $n\ge8$, not isomorphic to $A_n$ or
$S_n$. Then the maximum size of an independent generating set of $G$ is at
most $n-4$.
\end{prop}
\begin{proof}
Let $M(G)$ be the maximum size of an independent generating set of $G$.

We consider separately the three possibilities given by Theorem~\ref{marotiThm}. 

In case (a)  when $l\geq 2$ we have 
$n-4 = (^m_k)^l-4\geq m^l-4\geq ml-2\geq M(G)$. When $l=1$ we have $k\geq 2$, $m\leq n/2$ and the group is a subgroup of $S_m$ or $A_m$, so $M(G) \leq m-1$, much smaller than $n-4$.

In case (b) we have to consider the groups $M_{11}, M_{12}, M_{23}$ or $M_{24}$. The maximal length of a chain of subgroups of $M_{11}$, $M_{23}$ or $M_{24}$ is 7, 11 and 14 resp. (see \cite{Whis00}). 
If $G$ is isomorphic to $M_{12}$ then $M(G) \leq 9$ by~\cite{Whis00}. 
Suppose that $M(G) = 9$. Then one of the following subgroups $H$ of $M_{12}$, namely $M_{11}$ or $P\Gamma L(2,9)$, has to have $M(H) = 8$.
As $M(M_{11}) \leq 7$ (see~\cite{Whis00}), we must have $M(P\Gamma L(2,9))=8$. A quick look at the subgroup lattice of $M_{12}$ shows that this is impossible as two subgroups of order 1440 never intersect in a subgroup of order 720. 
Hence $M(M_{12}) \leq 8$.

In case (c) the chain length is bounded by $\log_2 \left[n.\prod_{i=0}^{\lfloor\log _2 n-1\rfloor}(n-2^i)\right]$ that is at most $n-4$ for $n\geq 26$.
We also know that if $|G| = p_1^{e_1}\ldots p_k^{e_k}$ then the chain length (and hence $M(G)$) is bounded by $e_1+ \ldots + e_k$.
Combining these two bounds, and using {\sc Magma}, we conclude the result holds for $n>9$ and for $n=8$ we are left with $P\Gamma L_1(8)$, $PSL_2(7)$, $PGL_2(7)$  and $ASL_3(2)$. But for those, looking at the subgroup lattice we get $M(G)\leq 4$.
\end{proof}

\subsection{The case $m=2$}

Suppose that $m=2$, so that $\Gamma_i\le S_k\wr S_2$. 
An involution interchanging the blocks is fixed-point-free so $k$ is even and $n = 0 \mod 4$.

We separate the argument into two cases, according as there is or is not a
value of $j\ne i$ such that $\Gamma_{i,j}$ is transitive. First, suppose
that such a $j$ exists.

\begin{prop}
If $\Gamma_{i,j}$ is transitive for some $j\neq i$ and $\Gamma_i$ is transitive imprimitive embedded into $S_{n/2}\wr S_2$, then $r\leq \frac{n-1}{2}$.
\end{prop}
\begin{proof}

Suppose $r>\frac{n-1}{2}$ and $\Gamma_{i,j}$ is transitive for $j\neq i$.
 Then the two
groups $\Gamma_i$ and $\Gamma_j$ are transitive, so each has two
blocks of size $n/2$. The stabilisers of the blocks for the two subgroups
each have index $2$, so their intersection is a normal subgroup of index $4$
in $\Gamma_{i,j}$.
The two block systems cannot be the same, since then they
would be preserved by $\langle\Gamma_i,\Gamma_j\rangle=\Gamma$. These
intersections are blocks for $\Gamma_{i,j}$, of size $n/4=l$, say, and the action
on the blocks (which is not primitive in this case) is isomorphic to the Klein
group $C_2\times C_2$ generated, modulo the normal subgroup, by a set $L$ of
size $2$.

Now we play the usual game: let $C$ be the set of generators commuting with
$L$ (so $C$ acts in the same way on each block, and has rank at most $l-1$),
and $R$ the remaining generators of $\Gamma_{i,j}$, so $|R|\le4$.

Thus, $r-2=|R|+|C|+|L|\le l+5$. As $r>\frac{n-1}{2}$, $l+7\ge2l$, so $l\le7$, and $n\le 28$. Since $n$ is a multiple of $4$,
we only need to consider the cases $n=16$, $20$, $24$ and $28$.

To deal with the exceptions, we first subdivide into two cases, according as
the permutations of $L$ commute. Let $L=\{\rho_s, \rho_r\}$.

Suppose first that $\rho_s$ and $\rho_r$  do not commute. Then they are
adjacent in the diagram; so we can improve our estimate to $|R|\le2$. Also,
there is a vertex $v$ such that, if we follow a path with labels $r,s,r,s$,
we arrive at a different point $w$ in the same block. Then the stabilizer of
$v$ in $\langle C\rangle$ also fixes $w$. So $\langle C\rangle$ is not the
symmetric group, and we have $|C|\le l-2$. Then we have $r-2\le l+2$, so
$n\le16$.

Now suppose that $\rho_s$ and $\rho_r$ commute. If $|R|=4$, or if
$|R|=3$ and the diagram of $C$ is connected, then at
least one element of $R$, say $\rho_h$, also commutes with $C$. If
$\langle C\rangle$ acts primitively on a block, then the centralizer of $\langle C\rangle$
is generated by $\rho_r$ and $\rho_s$, and so
$\rho_h\in\langle\rho_r,\rho_s\rangle$, a contradiction. So we conclude that
either $C$ is disconnected (giving $|C|\le l-2$), or $\langle C\rangle$ is imprimitive (giving $|C|\le l/2+1$)
or $|R|\le 2$. Putting these into our estimates shows that $n\le24$.

If $n=24$, in the worst case scenario, we have $|R| = 4$, $|C| = 4$ and $|L| = 2$ possibly giving $r=12$.
If there are two permutations $\rho_a$ and $\rho_b$ such that $\Gamma_{i,j,a,b}$ is transitive, then it may be assumed that $\Gamma_a$ and $\Gamma_b$ are both embedded into $S_{n/2}\wr S_2$.
This forces $n$ to be divisible by 16, a contradiction. If $|R|\geq 2$ then $\langle C\rangle$  is intransitive within the blocks, otherwise another two generators can be removed from $\Gamma_{i,j}$ and the group remains transitive. Let $\langle C\rangle$  be intransitive within the blocks. If $|C|=4$ then one element of $R$ fuses the orbits of $\langle C\rangle$ inside a block, hence $|R|\leq 2$, by the same reason. Thus the case $|C|=|R|=4$ leads to a contradiction.

We eliminate the case $n=20$ as follows. For $n=20$, we must have either
$|C|=|R|=3$, or $|C|=4$, $|R|=2$. In the first case, the diagram of $C$
is disconnected, so we must have $C\cong S_3\times S_2$, with orbit lengths
$3$ and $2$. Take an element of $R$ which commutes with one component of
the diagram of $C$, and modify it using $\rho_i$ and $\rho_j$ so that it
fixes all the blocks; this element commutes with $S_2\times S_3$, but this
group contains its centraliser in $S_5$, a contradiction. In the second
case, $\langle C\rangle=S_5$. Then $\langle C,\rho_j\rangle\cong S_5\times S_2$
fixes the blocks of imprimitivity for $\Gamma_i$. But this group is maximal
in the stabiliser of this block system. There can be at most one more generator
in $\Gamma_i$, giving $|R|\le1$, a contradiction.

For $n=16$, we used the computer to show the nonexistence of such a group.
\end{proof}

In what follows we consider that $\Gamma_{i,j}$ is intransitive for every $j\neq i$. We prove that $r\leq \frac{n-1}{2}$ using a certain fracture graph and we use the following results that are immediate consequences of the definition of a fracture graph of a given sggi $G$ with permutation representation graph $\mathcal{G}$. By an alternating square, in $\mathcal{G}$, we mean a square having opposed edges with the same label.
\begin{itemize}
\item (\cite{extension}, Proposition 3.2 ) if two edges of $\mathcal{G}$, $e$ and $e'$, have the same label and $e$ is an edge of a fracture graph, then at least one vertex of $e'$ is in a different component of that fracture graph;
\item  (\cite{extension}, Proposition 3.6) when two edges of an alternating  square  of $\mathcal{G}$, belong to a fracture graph of $G$, the vertex of the square not in these edges is in another component of that fracture graph.  
\end{itemize}
When representing $\mathcal{G}$, dashed edges are used to represent edges that are in  $\mathcal{G}$ but not in the chosen fracture graph.

\begin{prop}
If $\Gamma_{i,j}$ is intransitive for every $j\neq i$ and $\Gamma_i$ is transitive imprimitive embedded into $S_{n/2}\wr S_2$, then $r\leq \frac{n-1}{2}$.
\end{prop}
\begin{proof}
Let $L=\{\rho_l\}$.
Suppose that $R=\emptyset$. In this case $\Gamma_i\leq 2\times S_{n/2}$ and therefore we could see $\Gamma_I$ as an imprimitive group with blocks of size two which we dealt with before. Hence $r\leq \frac{n-1}{2}$.

We may now assume that $R$ is nonempty.

First observe that if $\rho_j\in R$ then $\rho_j$ connects at least two pairs of $\langle \rho_l\rangle$-orbits. 
Indeed as  $\rho_j$ is an even permutation and $\rho_j$  does not commute with $\rho_l$, there must be an even number (different from zero) of pairs of $\langle\rho_l\rangle$-orbits joined by a single $\rho_j$-edge. 

Suppose $\rho_{l-1}\in R$.
Consider the graph $\mathcal{L}$ whose vertices are the $\langle \rho_l\rangle$-orbits and with a $j$-edge for each element of $C\cup R$ connecting
$\langle \rho_l\rangle$-orbits in different $\Gamma_{i,j}$-orbits. 
Note that $\mathcal{L}$ is a fracture graph for the group action on the $\langle \rho_l\rangle$-orbits of $\Gamma_i$.
Since $\Gamma_{i,j}$ is intransitive for every $j\neq i$, we have that $\mathcal{L}$ has no cycles. 
Then $\rho_{l-1}$ must connect at least two pairs of $\langle \rho_l\rangle$-orbits, as observed before.
Let us denote two of them by $(L_1,L'_1)$, and $(L_2,L'_2)$. We have that $|\{L_1,L'_1,L_2,L'_2\}|\geq 3$.
Suppose that $\{L_1,L'_1\}$ is the $(l-1)$-edge of $\mathcal{L}$.
Then either $L_2$ or $L'_2$ is another connected  component of $\mathcal{L}$, different from the component of $L_1$.
Let $L_1$ and $L_2$ be in different connected components of $\mathcal{L}$.
Therefore $|C\cup R|\leq n/2-2$. 
Suppose we have the equality. 
Then $\mathcal{L}$ has exactly two connected components, and 
either $|C|=n/2-3$ and $|R|=1$, or $|C|=n/2-4$ and $R=2$. 

Let $|C|=n/2-3$ and $R=\{\rho_{l-1}\}$. 
We claim that the incident edges of $\mathcal{L}$  have consecutive labels.
Suppose the contrary. 
Let  $g$ and $h$ be non-consecutive labels of incident edges of $\mathcal{L}$. Consider first  $g,h\neq l-1$.
Then there exists an alternating square, with labels $g$ and $h$, whose vertices are $\langle\rho_l\rangle$-orbits.
Three vertices of the square belong to one connected component of  $\mathcal{L}$  and  the fourth belongs to another component of  $\mathcal{L}$.
Let us denote this fourth vertex by $L$. 
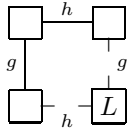
\[ \xymatrix@-.4pc{   *++[][F]{} \ar@{-}[d]_g\ar@{-}[r]^h  & *++[][F]{}\ar@{--}[d]^g \\
*++[][F]{} \ar@{--}[r]_h& *+[][F]{L}}\]
Hence it may be assumed that $L$ and $L_2$ are in the same connected component of $\mathcal{L}$.
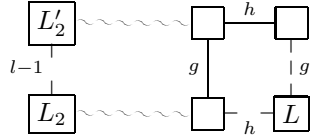
\[ \xymatrix@-.4pc{   *+[][F]{L'_2} \ar@{--}[d]_{l-1} \ar@{~}[rr]  && *++[][F]{} \ar@{-}[d]_g \ar@{-}[r]^h  & *++[][F]{}\ar@{--}[d]^g \\
*+[][F]{L_2} \ar@{~}[rr]&& *++[][F]{} \ar@{--}[r]_h& *+[][F]{L}}\]
There are paths in $\mathcal{L}$ from the two ends of the $g$-edge to $L_2$ and $L_2'$. 
But then these two vertices are in the same $\Gamma_g$-orbit, a contradiction. 
Now  suppose there is an $h$-edge with $h\neq l-2$ incident to an $(l-1)$-edge. These two edges are on an alternating square. 
Let $h$ be minimal with this property.
As $k>4$, there is a vertex incident to that square.
Now the label of the edge connecting that vertex to the square must be $h+1$.
Hence there is also an alternating square adjacent to the first one, with labels $h+1$ and $l-1$.
These two squares give three different connected components of $\mathcal{L}$, a contradiction.
Thus two incident edges in $\mathcal{L}$ must be consecutive.
This gives a unique possibility for the permutation representation graph of $\Gamma_i$.
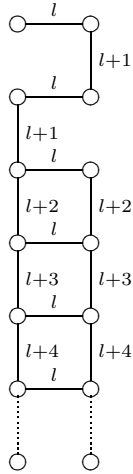
$$\xymatrix@-.2pc{ *+[o][F]{}  \ar@{-}[r]^l     & *+[o][F]{}  \ar@{-}[d]^{l+1}\\
 *+[o][F]{} \ar@{-}[d]^{l+1}    \ar@{-}[r]^l     & *+[o][F]{}  \\
 *+[o][F]{} \ar@{-}[d]^{l+2}    \ar@{-}[r]^l     & *+[o][F]{}  \ar@{-}[d]^{l+2}    \\ 
 *+[o][F]{} \ar@{-}[d]^{l+3}    \ar@{-}[r]^l     & *+[o][F]{}  \ar@{-}[d]^{l+3}    \\
  *+[o][F]{} \ar@{-}[d]^{l+4}    \ar@{-}[r]^l     & *+[o][F]{}  \ar@{-}[d]^{l+4}    \\
 *+[o][F]{}     \ar@{.}[d]  \ar@{-}[r]^l   & *+[o][F]{}  \ar@{.}[d]    \\  
  *+[o][F]{}   & *+[o][F]{}   } $$
Now there are two possibilities for $i$, namely either $i=0$ or $i=r-1$. 
But as $\rho_i$ must  break the block system, the only possibility is $i=r-1$ and $\rho_i$ a single transposition  connecting the two bottom vertices of the picture below, a contradiction.

Let  $|C|=n/2-4$ and $R=\{\rho_{l-1},\rho_{l+1}\}$. As observed at the beginning of the proof of this proposition, $\rho_{l+1}$ connects at least two pairs of $\langle\rho_l\rangle$-orbits. If it does not connect  $L_2$ to $L'_2$, then $\mathcal{L}$ has three connected components, a contradiction. 
Hence assume both $\rho_{l-1}$ and $\rho_{l+1}$ connect $L_2$ to $L'_2$.  
As $k\geq 3$ there exists another $\langle \rho_l\rangle$-orbit $L$ that is adjacent to either $L_2$ or $L'_2$. Let $\rho_h\in R\cup C$ be a permutation connecting $L$ and $L_2$. If $\rho_h\in C$ then, as it commutes with at least one of $\rho_{l-1}$ and $\rho_{l+1}$,
there is an alternating square in the permutation representation graph of $\Gamma$ with labels $h$ and one of $l-1$ and $l+1$.
This implies that $\rho_h$ also connects two pairs of $\langle \rho_l\rangle$-orbits different from $(L_2,L_2')$. 
Hence $|C\cup R|\leq n/2-3$, a contradiction. 
Thus $\rho_h\in R$. 
We may assume that $h=l-1$.
If a $(l+1)$-edge is incident to a $(l-1)$-edge then there exists an alternating square and we get a contradiction as before.
Hence there exists a double edge with labels $l+1$ and $l-1$ and no other incidence between edges with these labels.
Also if incident edges have labels in $C$, then the labels must be consecutive.
Moreover $\Gamma_i$ has the following permutation representation graph.
$$
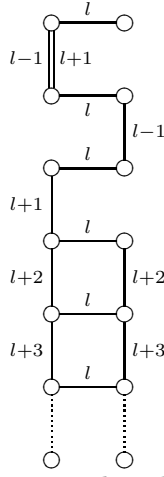
\xymatrix@-.2pc{ *+[o][F]{} \ar@{=}[d]_{l-1}^{l+1} \ar@{-}[r]^l     & *+[o][F]{}  \\
 *+[o][F]{}    \ar@{-}[r]_l     & *+[o][F]{} \ar@{-}[d]^{l-1} \\
 *+[o][F]{} \ar@{-}[d]_{l+1}    \ar@{-}[r]^l     & *+[o][F]{}    \\
 *+[o][F]{} \ar@{-}[d]_{l+2}    \ar@{-}[r]^l     & *+[o][F]{}  \ar@{-}[d]^{l+2}    \\ 
 *+[o][F]{} \ar@{-}[d]_{l+3}    \ar@{-}[r]^l     & *+[o][F]{}  \ar@{-}[d]^{l+3}    \\
 *+[o][F]{}     \ar@{.}[d]  \ar@{-}[r]^l   & *+[o][F]{}  \ar@{.}[d]    \\  
  *+[o][F]{}      & *+[o][F]{}   } $$
Now there are two possibilities for $i$, namely either $i=0$ or $i=r-1$. In each case either $\rho_i$ is an odd permutation or $\rho_i$ fixes the blocks, a contradiction.
\end{proof}

\section{All $\Gamma_i$'s intransitive: $2$-fracture graphs exist}
\label{s:2frac}

In this section and the next, we handle the case where all subgroups
$\Gamma_i$ are intransitive. As explained in the introduction, we use the
techniques of fracture graphs. In this section we deal with the case where
$2$-fracture graphs exist. In fact, for later use, our results are more
general: we do not assume that $\Gamma\cong A_n$ until the end of the section. 

Let $\mathcal{G}$ be a permutation representation graph on $n$ vertices for $\Gamma$, which is connected (meaning that this permutation representation of $\Gamma$ is transitive). From now on, we assume that $\Gamma$ is indeed the permutation group defined by $\mathcal{G}$. 

Suppose that $\Gamma$ has no transitive maximal parabolic subgroup $\Gamma_i$. 
We make the further assumption that each $\rho_i$ interchanges at least two pairs of points which lie in different $\Gamma_i$-orbits for all $i$.
Then $\mathcal{G}$ has a subgraph with $n$ vertices and $2r$ edges corresponding to two pairs of vertices in different $\Gamma_i$-orbits. We call this graph a  \emph{2-fracture graph}.

For ease of notation, we denote by $q_{i,j}$ an alternating square with labels $i$ and $j$, and call a cycle with more than four vertices a \emph{big cycle}.   

We give some basic results that follow immediately from the definition of a 2-fracture graph.

\begin{prop}\label{paths}
If $e=\{v,w\}$ is an $i$-edge of a 2-fracture graph $Q$ of $\Gamma$, then any path from $v$ to $w$ in $\mathcal{G}$ which does not contain $e$ must contain another $i$-edge.
\end{prop}

\begin{prop}\label{isoncy}
A cycle in a 2-fracture graph has either zero or two $i$-edges. In particular, a 2-fracture graph has no multiple edges.
\end{prop}



\begin{prop}\label{common}
If one $i$-edge is in the intersection of a pair of cycles of a 2-fracture graph, then both $i$-edges are in the intersection of those cycles. In particular, an edge of a square of a 2-fracture graph cannot be common to any other cycle.
\end{prop}
\begin{proof}
If only one $i$-edge is common to two cycles of a 2-fracture graph, then there is a cycle with only one $i$-edge, which is not possible by Proposition~\ref{isoncy}.
\end{proof}

\begin{prop}\label{cons}
A big cycle of a 2-fracture graph has adjacent edges with nonconsecutive labels.
\end{prop}
\begin{proof}
Let $i$ be the smallest edge label in a big cycle. By Proposition~\ref{isoncy} there are two $i$-edges in that cycle. These $i$-edges must be adjacent to at least three further edges. The only edges in the cycle with labels consecutive with $i$ have label $(i+1)$ (by minimality of $i$), and there are at most two of these; so the cycle has adjacent edges with nonconsecutive labels.
\end{proof}

In the following two propositions we show two useful  ways of getting a 2-fracture graph from another. In the first one an $i$-edge of a 2-fracture graph that is in a cycle is replaced by another $i$-edge in the same cycle. In the second proposition, one $i$-edge of a cycle is in a 2-fracture graph and 
the other $i$-edge in that cycle is replaced by the $i$-edge not
in the cycle in the 2-fracture graph.
\begin{prop}\label{swap}
If there is a cycle $C$ in $\mathcal{G}$ containing exactly two $i$-edges $e_1, e_2$, such that $e_1$ is in a 2-fracture graph $\mathcal{Q}$ and $e_2$ is not, then there is another 2-fracture graph $\mathcal{Q}'$ obtained by removing $e_1$ and adding $e_2$.
\end{prop}

\begin{proof}
The graph $\mathcal{Q}'$ is a 2-fracture graph, because the edge $e_2$ is between vertices in different $\Gamma_i$-orbits.
\end{proof}

\begin{prop}\label{putincycle}
If there is a cycle $C$ in $\mathcal{G}$ containing exactly two $i$-edges $e_1, e_2$, such that $e_1$ is in a 2-fracture graph $\mathcal{Q}$ and $e_2$ is not, then there is another 2-fracture graph $\mathcal{Q}'$ obtained by removing the $i$-edge of $\mathcal{Q}$ which is not $e_1$ and adding $e_2$.
\end{prop}

\begin{proof}
The proof is the same as that of the previous Proposition.
\end{proof}

Both of the above Propositions will be of particular use when the cycle $C$ is an alternating square.

\begin{prop}\label{movingsquares}
Let $q_{i,j}$ be an alternating square with a vertex $v$ $l$-adjacent to a vertex $w$ in a 2-fracture graph  $\mathcal{Q}$.  If $l$ is not consecutive with $i$, then the square can be moved to include the edge $\{v,w\}$.  That is, there is another 2-fracture graph $\mathcal{Q}'$ obtained from  $\mathcal{Q}$ by changing exactly two edges, as pictured below.   Furthermore $\mathcal{Q}'$ does not have more alternating squares than $\mathcal{Q}$.
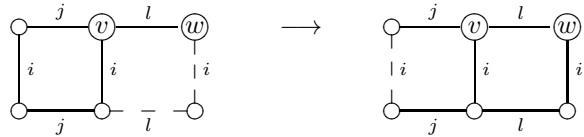
$$ \xymatrix@-1.3pc{*+[o][F]{}  \ar@{-}[rr]^j \ar@{-}[dd]^i && *+[o][F]{v}  \ar@{-}[rr]^l \ar@{-}[dd]^i&&*+[o][F]{w}\ar@{--}[dd]^i\\
 &&&&\\
 *+[o][F]{}\ar@{-}[rr]_j &&*+[o][F]{}\ar@{--}[rr]_l&& *+[o][F]{}}\qquad\longrightarrow\qquad\xymatrix@-1.3pc{*+[o][F]{}  \ar@{-}[rr]^j \ar@{--}[dd]^i && *+[o][F]{v}  \ar@{-}[rr]^l \ar@{-}[dd]^i&&*+[o][F]{w}\ar@{-}[dd]^i\\
 &&&&\\
 *+[o][F]{}\ar@{-}[rr]_j &&*+[o][F]{}\ar@{-}[rr]_l&& *+[o][F]{}}$$
\end{prop}

\begin{proof}
Let $l$ be not consecutive with $i$. There is an alternating square in $\mathcal{G}$, sharing an $i$-edge with $q_{i,j}$.  We apply Proposition~\ref{swap} to the other $i$-edge of $q_{i,j}$, and Proposition~\ref{putincycle} for the $l$-edges.

\end{proof}

\begin{prop}\label{changelabels}
If a 2-fracture graph  $\mathcal{Q}$ contains the subgraph on the left of the following figure, then there is a 2-fracture graph $\mathcal{Q}'$ containing the subgraph on the right, such that  $\mathcal{Q}$ and  $\mathcal{Q}'$ differ only in four edges.
$$ 
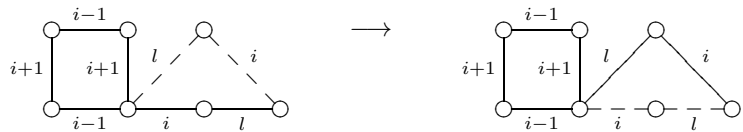
\xymatrix@-1.3pc{*+[o][F]{}  \ar@{-}[rr]^{i-1} \ar@{-}[dd]_{i+1}&& *+[o][F]{}   \ar@{-}[dd]_{i+1}&&*+[o][F]{} &&\\
 &&&&\\
 *+[o][F]{}\ar@{-}[rr]_{i-1}&&*+[o][F]{}\ar@{-}[rr]_i\ar@{--}[uurr]^l && *+[o][F]{}\ar@{-}[rr]_l&& *+[o][F]{}\ar@{--}[uull]_i}\qquad\longrightarrow\qquad\xymatrix@-1.3pc{*+[o][F]{}  \ar@{-}[rr]^{i-1} \ar@{-}[dd]_{i+1}&& *+[o][F]{}   \ar@{-}[dd]_{i+1}&&*+[o][F]{} &&\\
 &&&&\\
 *+[o][F]{}\ar@{-}[rr]_{i-1}&&*+[o][F]{}\ar@{--}[rr]_i\ar@{-}[uurr]^l && *+[o][F]{}\ar@{--}[rr]_l&& *+[o][F]{}\ar@{-}[uull]_i} $$
\end{prop}

\begin{proof}
This is a consequence of applying Proposition~\ref{swap} twice.
\end{proof}


\begin{prop}\label{nobigcicles}
If $\Gamma$ has a 2-fracture graph, then it has one with no big cycle.
\end{prop}
\begin{proof}
Consider a 2-fracture graph $\mathcal{Q}$ with $q$ big cycles. Suppose that $C$ is a big cycle of $\mathcal{Q}$. By Proposition~\ref{cons} there is at least one pair of adjacent edges with nonconsecutive labels $i$ and $j$. Hence these edges belong to an alternating square $q_{i,j}$ of $\mathcal{G}$. The other two edges of this square are not edges of $\mathcal{Q}$. Indeed the remaining $i$ and $j$ edges of $\mathcal{Q}$ must belong to $C$ by Proposition~\ref{isoncy}.
Consider the 2-fracture graph $\mathcal{Q}'$ that is obtained from $\mathcal{Q}$ by applying Proposition~\ref{swap} twice, replacing the $i$-edge and the $j$-edge of $C$ which are not in $q_{i,j}$ by those that are not in $C$ but are in $q_{i,j}$.
By Proposition~\ref{common} the edges of the square cannot belong to another cycle, thus $\mathcal{Q}'$ has $q-1$ big cycles.
Continuing this process we obtain a 2-fracture graph without big cycles.
\end{proof}

\begin{prop}\label{max1squareforcomp}
If $\mathcal{G}$ has a 2-fracture graph, then it has a 2-fracture graph such that each connected component has at most one cycle, and any cycle is an alternating square.
\end{prop}
\begin{proof}

By Proposition~\ref{nobigcicles}, $\mathcal{G}$ contains a 2-fracture graph $\mathcal{Q}_0$ having no big cycle.
If $\mathcal{Q}_0$ has a cycle, then it is an alternating square by Proposition~\ref{isoncy}.
Suppose that a connected component of $\mathcal{Q}_0$ has $p$ alternating squares with $p>1$.
We will prove that there exists a 2-fracture graph with $p-1$ squares and without any big cycle.  

The proof is a double induction. Suppose that $q_{i,j}$ and $q_{l,k}$ are squares of $\mathcal{Q}_0$ that have distance $s$ which is the smallest distance between two squares in $\mathcal{Q}_0$. We produce another 2-fracture graph with
either one fewer square, or the minimal distance $s$ reduced by $1$.

Either the squares $q_{i,j}$ and $q_{l,k}$ share a vertex, or there is a path in $\mathcal{Q}_0$ connecting them.

Suppose that they share a vertex. Then there exists another alternating square in $\mathcal{G}$. We may assume that it is $q_{i,k}$, which shares an edge with both $q_{i,j}$ and $q_{l,k}$. 
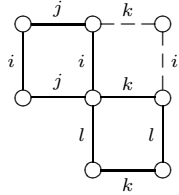
$$\xymatrix@-1.4pc{*+[o][F]{}\ar@{-}[rr]^j &&*+[o][F]{} \ar@{--}[rr]^k  &&*+[o][F]{}\ar@{--}[dd]^i\\
&&&&\\
*+[o][F]{}\ar@{-}[uu]^i  \ar@{-}[rr]^j&& *+[o][F]{} \ar@{-}[uu]^i  \ar@{-}[rr]^k   \ar@{-}[dd]_l&& *+[o][F]{}  \ar@{-}[dd]_l\\
&&&&\\
&& *+[o][F]{} \ar@{-}[rr]_k   &&*+[o][F]{}  }
$$
By Proposition~\ref{swap}, we can replace one of the $i$-edges of $\mathcal{Q}_0$ with the $i$-edge of $q_{i,k}$ not in $\mathcal{Q}_0$ to obtain a 2-fracture graph $\mathcal{Q}_1$.  Note that this new $i$-edge of $\mathcal{Q}_1$ does not belong to any cycle by Proposition~\ref{common}.  Thus $\mathcal{Q}_1$ has $p-1$ squares.

Now assume that the shortest path between $q_{i,j}$ and $q_{k,l}$ has $s \geq 1$ edges.  If the first or the last edge of this path has a label not consecutive with the square that it meets, then we can use Proposition~\ref{movingsquares}, to create a new 2-fracture graph which does not have more squares, and with a smaller minimal distance between squares.
$$ 
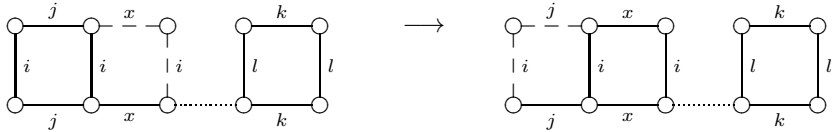
\xymatrix@-1.3pc{*+[o][F]{}  \ar@{-}[rr]^j \ar@{-}[dd]^i && *+[o][F]{}  \ar@{--}[rr]^x \ar@{-}[dd]^i&&*+[o][F]{}\ar@{--}[dd]^i&& *+[o][F]{}  \ar@{-}[rr]^k \ar@{-}[dd]^l&& *+[o][F]{} \ar@{-}[dd]^l\\
 &&&&&&&&\\
 *+[o][F]{}\ar@{-}[rr]_j &&*+[o][F]{}\ar@{-}[rr]_x&& *+[o][F]{}\ar@{.}[rr]&&*+[o][F]{}\ar@{-}[rr]_k&&*+[o][F]{}}\qquad\longrightarrow\qquad\xymatrix@-1.3pc{*+[o][F]{}  \ar@{--}[rr]^j \ar@{--}[dd]^i && *+[o][F]{}  \ar@{-}[rr]^x \ar@{-}[dd]^i&&*+[o][F]{}\ar@{-}[dd]^i&& *+[o][F]{}  \ar@{-}[rr]^k \ar@{-}[dd]^l&& *+[o][F]{} \ar@{-}[dd]^l\\
 &&&&&&&&\\
 *+[o][F]{}\ar@{-}[rr]_j &&*+[o][F]{}\ar@{-}[rr]_x&& *+[o][F]{}\ar@{.}[rr]&&*+[o][F]{}\ar@{-}[rr]_k&&*+[o][F]{}}$$

 Otherwise, we can use Proposition~\ref{changelabels} in order to create a new 2-fracture graph with the first or the last edge of a new path having a label not consecutive with the square that it meets, and then use Proposition~\ref{movingsquares} as before.

Continuing this process, we get a 2-fracture graph where the minimal distance between squares is zero, and then applying the argument above, we get a 2-fracture graph with fewer squares, and without big cycles.   This construction terminates with a 2-fracture graph with no big cycles and with at most one square in each connected component.

\end{proof}


\begin{prop}\label{3cases}
Suppose that $\Gamma$ has a disconnected 2-fracture graph $\mathcal{Q}$ with no big cycles such that every connected component has exactly one square. 
Then $\Gamma$ has a 2-fracture graph $\mathcal{Q}'$ with the same characteristics  as $\mathcal{Q}$ and such that the minimal distance between two squares in $\mathcal{G}$ is either one, two or three accordantly to one of the following three cases.
$$(1)
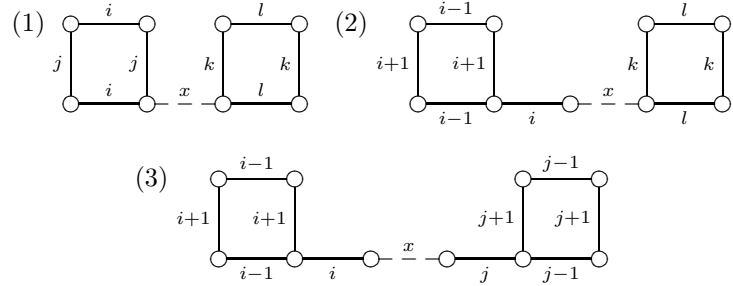
\xymatrix@-1.3pc{*+[o][F]{}  \ar@{-}[rr]^i \ar@{-}[dd]_j&& *+[o][F]{}   \ar@{-}[dd]_j&& *+[o][F]{}  \ar@{-}[rr]^l \ar@{-}[dd]_k&& *+[o][F]{}\ar@{-}[dd]_k  \\
&&&&&&\\
*+[o][F]{}  \ar@{-}[rr]^i && *+[o][F]{}  \ar@{--}[rr]^x&& *+[o][F]{}  \ar@{-}[rr]^l&& *+[o][F]{} }
 \quad (2)\xymatrix@-1.3pc{*+[o][F]{}  \ar@{-}[rr]^{i-1} \ar@{-}[dd]_{i+1}&& *+[o][F]{}   \ar@{-}[dd]_{i+1}&&&& *+[o][F]{}  \ar@{-}[rr]^l \ar@{-}[dd]_k&& *+[o][F]{}\ar@{-}[dd]_k  \\
&&&&&&&&\\
*+[o][F]{}  \ar@{-}[rr]_{i-1} && *+[o][F]{}  \ar@{-}[rr]_i&& *+[o][F]{}  \ar@{--}[rr]^x&& *+[o][F]{} \ar@{-}[rr]_l && *+[o][F]{} }$$
$$(3) \xymatrix@-1.3pc{*+[o][F]{}  \ar@{-}[rr]^{i-1} \ar@{-}[dd]_{i+1}&& *+[o][F]{}   \ar@{-}[dd]_{i+1}&&&&&& *+[o][F]{}  \ar@{-}[rr]^{j-1} \ar@{-}[dd]_{j+1}&& *+[o][F]{}\ar@{-}[dd]_{j+1}  \\
&&&&&&&&&&\\
*+[o][F]{}  \ar@{-}[rr]_{i-1} && *+[o][F]{}  \ar@{-}[rr]_i&& *+[o][F]{}  \ar@{--}[rr]^x&& *+[o][F]{} \ar@{-}[rr]_j && *+[o][F]{}\ar@{-}[rr]_{j-1} && *+[o][F]{} }$$
\end{prop}

\begin{proof}
As $\mathcal{Q}$ contains all vertices of $\mathcal{G}$, any component of $\mathcal{Q}$ is at distance one from another component of $\mathcal{Q}$. Consider  two components in $\mathcal{Q}$ at distance one. Repeatedly using Propositions~\ref{movingsquares} and \ref{changelabels}, we can obtain a 2-fracture graph that contains one of the above subgraphs.
\end{proof}


\subsection{Disconnected 2-fracture graphs}

We now separate the argument into two cases, dealing first with the case of
a disconnected 2-fracture graph.

\begin{prop}\label{dislc}
If $\mathcal{G}$  has no connected 2-fracture graph, then it has a 2-fracture graph that has at least one component which is a tree, all the others having only one cycle (which is an alternating square).
\end{prop}
\begin{proof}
Suppose that $\mathcal{Q}$ is a disconnected 2-fracture graph of $\Gamma$.
By Proposition~\ref{max1squareforcomp} there exists a 2-fracture graph $\mathcal{Q}$ of $\Gamma$ having at most one cycle, that is an alternating square, in each connected component. We proceed by induction on the number of connected components of $\mathcal{Q}$.  

Suppose that $\mathcal{Q}$ has $p$ connected components, each with an alternating square.  By Proposition~\ref{3cases}, we can choose $\mathcal{Q}$ so that it has a subgraph as shown in one of the three cases of Proposition~\ref{3cases}.
In what follows we deal with each of these cases separately.

Let us consider that  $\mathcal{G}$ has the subgraph (1)  of Proposition~\ref{3cases}. Observe here that $x$ is not consecutive with at least one of the labels of the squares.  In addition since the squares of $\mathcal{Q}$  are in different connected components of $\mathcal{Q}$, the $x$-edge in not in $\mathcal{Q}$. Suppose $x$ and $l$ are nonconsecutive. Then we have an alternating square as in the following picture on the left. Using Proposition~\ref{swap} we obtain $\mathcal{Q}'$ on the right.

$$
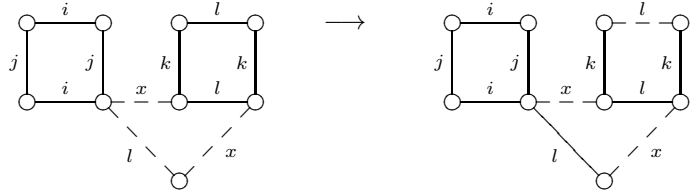
\xymatrix@-1.3pc{*+[o][F]{}  \ar@{-}[rr]^i \ar@{-}[dd]_j&& *+[o][F]{}   \ar@{-}[dd]_j&& *+[o][F]{}  \ar@{-}[rr]^l \ar@{-}[dd]_k&& *+[o][F]{}\ar@{-}[dd]_k  \\
&&&&&&\\
*+[o][F]{}  \ar@{-}[rr]^i && *+[o][F]{}  \ar@{--}[rr]^x&& *+[o][F]{}  \ar@{-}[rr]^l&& *+[o][F]{}\\
&&&&&&\\
&&&&*+[o][F]{}\ar@{--}[uurr]_x\ar@{--}[uull]^l && }\qquad\longrightarrow\qquad \xymatrix@-1.3pc{*+[o][F]{}  \ar@{-}[rr]^i \ar@{-}[dd]_j&& *+[o][F]{}   \ar@{-}[dd]_j&& *+[o][F]{}  \ar@{--}[rr]^l \ar@{-}[dd]_k&& *+[o][F]{}\ar@{-}[dd]_k  \\
&&&&&&\\
*+[o][F]{}  \ar@{-}[rr]^i && *+[o][F]{}  \ar@{--}[rr]^x&& *+[o][F]{}  \ar@{-}[rr]^l&& *+[o][F]{}\\
&&&&&&\\
&&&&*+[o][F]{}\ar@{--}[uurr]_x\ar@{-}[uull]^l && }$$
 If the other $x$-edge is not in $\mathcal{Q}$, then we created a connected component which is a tree.  On the other hand, if the other $x$-edge is in $\mathcal{Q}$, then the new 2-fracture graph has $p-1$ connected components.
 
Let us consider the second case of Proposition~\ref{3cases}. 
First suppose that $x$ is not consecutive either with $l$ or $k$.
Without loss of generality we assume that it is not consecutive with $l$; then we use Proposition~\ref{swap} as shown in the following diagram. 
$$
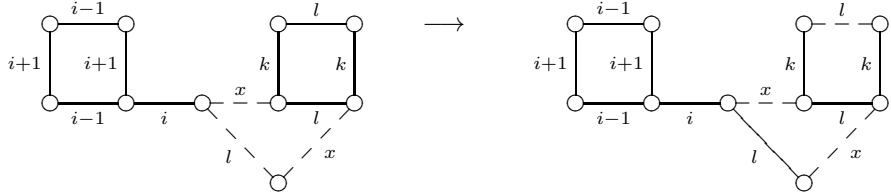
\xymatrix@-1.3pc{*+[o][F]{}  \ar@{-}[rr]^{i-1} \ar@{-}[dd]_{i+1}&& *+[o][F]{}   \ar@{-}[dd]_{i+1}&&&& *+[o][F]{}  \ar@{-}[rr]^l \ar@{-}[dd]_k&& *+[o][F]{}\ar@{-}[dd]_k  \\
&&&&&&&&\\
*+[o][F]{}  \ar@{-}[rr]_{i-1} && *+[o][F]{}  \ar@{-}[rr]_i&& *+[o][F]{}  \ar@{--}[rr]^x&& *+[o][F]{} \ar@{-}[rr]_l && *+[o][F]{} \\
&&&&&&&&\\
&&&&&&*+[o][F]{}\ar@{--}[uurr]_x\ar@{--}[uull]^l &&}\qquad\longrightarrow\qquad \xymatrix@-1.3pc{*+[o][F]{}  \ar@{-}[rr]^{i-1} \ar@{-}[dd]_{i+1}&& *+[o][F]{}   \ar@{-}[dd]_{i+1}&&&& *+[o][F]{}  \ar@{--}[rr]^l \ar@{-}[dd]_k&& *+[o][F]{}\ar@{-}[dd]_k  \\
&&&&&&&&\\
*+[o][F]{}  \ar@{-}[rr]_{i-1} && *+[o][F]{}  \ar@{-}[rr]_i&& *+[o][F]{}  \ar@{--}[rr]^x&& *+[o][F]{} \ar@{-}[rr]_l && *+[o][F]{} \\
&&&&&&&&\\
&&&&&&*+[o][F]{}\ar@{--}[uurr]_x\ar@{-}[uull]^l &&}$$
The argument for why this exchange does not create any new cycles is the same as the previous case.
Similar to above, if neither of the $x$-edges is in $\mathcal{Q}$, then we created a connected component which is a tree.  On the other hand, if one of the $x$-edges is in $\mathcal{Q}$, then the new 2-fracture graph has $p-1$ connected components.
Next, consider that $x$ is consecutive with both $l$ and $k$, then we have the following diagram.
$$
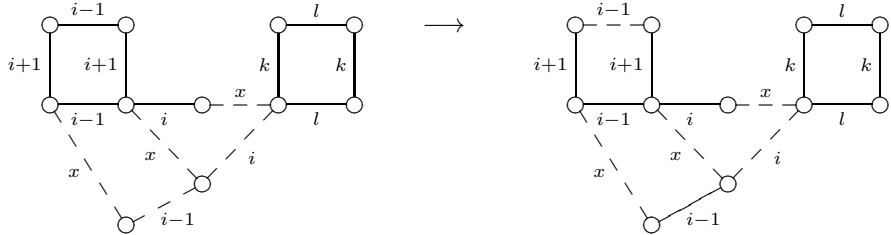
\xymatrix@-1.3pc{*+[o][F]{}  \ar@{-}[rr]^{i-1} \ar@{-}[dd]_{i+1}&& *+[o][F]{}   \ar@{-}[dd]_{i+1}&&&& *+[o][F]{}  \ar@{-}[rr]^l \ar@{-}[dd]_k&& *+[o][F]{}\ar@{-}[dd]_k  \\
&&&&&&&&\\
*+[o][F]{}  \ar@{-}[rr]_{i-1} && *+[o][F]{}  \ar@{-}[rr]_i&& *+[o][F]{}  \ar@{--}[rr]^x&& *+[o][F]{} \ar@{-}[rr]_l && *+[o][F]{} \\
&&&&&&&&\\
&&&&*+[o][F]{}\ar@{--}[uurr]_i\ar@{--}[uull]^x&& &&\\
&&*+[o][F]{}\ar@{--}[urr]_{i-1}\ar@{--}[uuull]^x&&&&&&}\qquad\longrightarrow\qquad\xymatrix@-1.3pc{*+[o][F]{}  \ar@{--}[rr]^{i-1} \ar@{-}[dd]_{i+1}&& *+[o][F]{}   \ar@{-}[dd]_{i+1}&&&& *+[o][F]{}  \ar@{-}[rr]^l \ar@{-}[dd]_k&& *+[o][F]{}\ar@{-}[dd]_k  \\
&&&&&&&&\\
*+[o][F]{}  \ar@{-}[rr]_{i-1} && *+[o][F]{}  \ar@{-}[rr]_i&& *+[o][F]{}  \ar@{--}[rr]^x&& *+[o][F]{} \ar@{-}[rr]_l && *+[o][F]{} \\
&&&&&&&&\\
&&&&*+[o][F]{}\ar@{--}[uurr]_i\ar@{--}[uull]^x&& &&\\
&&*+[o][F]{}\ar@{-}[urr]_{i-1}\ar@{--}[uuull]^x&&&&&&}$$

Note that $x$ is not consecutive with $i$, and not consecutive with either $i-1$ or $i+1$.  Without loss we assume it is not consecutive with $i-1$.  Suppose that both $x$-edges of the square $q_{x,i-1}$
are in  $\mathcal{Q}$. Then using Proposition~\ref{swap}, we create a square $q_{x,i-1}$ in $\mathcal{Q}'$. 
Now if both $i$-edges are in  $\mathcal{Q}$ we have reduced the number of connected components; otherwise, we go back to case (1).
Similarly, if an $x$-edge of $q_{x,i-1}$ is not in $\mathcal{Q}$, then we either created a 2-fracture graph with $p-1$ components or with a component that is a tree.

Finally, we consider the diagram below for case (3). Here, $x$ is not consecutive with either $i$ or $j$.  We may assume it is not consecutive with $i$.  Furthermore,  $x$ is not consecutive either with $i-1$ or $i+1$.  Let us assume it is not consecutive with  $i-1$.   A similar argument shows that the new graph either has $p-1$ components or a component that is a tree.

$$
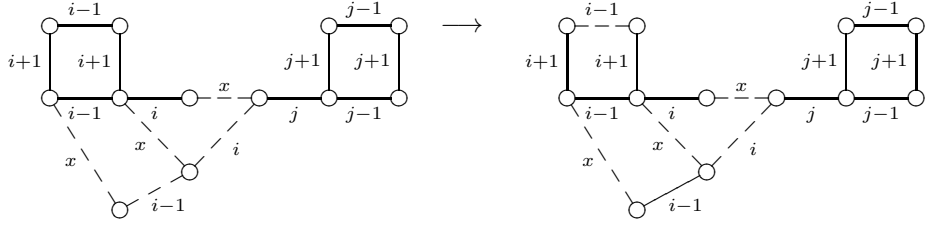
\xymatrix@-1.4pc{*+[o][F]{}  \ar@{-}[rr]^{i-1} \ar@{-}[dd]_{i+1}&& *+[o][F]{}   \ar@{-}[dd]_{i+1}&&&&&& *+[o][F]{}  \ar@{-}[rr]^{j-1} \ar@{-}[dd]_{j+1}&& *+[o][F]{}\ar@{-}[dd]_{j+1}  \\
&&&&&&&&&&\\
*+[o][F]{}  \ar@{-}[rr]_{i-1} && *+[o][F]{}  \ar@{-}[rr]_i&& *+[o][F]{}  \ar@{--}[rr]^x&& *+[o][F]{} \ar@{-}[rr]_j && *+[o][F]{}\ar@{-}[rr]_{j-1} && *+[o][F]{} \\
&&&&&&&&&&\\
&&&& *+[o][F]{}\ar@{--}[uurr]_i\ar@{--}[uull]^x &&&&&&\\
&&*+[o][F]{}\ar@{--}[urr]_{i-1}\ar@{--}[uuull]^x&&&&&&&&}\quad\longrightarrow\quad\xymatrix@-1.4pc{*+[o][F]{}  \ar@{--}[rr]^{i-1} \ar@{-}[dd]_{i+1}&& *+[o][F]{}   \ar@{-}[dd]_{i+1}&&&&&& *+[o][F]{}  \ar@{-}[rr]^{j-1} \ar@{-}[dd]_{j+1}&& *+[o][F]{}\ar@{-}[dd]_{j+1}  \\
&&&&&&&&&&\\
*+[o][F]{}  \ar@{-}[rr]_{i-1} && *+[o][F]{}  \ar@{-}[rr]_i&& *+[o][F]{}  \ar@{--}[rr]^x&& *+[o][F]{} \ar@{-}[rr]_j && *+[o][F]{}\ar@{-}[rr]_{j-1} && *+[o][F]{} \\
&&&&&&&&&&\\
&&&& *+[o][F]{}\ar@{--}[uurr]_i\ar@{--}[uull]^x &&&&&&\\
&&*+[o][F]{}\ar@{-}[urr]_{i-1}\ar@{--}[uuull]^x&&&&&&&&}$$

After any of the exchanges seen above, we have a 2-fracture graph with no big cycle, and either $p-1$ connected components or a component which is a tree.  Therefore, by induction, there exists a 2-fracture graph that is either connected or disconnected with at least one component being a tree.

\end{proof}



\subsection{Connected 2-fracture graphs}

In this subsection we will assume that  $\Gamma$ has a connected 2-fracture graph.

\begin{prop}\label{squareinconeclcomp0}
Let $q_{i,j}$ be an alternating square of a connected 2-fracture graph $\mathcal{Q}$ of $\Gamma$.  Let $q_{i,l}$ be an alternating square in $\mathcal{G}$ sharing an $i$-edge with $q_{i,j}$.  Then neither or both the $l$-edges of $q_{i,l}$ are in $\mathcal{Q}$.
\end{prop}

\begin{proof}
Suppose that exactly one $l$-edge of $q_{i,l}$ is in $\mathcal{Q}$.  Let $u$ be the vertex of $q_{i,l}$ which is adjacent to $q_{i,j}$ in $\mathcal{G}$ by the $l$ edge of $q_{i,l}$ which is not in $\mathcal{Q}$.  
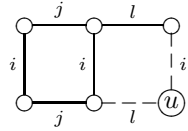
$$\xymatrix@-1.4pc{*+[o][F]{}  \ar@{-}[rr]^j \ar@{-}[dd]_i&& *+[o][F]{} \ar@{-}[rr]^l   \ar@{-}[dd]_i&& *+[o][F]{}  \ar@{--}[dd]^i\\
&&&&\\
*+[o][F]{}  \ar@{-}[rr]_j && *+[o][F]{} \ar@{--}[rr]_l   && *+[o][F]{u} }$$
Any possibility to connect $u$ to $q_{i,j}$ gives a contradiction to Proposition~\ref{paths}.
\end{proof}

\begin{prop}\label{squareinconeclcomp2}
Let $q_{i,j}$ be an alternating square of a connected 2-fracture graph $\mathcal{Q}$.
All edges of $\mathcal{G}$ meeting the vertices of $q_{i,j}$ belong to $\mathcal{Q}$ and have labels consecutive either with $i$ or $j$.
\end{prop}
\begin{proof}
Let $w$ be a vertex of $q_{i,j}$ $k$-adjacent (in $\mathcal{G}$) to a vertex $v$.
Suppose that $\{v,w\}$ is not in $\mathcal{Q}$.
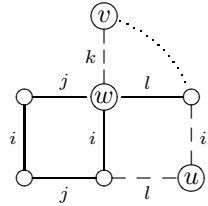
$$\xymatrix@-1.4pc{&&*+[o][F]{v}\ar@/^/@{.}[ddrr] &&\\
&&&&\\
*+[o][F]{}  \ar@{-}[rr]^j \ar@{-}[dd]_i&& *+[o][F]{w} \ar@{--}[uu]^k  \ar@{-}[rr]^l   \ar@{-}[dd]_i&& *+[o][F]{}  \ar@{--}[dd]^i\\
&&&&\\
*+[o][F]{}  \ar@{-}[rr]_j && *+[o][F]{} \ar@{--}[rr]_l   && *+[o][F]{u} }$$
There exists a path in $\mathcal{Q}$ from $v$ to $w$.
Let $l$ be the label of the edge of this path meeting $q_{i,j}$. 
Suppose that there is a square $q_{i,l}$ or $q_{j,l}$ sharing an edge with $q_{i,j}$. 
Suppose first that $q_{i,l}$ contains $w$.
Let $u$ be the vertex of $q_{l,i}$ diagonally opposed to $w$. One of the $l$-edges of $q_{l,i}$ is not in $\mathcal{Q}$, because both are in the path from $w$ to $v$.   By Proposition \ref{squareinconeclcomp0}, we have a contradiction.
Similar arguments rule out the case where $q_{i,l}$ shares the other $i$-edge with $q_{i,j}$.
Thus $|i-j|=2$ and $l$ is consecutive with both $i$ and $j$. But then $k$ is not consecutive either with $i$ or with $j$. 
Suppose without loss of generality it is not consecutive with $j$, thus we have a square $q_{j,k}$ sharing an edge with $q_{i,j}$.
Let $\{v',w'\}$ be the other $k$-edge of $q_{i,k}$, with $v'$ not in $q_{i,j}$.
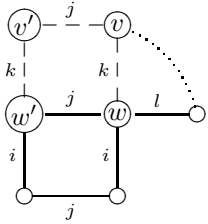
$$\xymatrix@-1.4pc{*+[o][F]{v'}\ar@{--}[dd]_k\ar@{--}[rr]^j&&*+[o][F]{v}\ar@/^/@{.}[ddrr] &&\\
&&&&\\
*+[o][F]{w'}  \ar@{-}[rr]^j \ar@{-}[dd]_i&& *+[o][F]{w} \ar@{--}[uu]^k  \ar@{-}[rr]^l   \ar@{-}[dd]_i&& *+[o][F]{}  \\
&&&&\\
*+[o][F]{}  \ar@{-}[rr]_j && *+[o][F]{} && }$$
By Proposition~\ref{squareinconeclcomp0}, $\{v',w'\}$ is not in $\mathcal{Q}$.
Thus there is a path in $\mathcal{Q}$, connecting $w'$ to $v'$, and this path does not have any $l$-edges.  Let $l'$ be the label of the edge of this path meeting $w'$.  As $l'$ is not consecutive with both $i$ and $j$ we have a contradiction as before. Thus $\{v,w\}$ is in $\mathcal{Q}$.

Furthermore, as all edges of $\mathcal{G}$ adjacent to $q_{i,j}$ are in $\mathcal{Q}$, $k$ must be consecutive to either $i$ or $j$.  Otherwise, we would have at least four edges with the same label in $\mathcal{Q}$.
\end{proof}

\begin{prop}\label{squareedge}
If $q_{i,j}$ is an alternating square of a connected 2-fracture graph $\mathcal{Q}$ with $n \geq 9$ vertices, then each vertex of  $q_{i,j}$ has degree at most three in $\mathcal{Q}$.
\end{prop}
\begin{proof}
Suppose that $q_{i,j}$ has a vertex adjacent, in $\mathcal{Q}$, to two other edges, and let $k$ and $l$ be the labels of those edges.  
By Proposition~\ref{squareinconeclcomp2} $k$ and $l$ must be consecutive with one of the labels of $q_{i,j}$.
We may assume that $i<j$ and $l<k$. Then there are three possibilities either $i<l<k<j$, $l<i<j<k$ or $i<l<j<k$ corresponding to the following graphs.
$$
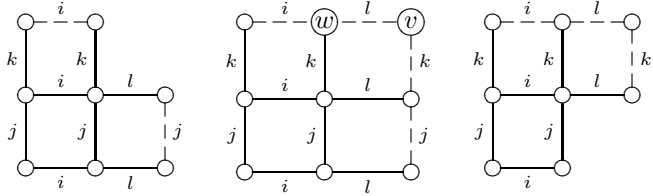
\xymatrix@-1.4pc{*+[o][F]{}\ar@{--}[rr]^i  &&*+[o][F]{} &&\\
&&&&\\
*+[o][F]{}\ar@{-}[uu]^k   \ar@{-}[rr]^i\ar@{-}[dd]_j&& *+[o][F]{} \ar@{-}[uu]^k  \ar@{-}[rr]^l   \ar@{-}[dd]_j&& *+[o][F]{}  \ar@{--}[dd]^j\\
&&&&\\
*+[o][F]{}  \ar@{-}[rr]_i && *+[o][F]{} \ar@{-}[rr]_l   && *+[o][F]{} }\quad 
\xymatrix@-1.4pc{*+[o][F]{}\ar@{--}[rr]^i  &&*+[o][F]{w} \ar@{--}[rr]^l &&*+[o][F]{v} \ar@{--}[dd]^k\\
&&&&\\
*+[o][F]{}\ar@{-}[uu]^k   \ar@{-}[rr]^i\ar@{-}[dd]_j&& *+[o][F]{} \ar@{-}[uu]^k  \ar@{-}[rr]^l   \ar@{-}[dd]_j&& *+[o][F]{}  \ar@{--}[dd]^j\\
&&&&\\
*+[o][F]{}  \ar@{-}[rr]_i && *+[o][F]{} \ar@{-}[rr]_l   && *+[o][F]{} }\quad
 \xymatrix@-1.4pc{*+[o][F]{}\ar@{--}[rr]^i  &&*+[o][F]{} \ar@{--}[rr]^l  &&*+[o][F]{}\ar@{--}[dd]^k\\
&&&&\\
*+[o][F]{}\ar@{-}[uu]^k   \ar@{-}[rr]^i\ar@{-}[dd]_j&& *+[o][F]{} \ar@{-}[uu]^k  \ar@{-}[rr]^l   \ar@{-}[dd]_j&& *+[o][F]{}  \\
&&&&\\
*+[o][F]{}  \ar@{-}[rr]_i && *+[o][F]{}   &&  }
$$
Consider the second case. As $\mathcal{Q}$ is connected, there is a path in $\mathcal{Q}$ between  $v$ and $w$. Then by Proposition~\ref{swap}, we can create a cycle in a 2-fracture graph having only one edge with some label, contradicting Proposition~\ref{isoncy}.

The third case becomes the same as the first case by using Proposition~\ref{swap} on the $i$-edges.

Now we consider the first case. If the labels are not all consecutive, then another pair of labels give an alternating square and we get the same contradiction as in the second case.
Suppose the labels are consecutive, then we use the fact that there is another vertex of $\mathcal{Q}$ connected to one of the vertices of one of the graphs above.  Again using Propositions~\ref{swap} and \ref{squareinconeclcomp2}, independent of how this other vertex is connected, we can obtain the same contradiction as in the second case.
\end{proof}
\begin{prop}\label{minimalsquare}
Let $n\geq 9$. If $\Gamma$ has a 2-fracture graph with a square, then it has a 2-fracture graph having a square  $q_{i,j}$ such that $|i-j|=2$.
\end{prop}
\begin{proof}
Choose $\mathcal{Q}$, $i$, and $j$ so that the alternating square $q_{i,j}$ has the property that $|i- j|$ is minimal.  

First suppose that $|i - j|  \geq 3$. By Proposition~\ref{squareinconeclcomp2} there is a vertex of the square $q_{i,j}$ which is $k_1$-adjacent to another vertex with $k_1\in \{ i-1, i+1, j-1, j+1\}$. Without loss of generality, we can assume $k_1 = i \pm 1$, and thus there is an alternating square $q_{k_1,j}$ sharing an edge with $q_{i,j}$.  By Proposition~\ref{swap}, there is another connected 2-fracture graph $\mathcal{Q}'$ with $q_{k_1,j}$ as its alternating square.  Since we assumed that  $| i- j |$ is minimal we know that $ | k_1- j| > | i - j |$. 
Consider the vertices $v_i, \, i\in\{0,\ldots,6\}$ as in the following figure.
$$
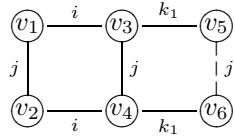
\xymatrix@-1.4pc{
*+[o][F]{v_1}\ar@{-}[rr]^i\ar@{-}[dd]_j &&*+[o][F]{v_3}   \ar@{-}[dd]^j\ar@{-}[rr]^{k_1}&& *+[o][F]{v_5}\ar@{--}[dd]^j \\
&&&&\\
*+[o][F]{v_2}\ar@{-}[rr]_i&&*+[o][F]{v_4} \ar@{-}[rr]_{k_1}&&  *+[o][F]{v_6}  }$$

By Proposition~\ref{squareedge} the degree of the vertices $v_3$ and $v_4$ is three. Suppose that  $v_1$ has degree three, then the label $k$ of the edge incident to $v_1$, not in $q_{i,j}$, would satisfy $ | k- j | < | i- j|$.   As above, using Proposition~\ref{swap}, we would obtain a contradiction.  Thus the degree of $v_1$ is two, and similarly the degree of $v_2$ is two.

As $n \geq 9$ the degree of either $v_5$ or $v_6$ is three.  Assume without loss of generality that $v_5$ is $k_2$-adjacent to a vertex $v_7$.  We have that  $k_2$ must be consecutive with $k_1$; furthermore, there is an alternating square $q_{k_2,j}$ with vertices $v_5,v_6,v_7, v_8$, and thus $ | k_2 - j| > | k_1 - j | > | i- j| $.  Continuing this process we get a sequence of labels $k_3,\ldots, k_s$ with $ | k_s - j| >\ldots> | k_1 - j | > | i- j| $. Thus $\mathcal{Q}$ does not have edges with labels between $i$ and $j$. We can then conclude that $\Gamma$ is not connected, a contradiction. Hence $|i - j|  \leq 2$.

Suppose that $|i-j|=1$, say $j=i+1$.
Then by Proposition~\ref{squareinconeclcomp2} there is a vertex either $i-1$ or $i+2$ adjacent to this square.  This situation guarantees another alternating square either $q_{i-1,i+1}$ or $q_{i+2,i}$.  By Proposition~\ref{swap}, there is another connected 2-fracture graph having a square $q_{i,j}$ with $j=i+2$, as wanted.
\end{proof}
\begin{prop}\label{i-j=2}
Let $n\geq 9$.  If $q_{i-1,i+1}$ is an alternating square of a connected 2-fracture graph $\mathcal{Q}$, then both $i$-edges are incident to the square as shown in the following figure. 
$$
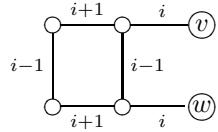
\xymatrix@-1.4pc{
*+[o][F]{}\ar@{-}[rr]^{i+1} \ar@{-}[dd]_{i-1} &&*+[o][F]{}   \ar@{-}[dd]^{i-1}\ar@{-}[rr]^i&& *+[o][F]{v} \\
&&&&\\
*+[o][F]{}\ar@{-}[rr]_{i+1}&&*+[o][F]{} \ar@{-}[rr]_i &&  *+[o][F]{w}  }$$
Moreover the degree, in $\mathcal{Q}$, of vertices $v$ and $w$  is one.
\end{prop}
\begin{proof}
Suppose that there is no $i$-edge incident to a vertex of the square $q_{i+1,i-1}$. 
Then, by Proposition~\ref{squareinconeclcomp2}, an edge incident to $q_{i+1,i-1}$ has labels $i+2$ or $i-2$.
By duality we may assume there is an edge incident to the square having label $i+2$. Then there is a square 
sharing an edge with $q_{i+1,i-1}$, as shown in the following figure. 
$$
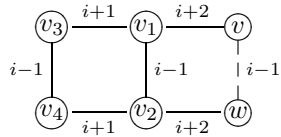
\xymatrix@-1.4pc{
*+[o][F]{v_3}\ar@{-}[rr]^{i+1} \ar@{-}[dd]_{i-1} &&*+[o][F]{v_1}   \ar@{-}[dd]^{i-1}\ar@{-}[rr]^{i+2}&& *+[o][F]{v} \ar@{--}[dd]^{i-1} \\
&&&&\\
*+[o][F]{v_4}\ar@{-}[rr]_{i+1}&&*+[o][F]{v_2} \ar@{-}[rr]_{i+2}&& *+[o][F]{w} }$$
By Proposition~\ref{squareedge} there is no $i$-edge incident to the vertices of the figure above. Furthermore, the vertices $v_1$ and $v_2$ have degree three and the vertices $v_3$ and $v_4$ have degree two. Thus one of the vertices on the right must have degree three. Suppose $v$ has degree three. Then it is $(i+3)$-incident to another vertex, and we get a square $q_{i+3,i-1}$ as pictured below. 
$$
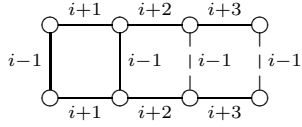
\xymatrix@-1.4pc{
*+[o][F]{}\ar@{-}[rr]^{i+1} \ar@{-}[dd]_{i-1} &&*+[o][F]{}   \ar@{-}[dd]^{i-1}\ar@{-}[rr]^{i+2}&& *+[o][F]{} \ar@{-}[rr]^{i+3} \ar@{--}[dd]^{i-1} && *+[o][F]{} \ar@{--}[dd]^{i-1}\\
&&&&&&\\
*+[o][F]{}\ar@{-}[rr]_{i+1}&&*+[o][F]{} \ar@{-}[rr]_{i+2}&& *+[o][F]{} \ar@{-}[rr]_{i+3} && *+[o][F]{}  }$$
Again, by  Proposition~\ref{squareedge}, there is no $i$-edges incident to the vertices of this figure. Continuing this process we get a graph that doesn't have $i$-edges, hence $\Gamma$ is disconnected, a contradiction.

Hence there is a vertex of the square $i$-adjacent to another vertex $v$. Consider $v_1$, $v_2$, $v_3$ and $v_4$ as in the following figure.
$$
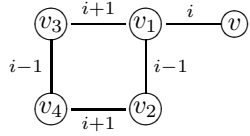
\xymatrix@-1.4pc{
*+[o][F]{v_3}\ar@{-}[rr]^{i+1} \ar@{-}[dd]_{i-1} &&*+[o][F]{v_1}   \ar@{-}[dd]^{i-1}\ar@{-}[rr]^i&& *+[o][F]{v} \\
&&&&\\
*+[o][F]{v_4}\ar@{-}[rr]_{i+1}&&*+[o][F]{v_2}  &&   && }$$

If the degree of $v$, in $\mathcal{Q}$, is greater than one, then there exist an edge with label $j$ not consecutive with $i$ and incident to $v$.
Then there is an alternating square $q_{i,j}$ containing $v$ and $v_1$.
Hence $v_1$ has degree at least four in  $\mathcal{G}$, which is not possible by Proposition ~\ref{squareedge}.
Thus the degree of $v$ is one in $\mathcal{Q}$.

Suppose that $v_4$ is $i$-adjacent, in $\mathcal{Q}$, to a vertex $w$. Then the degree of $w$, analogously to $v$, is one.  Thus  either $v_2$ or $v_3$ is $k$-adjacent to another vertex.  As $k$ is not consecutive either with $i+1$ or $i-1$ there is a square $q_{k, i\pm 1}$ containing $v_1$ or $v_4$. Thus $v_1$ or $v_4$ is of degree four and we get a contradiction.

Now suppose that neither $v_2$ nor $v_3$ is $i$-adjacent to any vertex.
Either $v_2$, $v_3$, or $v_4$ has degree three. Using duality we may assume that either $v_3$ or $v_4$ has degree three. Suppose at first that $v_3$ has degree three.
By the same arguments as above $v_3$ is $(i+2)$-adjacent to another vertex $v_5$. Then there exists an alternating square $q_{i+2,i-1}$ containing $v_3$, $v_4$, $v_5$, and $v_6$ as seen in the following figure. Thus, both $v_3$ and $v_4$ have degree three.

$$
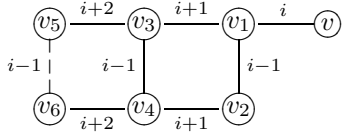
\xymatrix@-1.4pc{
*+[o][F]{v_5}\ar@{-}[rr]^{i+2}\ar@{--}[dd]_{i-1}&&*+[o][F]{v_3}\ar@{-}[rr]^{i+1} \ar@{-}[dd]_{i-1} &&*+[o][F]{v_1}   \ar@{-}[dd]^{i-1}\ar@{-}[rr]^i&& *+[o][F]{v} \\
&&&&&&\\
*+[o][F]{v_6}\ar@{-}[rr]_{i+2}&&*+[o][F]{v_4}\ar@{-}[rr]_{i+1}&&*+[o][F]{v_2} && }$$

  On the other hand, if we assume that $v_4$ has degree three, then using duality, we may assume it is incident to an edge of label $i+2$.  Thus there is an alternating square $q_{i-1,i+2}$ containing $v_3$; thus both situations give the same result.
  
Now if the other $i$-edge of $\mathcal{Q}$ is incident to either $v_5$ or $v_6$, this creates an alternating square $q_{i,i+2}$ containing either $v_3$ or $v_4$, which is not possible because they have degree three. Continuing this process we never get a way to connect the other $i$-edge, giving a contradiction.  

Hence either $v_2$ or $v_3$ is $i$-adjacent,  in $\mathcal{Q}$, to some vertex $w$. Moreover the same argument used to prove that $v$ has degree one holds for $w$.
\end{proof}
\begin{prop}\label{notconwithsquare}
Let $n\geq 9$. If $\mathcal{G}$ has a connected 2-fracture graph, then either the 2-fracture graph is a tree or $\Gamma$ is embedded in $C_2\wr S_{n/2}$.
\end{prop}
\begin{proof}
Suppose  $\mathcal{G}$ has a connected 2-fracture graph that is not a tree and let  $\mathcal{Q}$ be a 2-fracture graph  with a square $q_{i+1,i-1}$ accordantly with Proposition~\ref{minimalsquare}. 

By Proposition~\ref{i-j=2} there are two $i$-edges incident to the square. Let $v_1$, $v_2$, $v_3$ and $v_4$ be the vertices of $q_{i+1,i-1}$, as in the following figure.
$$
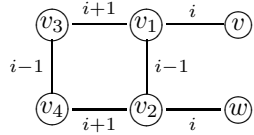
\xymatrix@-1.4pc{
*+[o][F]{v_3}\ar@{-}[rr]^{i+1} \ar@{-}[dd]_{i-1} &&*+[o][F]{v_1}   \ar@{-}[dd]^{i-1}\ar@{-}[rr]^i&& *+[o][F]{v} \\
&&&&\\
*+[o][F]{v_4}\ar@{-}[rr]_{i+1}&&*+[o][F]{v_2} \ar@{-}[rr]_i &&  *+[o][F]{w}  }$$

The degree of $v$ and $w$ are both one. Hence, by Proposition~\ref{i-j=2}, either  $v_3$ or $v_4$ has degree three.
Suppose without loss of generality that $v_3$ has degree three.
Now the label of the edge incident to $v_3$, not in the square $q_{i-1,i+1}$, must have label $i+2$. 
This creates an alternating square $q_{i-1,i+2}$ in $\mathcal{G}$. 
In particular the degree of $v_4$ is also three.
$$
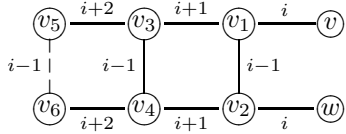
\xymatrix@-1.4pc{
*+[o][F]{v_5}\ar@{-}[rr]^{i+2}\ar@{--}[dd]_{i-1}&&*+[o][F]{v_3}\ar@{-}[rr]^{i+1} \ar@{-}[dd]_{i-1} &&*+[o][F]{v_1}   \ar@{-}[dd]^{i-1}\ar@{-}[rr]^i&& *+[o][F]{v} \\
&&&&&&\\
*+[o][F]{v_6}\ar@{-}[rr]_{i+2}&&*+[o][F]{v_4}\ar@{-}[rr]_{i+1}&&*+[o][F]{v_2} \ar@{-}[rr]_i&& *+[o][F]{w} }$$
As $n\geq 9$ either $v_5$ or $v_6$ has degree three. Assume that $v_5$ is incident to another vertex $v_7$. Then the label of the edge $\{v_5,v_7\}$ must be $(i+3)$ and the degree of $v_5$ must be three. Then there exists an alternating square $q_{i-1,i+3}$ in $\mathcal{G}$, and therefore the degree of $v_6$ must be three. If $n=10$ we are done. Otherwise we continue this process. The result is the graph below where $i=1$.
$$
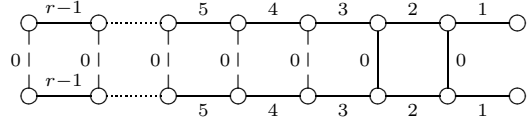
\xymatrix@-1.4pc{
*+[o][F]{}\ar@{--}[dd]_0\ar@{-}[rr]^{r-1}&&*+[o][F]{}\ar@{--}[dd]_0\ar@{.}[rr]&&*+[o][F]{}\ar@{-}[rr]^5\ar@{--}[dd]_0&&*+[o][F]{}\ar@{-}[rr]^4\ar@{--}[dd]_0&&*+[o][F]{}\ar@{-}[rr]^3\ar@{--}[dd]_0&&*+[o][F]{}\ar@{-}[rr]^2 \ar@{-}[dd]_0 &&*+[o][F]{}   \ar@{-}[dd]^0\ar@{-}[rr]^1&& *+[o][F]{} \\
&&&&&&&&&&&&&&\\
*+[o][F]{}\ar@{-}[rr]^{r-1}*+[o][F]{}&&*+[o][F]{}\ar@{.}[rr]&&*+[o][F]{}\ar@{-}[rr]_5&&*+[o][F]{}\ar@{-}[rr]_4&&*+[o][F]{}\ar@{-}[rr]_3&&*+[o][F]{}\ar@{-}[rr]_2&&*+[o][F]{} \ar@{-}[rr]_1&& *+[o][F]{} }$$
 Now the permutation graph of $\Gamma$ is the graph above or has another $0$-edge connecting the vertices on the right. Indeed these are all the edges of $\mathcal{G}$ for otherwise there is a cycle containing an edge of  $\mathcal{Q}$ and no other edge with the same label, a contradiction.
In any case we get $\Gamma$ embedded into $C_2\wr S_{\frac{n}{2}}$  (since all the permutations $\rho_i$ preserve the partition whose parts are the $0$-edges and  the pair of vertices on the right). 
\end{proof}
\begin{coro}
With the assumptions of Proposition~\ref{notconwithsquare}, if the 2-fracture graph is not a tree, there are two possibilities for $\Gamma$, namely
\begin{enumerate}
\item The permutation representation graph is the following and $\Gamma\cong C_2\times S_{n/2}$.
$$
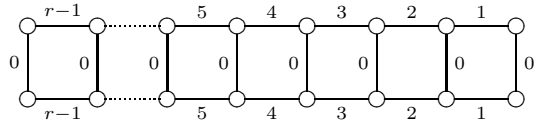
\xymatrix@-1.4pc{
*+[o][F]{}\ar@{-}[dd]_0\ar@{-}[rr]^{r-1}&&*+[o][F]{}\ar@{-}[dd]_0\ar@{.}[rr]&&*+[o][F]{}\ar@{-}[rr]^5\ar@{-}[dd]_0&&*+[o][F]{}\ar@{-}[rr]^4\ar@{-}[dd]_0&&*+[o][F]{}\ar@{-}[rr]^3\ar@{-}[dd]_0&&*+[o][F]{}\ar@{-}[rr]^2 \ar@{-}[dd]_0 &&*+[o][F]{}   \ar@{-}[dd]^0\ar@{-}[rr]^1&& *+[o][F]{} \ar@{-}[dd]^0\\
&&&&&&&&&&&&&&\\
*+[o][F]{}\ar@{-}[rr]_{r-1}*+[o][F]{}&&*+[o][F]{}\ar@{.}[rr]&&*+[o][F]{}\ar@{-}[rr]_5&&*+[o][F]{}\ar@{-}[rr]_4&&*+[o][F]{}\ar@{-}[rr]_3&&*+[o][F]{}\ar@{-}[rr]_2&&*+[o][F]{} \ar@{-}[rr]_1&& *+[o][F]{} }$$
\item 
 The permutation representation graph is the following and $\Gamma$ depends on parity of $n$.
 If $n/2$ is even, then $\Gamma\cong C_2\wr S_{n/2}$ and if $n/2$ is odd then $\Gamma\cong C_2^{n/2-1}: S_{n/2}$ which is a subgroup of index $2$ in $C_2\wr S_{n/2}$.
$$
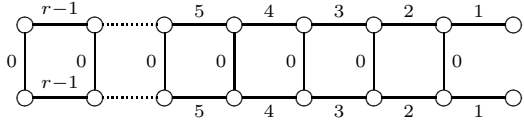
\xymatrix@-1.4pc{
*+[o][F]{}\ar@{-}[dd]_0\ar@{-}[rr]^{r-1}&&*+[o][F]{}\ar@{-}[dd]_0\ar@{.}[rr]&&*+[o][F]{}\ar@{-}[rr]^5\ar@{-}[dd]_0&&*+[o][F]{}\ar@{-}[rr]^4\ar@{-}[dd]_0&&*+[o][F]{}\ar@{-}[rr]^3\ar@{-}[dd]_0&&*+[o][F]{}\ar@{-}[rr]^2 \ar@{-}[dd]_0 &&*+[o][F]{}   \ar@{-}[dd]^0\ar@{-}[rr]^1&& *+[o][F]{} \\
&&&&&&&&&&&&&&\\
*+[o][F]{}\ar@{-}[rr]^{r-1}*+[o][F]{}&&*+[o][F]{}\ar@{.}[rr]&&*+[o][F]{}\ar@{-}[rr]_5&&*+[o][F]{}\ar@{-}[rr]_4&&*+[o][F]{}\ar@{-}[rr]_3&&*+[o][F]{}\ar@{-}[rr]_2&&*+[o][F]{} \ar@{-}[rr]_1&& *+[o][F]{} }$$
\end{enumerate}

\label{2fracfinal}
\end{coro}

Finally in this section, we prove the main theorem in the case where all
$\Gamma_i$ are intransitive and $2$-fracture graphs exist. So our assumptions
now are:
\begin{itemize}\itemsep0pt
\item $\Gamma\cong A_n$;
\item all $\Gamma_i$'s are intransitive, and each $\rho_i$ interchanges at least
two pairs of points in different $\Gamma_i$-orbits.
\end{itemize}

The conclusions we reached from the second assumption are given in Propositions
\ref{dislc} and \ref{notconwithsquare}: either
\begin{itemize}\itemsep0pt
\item there is a $2$-fracture graph of which one component is a tree and the others are unicyclic; or
\item the permutations are as given in Corollary~\ref{2fracfinal}.
\end{itemize}
In the first case, such a graph has $n-1$
edges. So $2r\le n-1$, as required. 
In the second, the group $\Gamma$ is not $A_n$.

\section{All $\Gamma_i$'s intransitive: no $2$-fracture graphs}
\label{s:no2frac}

In this section we continue to handle the case where all subgroups $\Gamma_i$ are intransitive.  In particular, we deal with the case where $\Gamma$ does not have any possible 2-fracture graph.   Although some of our results are more general, throughout this section we will make the assumption that $\Gamma$ is isomorphic to $A_n$.

Suppose that all maximal parabolic subgroups of $\Gamma$ are intransitive but there exists $i\in\{0,\ldots, r-1\}$ such that $\rho_i$ permutes only one pair $\{a,b\}$ of vertices in different $\Gamma_i$-orbits. 
Consequently, the only generators that can act non-trivially on $a$ and $b$ are $\rho_{i-1}$ and $\rho_{i+1}$. 

We will say that the orbit of $a$ is the first $\Gamma_i$-orbit and the orbit of $b$ is the second $\Gamma_i$-orbit.
Let $n_1$ and $n_2$ be the sizes of the first and the second $\Gamma_i$-orbit, respectively; and let $A$ and $B$ be the correspondent groups, determined by the action of $\Gamma_i$ on each orbit.  Both $A$ and $B$ are string groups generating by involutions (or sggi's). Indeed let $\rho_j=\alpha_j\beta_j$ with $\alpha_j$ and $\beta_j$ being the permutations in each $\Gamma_i$-orbit.
Then $A := \langle \alpha_i\, |\, i\in \{0,\ldots, r-1\}\rangle$ and $B := \langle \beta_i\, |\, i\in \{0,\ldots, r-1\}\rangle$.

\begin{prop}\label{CCD}
If  $A$ is primitive, then the set $J_A:= \{ i \,|\, i \in \{0,\ldots, r-1\}$ and $\alpha_i \neq 1_A\}$ is an interval.  The same result holds for $B$.
\end{prop}
\begin{proof}
Suppose that $J_A$ is not an interval.
Then $A = H \times K$ for $H = \langle \alpha_j \mid j \in J_1 \rangle$ and $K = \langle \alpha_j \mid j \in J_2 \rangle$ for some disjoint index sets $J_1$ and $J_2$ such that $J_A = J_1 \cup J_2$.  As both $H$ and $K$ are transitive on the $n_1$ points, the cardinality of  $J_1$ and $J_2$ is at least two.
Moreover each generator $\alpha_j$ commutes with all generators of a transitive group on $n_1$ points, either $H$ or $K$, which implies that it  has full support on the first $\Gamma_i$-orbit.  Therefore, it has a nontrivial action on $a$.  However, we have seen that the only generators that can act nontrivially on $a$ are $\rho_{i-1}$ and $\rho_{i+1}$.  This gives a contradiction, so $J_A$ is an interval.
The proof also works for $B$.
\end{proof}

Thanks to Proposition~\ref{CCD} we consider (up to duality) separately  the following cases : Case (1) $A$ and $B$ are both imprimitive; Case (2): $J_A$ and $J_B$ are intervals and $i\notin\{ 0,r-1\}$; Case (3): we deal with the remaining cases, particularly we assume that $J_B= \emptyset$ or an interval.
\subsection{Case (1): $A$ and $B$ are both imprimitive}

Let $A$ be embedded into $S_{k_1}\wr S_{m_1}$ and $B$ be embedded into $S_{k_2}\wr S_{m_2}$ with $n_1=m_1k_1$ and $n_2=m_2k_2$. 
Consider a minimal subset $M$ of the set of generators of $\Gamma_i$ generating the group induced on the two block systems. Let $R$ be the set containing the remaining generators of $\Gamma_i$.  

Consider the permutation representation graph  $\mathcal{X}$  for the block action, that is, a graph having  $m_1+m_2$ vertices, corresponding to the blocks, and with a $j$-edge between two blocks whenever $\rho_j$ swaps them.
As $\Gamma_i$ has exactly two orbits, the graph $\mathcal{X}$ has two connected components.   
Also, consider the subgraph $\mathcal{\bar{X}}$ of $\mathcal{X}$ with the same vertices and with a $j$-edge for each element of $M$, between blocks in different $\Gamma_j$-orbits. This is a fracture subgraph of $\mathcal{X}$, particularly $\mathcal{\bar{X}}$ has no cycles. Hence $|M| \leq m_1+m_2- 2 $.

Similarly, consider the graph $\mathcal{Y}$ with $k$ vertices corresponding to the $\langle M\rangle$-orbits, with  a $j$-edge between a pair of $\langle M\rangle$-orbits $L$ and $L'$  whenever there is  $x\in L$ such that $x\rho_j\in L'$ with $\rho_j\in R$. Let $\mathcal{\bar{Y}}$ be a fracture subgraph of  $\mathcal{Y}$  having only one $j$-edge for each element $\rho_j\in R$ between $\langle M\rangle$-orbits in different $\Gamma_j$-orbits. As before, $\mathcal{\bar{Y}}$ has no cycles and has at least two components. Hence $|R| \leq k - 2$. Note that $k\leq k_1+k_2$, hence $|R| \leq k_1+k_2- 2$.

\begin{prop}\label{M=m-2}
If $|M|= m_1+m_2-2$ then  $\mathcal{X}$ has two connected components and consecutive labels. Up to a duality,  $\mathcal{X}$ is the following graph.
$$\xymatrix@-1.3pc{ *+[F]{} \ar@{-}[rr]^{i-m_1+1} && *+[F]{ }  \ar@{.}[rr] && *+[F]{ }  \ar@{-}[rr]^{i-2} &&*+[F]{}  \ar@{-}[rr]^{i-1} && *+[F]{a} &&  *+[F]{b}  \ar@{-}[rr]^{i+1} && *+[F]{ } \ar@{-}[rr]^{i+2} && *+[F]{ }  \ar@{.}[rr]&&*+[F]{ } \ar@{-}[rr] ^{i+m_2-1}&&*+[F]{ }} $$
\end{prop}
\begin{proof}
Since $|M|= m_1+m_2-2$ and $\mathcal{\bar{X}}$ has two components, $\mathcal{\bar{X}} = \mathcal{X}$.  
As $\rho_i$ only swaps the points $a$ and $b$ in different $\Gamma_i$-orbits, a $j$-edge of $\mathcal{X}$ incident to the blocks containing $a$ and $b$,  must be consecutive with $i$.

The graph $\mathcal{X}$ does not have any alternating squares, and thus only edges with consecutive labels are incident.  Up to duality, we have determined $\mathcal{X}$.
\end{proof}

\begin{prop}\label{M=m-3(1)}
If $|M|= m_1+m_2-3$ then, up to duality, either $m_1=2$ or $m_1\geq 4$,  accordantly to one of the following graphs.
$$(1) \,
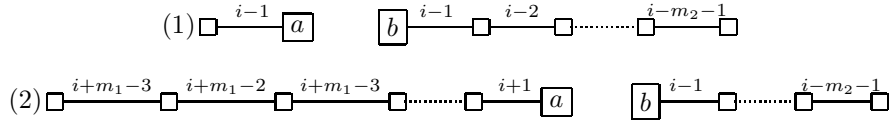
\xymatrix@-1.2pc{  *+[F]{} \ar@{-}[rr]^{i-1} &&*+[F]{a}  &&*+[F]{b} \ar@{-}[rr]^{i-1} && *+[F]{ }  \ar@{-}[rr]^{i-2}& &*+[F]{} \ar@{.}[rr] &&*+[F]{} \ar@{-}[rr] ^{i-m_2-1}&&*+[F]{ } } $$
$$(2) \,\xymatrix@-1.3pc{  *+[F]{} \ar@{-}[rrr]^{i+m_1-3} &&&*+[F]{} \ar@{-}[rrr]^{i+m_1-2} &&&*+[F]{} \ar@{-}[rrr]^{i+m_1-3} &&&*+[F]{} \ar@{.}[rr] &&*+[F]{} \ar@{-}[rr] ^{i+1}&&*+[F]{a }  & &*+[F]{b} \ar@{-}[rr]^{i-1} &&*+[F]{}  \ar@{.}[rr]&&*+[F]{} \ar@{-}[rr] ^{i-m_2-1}&&*+[F]{ }} $$
\end{prop}
\begin{proof}
As $\mathcal{\bar{X}}$ has $m_1+m_2-3$ edges, it has 3 connected components.  Since $\mathcal{X}$ has two components, there exists a $j$-edge $e_j$ of $\mathcal{X}$ such that $\mathcal{\bar{X}} \cup e_j$ has two components and no cycles.  Hence, as in Proposition~\ref{M=m-2} incident edges of $\mathcal{\bar{X}} \cup e_j$ must  have consecutive labels.

First suppose that the $j$-edges of $\mathcal{\bar{X}} \cup e_j$ are different $\Gamma_i$-orbits. We claim that $j = i \pm 1$. Assume the contrary.  Then there is a path in the first $\Gamma_i$-orbit connecting the block containing $a$ to the block moved by $\rho_j$; the same happens in the second  $\Gamma_i$-orbit: there is a path connecting the block containing $b$ to the block moved by $\rho_j$. 
$$\xymatrix@-1.3pc{  *+[F]{} \ar@{-}[rr]^j  &&*+[F]{}  \ar@{.}[rr] && *+[F]{} \ar@{-}[rr]^{i-1} &&*+[F]{a}  &&*+[F]{b} \ar@{-}[rr]^{i+1} && *+[F]{ }  \ar@{.}[rr]& &*+[F]{} \ar@{-}[rr]^j  &&*+[F]{}  } $$
Assume that $ j > i +1$. Then both of these paths have to contain the label $i+1$, a contradiction. Thus $j=i\pm 1$. Up to duality we may consider $j=i-1$ corresponding to the first possibility for $\mathcal{X}$.

Now consider the $j$-edges of $\mathcal{\bar{X}} \cup e_j$ in the same $\Gamma_i$-orbit; assume it is the first orbit.  There is a path in $\mathcal{\bar{X}} \cup e_j$ joining the four blocks swapped by $\rho_j$ which has consecutive labels and no repeating labels other than $j$.    Thus this path is a single edge with label $j \pm 1$.  Assume that this edge has label $j+1$. At least one of the four vertices of $\mathcal{\bar{X}} \cup e_j$ that are incident to the $j$-edges has degree one.  Otherwise, there would be another repeated label. This gives the second  possibility for the graph $\mathcal{X}$ where $j=i-m_1-3$ with $m_1\geq 4$.
\end{proof}

\begin{prop}\label{R<=k-3}
If an element of $R$ has a nontrivial action between more than one pair of $\langle M\rangle$-orbits, then $|R|\leq k -3$.
\end{prop}

\begin{proof}
In this case $\mathcal{\bar{Y}}$ has at least three connected components, and still has no cycles.  Thus $|R|\leq k-3$.
\end{proof}

In what follows let $L_a$ be the $\langle M\rangle$-orbit containing $a$ and $L_b$ the $\langle M\rangle$-orbit containing $b$.

\begin{prop}\label{M=m-3(2)}
If $|M|=m_1+m_2-3$ then $|R|\leq k_1+k_2-3$.
\end{prop}
\begin{proof}
As $\mathcal{\bar{Y}}$ is a forest and has at least two components, $|R|\leq k-2$. 
If $k< k_1+k_2$ then $|R|\leq k_1+k_2-3$. Assume that $k=k_1+k_2$ and that we have 
the equality $|R|= k_1+k_2-2$. Then $\mathcal{\bar{Y}}$ has exactly two components and there are at least two $\langle M\rangle$-orbits in each $\Gamma_i$-orbit.
As $|M|=m_1+m_2-3$ we have, up to duality, one of the two possibilities for $\mathcal{X}$ given in Proposition~\ref{M=m-3(1)}.

First suppose that $m_1+m_2> 4$.  There are, up to duality, the two possibilities for $\mathcal{X}$ given in  Proposition~\ref{M=m-3(1)}. In graph (1) there is only one possibility to connect $L_a$ to another $\langle M\rangle$-orbit, that is using a pair of $(i+1)$-edges. Then in both cases,  (1) and (2), there is only one possibility to connect  $L_b$ to another $\langle M\rangle$-orbit, that is,  with a single edge with label $l=i-m_2-2$ (between vertices of the last block of the second $\Gamma_i$-orbit).
Furthermore, $\rho_l$ swaps exactly one pair of vertices of these $\langle M\rangle$-orbits. Then as $\Gamma$ is even $\rho_l$ must swap another pair of $\langle M\rangle$-orbits, hence  we have a contradiction with Proposition~\ref{R<=k-3}. 

Now let $m_1+m_2=4$. In this case $\mathcal{X}$ is as in figure (1) of Proposition~\ref{M=m-3(1)} and $M=\{\rho_{i-1}\}$. To connect $L_a$ to another $\langle M\rangle$-orbit, or $L_b$ to another $\langle M\rangle$-orbit, there are only two possibilities for the labels, either $l=i-2$  or $l=i+1$. By Proposition~\ref{R<=k-3} we may assume that either $L_a$ or $L_b$ is $(i-2)$-adjacent to another $\langle M\rangle$-orbit. Then we use the fact that $\Gamma$ is even and  Proposition~\ref{R<=k-3} to get a contradiction.
\end{proof}

\begin{prop}\label{MM=m-2}
If $|M| = m_1+m_2-2$ then $|R| \leq k_1+k_2-4$.
\end{prop}
\begin{proof}
Up to duality we may consider $\mathcal{X}$ as in Proposition~\ref{M=m-2}. First the  elements of $R$ must fix all the blocks.
Otherwise there is $\rho_j\in M$ such that $\Gamma_j$ is transitive, a contradiction.  Let $C$ be the set of generators in $R$ that commute with all elements of $M$. We have that $|R \setminus C| \leq 2$. 

There is at most one $\langle M\rangle$-orbit adjacent to $L_a$ in $\mathcal{\bar{Y}}$, and the label of the edge which might connect them is $\rho_f := \rho_{i - m_1}$.  Similarly, there is at most one $\langle M\rangle$-orbit adjacent to $L_b$ in $\mathcal{\bar{Y}}$, and the label of the edge which might connect them is $\rho_l := \rho_{i + m_2}$.  We denote the $\langle M\rangle$-orbits adjacent to $L_a$ and $L_b$, $L'_a$ and $L'_b$ resp. if they exist.
Furthermore, both $\rho_f$ and $\rho_l$, if they exist, swap a single pair of points in these $\langle M\rangle$-orbits.  
$$ 
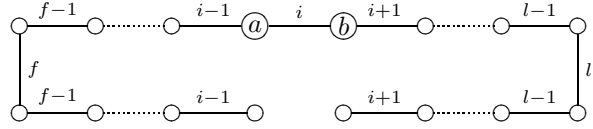
\xymatrix@-1.3pc{*+[o][F]{} \ar@{-}[rr]^{f-1}\ar@{-}[dd]^{f}  &&*+[o][F]{}  \ar@{.}[rr]&&*+[o][F]{}\ar@{-}[rr] ^{i-1}&& *+[o][F]{a}\ar@{-}[rr] ^i  &&*+[o][F]{b} \ar@{-}[rr]^{i+1} && *+[o][F]{ }  \ar@{.}[rr]&&*+[o][F]{} \ar@{-}[rr]^{l-1}  &&*+[o][F]{} \ar@{-}[dd]^l\\
&&&&&&&&&&&&&&\\
 *+[o][F]{} \ar@{-}[rr]^{f-1} &&*+[o][F]{}  \ar@{.}[rr]&& *+[o][F]{}\ar@{-}[rr]^{i-1}&& *+[o][F]{} &&*+[o][F]{} \ar@{-}[rr]^{i+1} && *+[o][F]{ }  \ar@{.}[rr]&&*+[o][F]{} \ar@{-}[rr]^{l-1}  &&*+[o][F]{}
   } $$
Since $\Gamma$ is even, $\rho_f$ and $\rho_l$ must both have nontrivial action on a point in another $\langle M\rangle$-orbit. We now consider separately the cases: $|R \setminus C| = 0$, $|R \setminus C| = 1$ and $|R \setminus C| = 2$.

If $R \setminus C = \emptyset$, then $L_a$ and $L_b$ are the unique $\langle M\rangle$-orbits, $R = \emptyset$ and $k=2$.  Thus $|R|=0\leq k_1+k_2-4$.

Suppose that  $R \setminus C=\{\rho_f\}$. In this case $L_b$ coincide with the second $\Gamma_i$-orbit. As $\mathcal{Y}$ has at least two $f$-edges, by Proposition~\ref{R<=k-3}, $|R| \leq k-3$. In addition, $k\leq k_1+k_2-1$, thus  $|R| \leq k-3\leq k_1+k_2-4$.

Now let $ R \setminus C=\{ \rho_f, \rho_l \}$.  We may assume $\mathcal{\bar{Y}}$ with an $f$-edge between $L_a$ and $L'_a$, and with an $l$-edge between $L_b$ and $L'_b$.  Since $\Gamma$ is even, both $\rho_f$ and $ \rho_l$ act non-trivially on a point in a $\langle M\rangle$-orbit other than $L_a, L_b, L'_a$ and $L'_b$. By Proposition~\ref{R<=k-3},
 $\mathcal{\bar{Y}}$  has at least three components. Furthermore, to have three components there must exist a pair of $\langle M\rangle$-orbits $\{L,L'\}$ such that both  $ \rho_l$ and $\rho_f$  swap a point in $L$ with  a point in $L'$. Indeed,  $\rho_l$ acts trivially on the points not in $L, L', L_b$ or $L_b'$ and  $\rho_f$ acts trivially on the points not in $L, L', L_a$ or $L_a'$.  Let us assume that $L$ and $L'$ are both in the first $\Gamma_i$-orbit. Then $\rho_l$ swaps $L$ and $L'$ entirely. 
 
First suppose $L = L'_a$.  Then, both $\rho_f$ and $\rho_l$, act nontrivially on a point in $L'_a$.  Thus there is an alternating square, with labels $f$ and $l$, containing this point.  Hence $\rho_l$ acts nontrivially on a point in $L_a$, which we have shown is impossible. Consequently  $L \neq  L'_a$, as well as $L'\neq L'_a$. 

Since $\Gamma_i$ has only two orbits, there is a generator $\rho_j\in R$  that sends a point in either $L$ or $L'$ to a point not in these two $\langle M\rangle$-orbits; without loss of generality we may assume that $\rho_j$ sends a point in $L'$ to a point in a $\langle M\rangle$-orbit which we denote $L''$. Then the unique possibility is $j = l + 1$. However, then $\rho_j$ commutes with all the elements of $M$ which act between blocks in the first $\Gamma_i$-orbit.  Thus $\rho_j$ swaps $L'$ and $L''$.  This guarantees an  alternating  square with labels $j$ and $f$, with $\rho_f$ acting nontrivially on $L''$. This implies that $L''=L'_a$. Moreover this forces $\rho_j$ to have a nontrivial action on $L_a$, a contradiction.
\end{proof}

\begin{prop}\label{bothimp}
If $A$ and $B$ are both imprimitive then  $r \leq  \frac{n -1}{2}$.
\end{prop}
\begin{proof}
By Propositions~\ref{M=m-3(2)} and~\ref{M=m-2}, we have $r = |R| + | M |+1\leq k_1+k_2+ m_1+m_2 - 5$. As $k_1+ m_1 + k_2 + m_2 -5 \leq \frac{k_1m_1 +k_2m_2 - 1}{2}$, we conclude that $r \leq  \frac{n -1}{2}$.
\end{proof}

\subsection{Case 2: $J_A$ and $J_B$ are intervals and $i\notin\{ 0,r-1\}$} 
We first recall two propositions on sggi that can be found in \cite{flm2} that we use to deal with this case.
\begin{prop}\cite[Proposition 3.3]{flm2}\label{genAn}
Let $\Phi = \langle \alpha_0,\,\ldots,\,\alpha_{d-1} \rangle$ be a transitive permutation group acting on the set of points $\{ 1,\ldots, n \}$ with $n \geq 5$, and let $\Phi^* = \langle \alpha_0,\,\ldots,\,\alpha_{d-1},\,\alpha_d,\, \alpha_{d+1} \rangle$, where
\[\begin{array}{rl}
\alpha_r&=(i,n+1)(n+2,n+3) \mbox{ for some } i \in \{1,\ldots,n\}\\
\alpha_{r+1}&=(n+1,n+2)(n+3,n+4).
\end{array}\]
Then $\Phi^*$ is isomorphic to $S_{n+4}$ if it contains an odd permutation, and to $A_{n+4}$ otherwise .
\end{prop}


The term sesqui-extension was first introduced in \cite{flm}. Let us recall its meaning.
Let $\Phi = \langle \alpha_0,\ldots,\alpha_{d-1} \rangle$ be a sggi, and let $\tau$ be an involution in a supergroup of $\Phi$ such that $\tau \not \in \Phi$ and $\tau$ commutes with all of $\Phi $.  For fixed $k$, we define the group $\Phi^*= \langle \alpha_i \tau^{\eta_i}\,|\, i\in \{0,\,\ldots,\,d-1\} \rangle$ where $\eta_i = 1$ if $i=k$ and 0 otherwise, the {\it sesqui-extension} of $\Phi $ with respect to $\alpha_k$ and $\tau$.

\begin{prop} \label{sesqui}\cite[Proposition 5.4]{flm2}
If $\Phi =\left<\alpha_i\,|\, i=0,\,\ldots,\,d-1\right>$ and $\Phi^*=\langle \alpha_i \tau^{\eta_i}\,|\, i\in \{0,\,\ldots,\,d-1\} \rangle$ is a sesqui-extension of $\Phi$ with respect to $\alpha_k$, then:
\begin{enumerate}
\item $ \Phi^* \cong \Phi$ or $\Phi^* \cong \Phi\times \langle \tau\rangle\cong\Phi\times 2$.
\item whenever $\tau \notin \Phi^*$, $\Phi$ is a string C-group if and only if $\Phi^*$ is a string C-group.
\end{enumerate}
\end{prop}

In this case we may assume that $A=\Gamma_{<i}$ and $B=\Gamma_{>i}$. 
As $A$ or $B$ can have small degree $\leq 11$, in what follows we list all primitive even  string C-groups of small degree, with intransitive maximal parabolic subgroups, having rank $r\geq\frac{n-1}{2}$.

\begin{prop}\label{smalln}
Let  $\Phi$ be string C-group with a connected Coxeter diagram isomorphic to an even primitive group of degree $n\leq 11$. If $\Phi_j$ is intransitive for every $j\in\{0,\ldots,r-1\}$ then either $r \leq \frac{n-2}{2}$ or $\Phi$ has one of the permutation representation graphs given in Table~\ref{small}.
\begin{small}
\begin{table}
\begin{tabular}{|ccc|}
\hline
(1)&$D_{10}$& $ \xymatrix@-1pc{
     *+[o][F]{}   \ar@{-}[r]^0      & *+[o][F]{}  \ar@{-}[r]^1   & *+[o][F]{}  \ar@{-}[r]^0 & *+[o][F]{}  \ar@{-}[r]^1  & *+[o][F]{}  }$\\
(2)&$L_2(5)$ & $ \xymatrix@-1pc{
     *+[o][F]{}   \ar@{=}[r]^0_2      & *+[o][F]{}  \ar@{-}[r]^1   & *+[o][F]{}  \ar@{-}[r]^0 & *+[o][F]{}  \ar@{-}[r]^1 & *+[o][F]{}  \ar@{-}[r]^2 & *+[o][F]{}  }$\\
(3)&$L_2(5)$& $ \xymatrix@-1pc{
     *+[o][F]{}   \ar@{-}[r]^0      & *+[o][F]{}  \ar@{-}[r]^1   & *+[o][F]{}  \ar@{=}[r]^2_0 & *+[o][F]{}  \ar@{-}[r]^1 & *+[o][F]{}  \ar@{-}[r]^2 & *+[o][F]{}  }$\\
(4)&$A_9$ &$ \xymatrix@-1pc{ *+[o][F]{}   \ar@{-}[r]^0      & *+[o][F]{}  \ar@{-}[r]^1   & *+[o][F]{}  \ar@{-}[r]^0 & *+[o][F]{}  \ar@{-}[r]^1 & *+[o][F]{}  \ar@{-}[r]^2& *+[o][F]{} \ar@{-}[r]^3  & *+[o][F]{} \ar@{-}[r]^2& *+[o][F]{} \ar@{-}[r]^3  & *+[o][F]{} } $ \\
  (5)&$A_9$ &$ \xymatrix@-1pc{ *+[o][F]{}   \ar@{-}[r]^3      & *+[o][F]{}  \ar@{-}[r]^2   & *+[o][F]{}  \ar@{-}[r]^3 & *+[o][F]{}  \ar@{-}[r]^2 & *+[o][F]{}  \ar@{-}[r]^1& *+[o][F]{} \ar@{-}[r]^0 \ar@{-}[d]_2 & *+[o][F]{} \ar@{=}[d]^1_2\\
&&&&& *+[o][F]{}   \ar@{-}[r]_0      & *+[o][F]{} } $ \\
(6)&$A_9$ &$ \xymatrix@-1pc{ *+[o][F]{}   \ar@{-}[r]^3      & *+[o][F]{}  \ar@{-}[r]^2   & *+[o][F]{}  \ar@{-}[r]^1 \ar@{-}[d]_3& *+[o][F]{}  \ar@{-}[r]^0 \ar@{-}[d]_3& *+[o][F]{}\ar@{-}[d]_3\\
&*+[o][F]{}   \ar@{-}[r]_2  & *+[o][F]{}   \ar@{-}[r]_1      & *+[o][F]{}  \ar@{-}[r]_0 & *+[o][F]{}  } $ \\
(7)&$A_9$& $ \xymatrix@-1pc{ *+[o][F]{}   \ar@{-}[r]^3      & *+[o][F]{}  \ar@{-}[r]^2   & *+[o][F]{}  \ar@{-}[r]^1 \ar@{-}[d]_3& *+[o][F]{}  \ar@{-}[r]^0 \ar@{=}[d]_3^2& *+[o][F]{}\ar@{=}[d]_3^2\\
&*+[o][F]{}   \ar@{-}[r]_2  & *+[o][F]{}   \ar@{-}[r]_1      & *+[o][F]{}  \ar@{-}[r]_0 & *+[o][F]{}  } $ \\
(8)&$A_9$& $ \xymatrix@-1pc{ *+[o][F]{}   \ar@{-}[r]^1      &*+[o][F]{}   \ar@{-}[r]^0      & *+[o][F]{}  \ar@{-}[r]^1   & *+[o][F]{}  \ar@{-}[r]^2 \ar@{-}[d]_0& *+[o][F]{}  \ar@{-}[r]^3 \ar@{-}[d]_0& *+[o][F]{}\ar@{-}[d]_0\\
&&& *+[o][F]{}   \ar@{-}[r]_2      & *+[o][F]{}  \ar@{-}[r]_3 & *+[o][F]{}  } $\\
(9)&$A_9$ & $ \xymatrix@-1pc{  *+[o][F]{}   \ar@{=}[r]^0_2&*+[o][F]{}   \ar@{-}[r]^1      &*+[o][F]{}   \ar@{-}[r]^0      & *+[o][F]{}  \ar@{-}[r]^1   & *+[o][F]{}  \ar@{-}[r]^2 & *+[o][F]{}  \ar@{-}[r]^3 \ar@{-}[d]_1& *+[o][F]{}\ar@{-}[d]^1\\
&&&&  & *+[o][F]{}  \ar@{-}[r]_3 & *+[o][F]{}  } $\\
(10)&$A_{10}$ & $ \xymatrix@-1pc{
    *+[o][F]{}   \ar@{-}[r]^4      & *+[o][F]{}  \ar@{-}[r]^3   & *+[o][F]{}   \ar@{-}[r]^4      & *+[o][F]{}  \ar@{-}[r]^3   & *+[o][F]{}  \ar@{-}[r]^2 & *+[o][F]{}  \ar@{-}[r]^1 & *+[o][F]{}  \ar@{-}[r]^0& *+[o][F]{}  \ar@{-}[r]^1& *+[o][F]{}  \ar@{=}[r]^2_0 & *+[o][F]{}  }$\\
  (11)&$A_{10}$& $ \xymatrix@-1pc{
    *+[o][F]{}   \ar@{-}[r]^0      & *+[o][F]{}  \ar@{-}[r]^1   & *+[o][F]{}   \ar@{=}[r]^0_2      & *+[o][F]{}  \ar@{-}[r]^1   & *+[o][F]{}  \ar@{-}[r]^2 & *+[o][F]{}  \ar@{-}[r]^3 & *+[o][F]{}  \ar@{-}[r]^4& *+[o][F]{}  \ar@{-}[r]^3 &*+[o][F]{}  \ar@{-}[r]^4 & *+[o][F]{}  }$\\
(12)&$A_{10}$& $ \xymatrix@-1pc{  *+[o][F]{}   \ar@{-}[r]^4&*+[o][F]{}   \ar@{-}[r]^3      &*+[o][F]{}   \ar@{=}[r]^4_2      & *+[o][F]{}  \ar@{-}[r]^3   & *+[o][F]{}  \ar@{-}[r]^2 & *+[o][F]{}  \ar@{-}[r]^1& *+[o][F]{}  \ar@{-}[r]^0 \ar@{-}[d]_2& *+[o][F]{}\ar@{=}[d]^1_2\\
&&&&&& *+[o][F]{}  \ar@{-}[r]_0 & *+[o][F]{}  } $\\
 (13)&$A_{10}$&$ \xymatrix@-1pc{  *+[o][F]{}   \ar@{=}[r]^0_2&*+[o][F]{}   \ar@{-}[r]^1      &*+[o][F]{}   \ar@{-}[r]^0      & *+[o][F]{}  \ar@{-}[r]^1   & *+[o][F]{}  \ar@{-}[r]^2 & *+[o][F]{}  \ar@{-}[r]^3& *+[o][F]{}  \ar@{-}[r]^4 \ar@{-}[d]_2& *+[o][F]{}\ar@{=}[d]^3_2\\
&&&&&& *+[o][F]{}  \ar@{-}[r]_4 & *+[o][F]{}  } $\\
(14)&$A_{11}$ & $ \xymatrix@-1pc{
    *+[o][F]{}   \ar@{=}[r]^0_2      & *+[o][F]{}  \ar@{-}[r]^1   & *+[o][F]{}   \ar@{-}[r]^0      & *+[o][F]{}  \ar@{-}[r]^1   & *+[o][F]{}  \ar@{-}[r]^2 & *+[o][F]{}  \ar@{-}[r]^3 & *+[o][F]{}  \ar@{-}[r]^4& *+[o][F]{}  \ar@{=}[r]^3_5 &*+[o][F]{}  \ar@{-}[r]^4 & *+[o][F]{}   \ar@{-}[r]^5 & *+[o][F]{}  }$\\
    (15)&$A_{11}$& $ \xymatrix@-1pc{
    *+[o][F]{}   \ar@{=}[r]^0_2      & *+[o][F]{}  \ar@{-}[r]^1   & *+[o][F]{}   \ar@{-}[r]^0      & *+[o][F]{}  \ar@{-}[r]^1   & *+[o][F]{}  \ar@{-}[r]^2 & *+[o][F]{}  \ar@{-}[r]^3 & *+[o][F]{}  \ar@{-}[r]^4& *+[o][F]{}  \ar@{-}[r]^5 &*+[o][F]{}  \ar@{-}[r]^4 & *+[o][F]{}   \ar@{=}[r]^5_3 & *+[o][F]{}  }$\\
 (16)&$A_{11}$ & $ \xymatrix@-1pc{
    *+[o][F]{}   \ar@{-}[r]^0     & *+[o][F]{}  \ar@{-}[r]^1   & *+[o][F]{}   \ar@{=}[r]^0_2      & *+[o][F]{}  \ar@{-}[r]^1   & *+[o][F]{}  \ar@{-}[r]^2 & *+[o][F]{}  \ar@{-}[r]^3 & *+[o][F]{}  \ar@{-}[r]^4& *+[o][F]{}  \ar@{=}[r]^3_5 &*+[o][F]{}  \ar@{-}[r]^4 & *+[o][F]{}   \ar@{-}[r]^5 & *+[o][F]{}  }$\\
\hline
\end{tabular}
\caption{Even transitive string C-groups of degree $m$ with connected Coxeter diagram  having intransitive maximal parabolic subgroups and rank $\geq  \frac{m-1}{2}$ .}\label{small}
\end{table}
\end{small}
\end{prop}
\begin{proof}
We used \textsc{Magma} to get this result.
\end{proof}

\begin{prop}\label{small123}
Let $j\in \{0,\ldots,\, r-1\}$. Suppose that $\Gamma_{<j}$ is transitive on $m$ points and fixes the remaining $m-n$ points.
If $\Gamma_{<j}$ is a primitive group of degree $m<12$ and $j\geq \frac{m-1}{2}$ then it must have permutation representation graph (1), (2) or (3) of Table~\ref{small}. 
\end{prop}
The dual version of this proposition is also true.
\begin{proof}
By Proposition~\ref{smalln}, $\Gamma_{<j}$ has one of the 16 permutation representation graphs given in Table~\ref{small}.
Let $X:=\{1,\ldots,n\}\setminus \mathrm{Fix}(\Gamma_{<j})$ with $|X|=m$.
Then there exists a generator  $\rho_k$ of $\Gamma$, with $k\geq j$, such that  $c\rho_k=d$ with $c\in X$ and $d\not \in X$. 
As $\Gamma_{<j}$ is primitive, $k$ must be consecutive with $j-1$.
Moreover, $c$ has degree two with the edges incident to $c$ having labels $j-1$ and $j$. 
In graphs (8), (9),  (13) and (15) there is no such vertex $c$. Let us consider the remaining graphs.

Now let (4) be the permutation representation graph of $\Gamma_{<j}$. Then $\Gamma_{<j}=\Gamma_{<4}\cong A_9$ and $\Gamma_{>1}\cong A_{n-4}$ by Proposition~\ref{genAn} (note that $n-4\geq 5$), thus $\Gamma_{>1} \cap \Gamma_{<4}\cong A_5$. However, $\langle \rho_2, \rho_3 \rangle$ is a dihedral group. Consequently $\Gamma$ does not satisfy the intersection condition, a contradiction.

Only in cases (11) and (16) there are two possibilities for the vertex $c$, but as in case (4),  case (16) is self-dual. 
Using similar arguments, summarised below, we conclude that $\Gamma_{<j}$ cannot be any of the graphs of Table~\ref{small} except graphs (1), (2) and (3).  Let (11)* denote the dual of (11)  with  its labels interchanged by $k \leftrightarrow 4-k$. 
\begin{center}

\begin{tabular}{cl}
(5), (6), (7): &$\Gamma_{>1}\cong 2\times A_{n-4}$ or $\Gamma_{>1}\cong A_{n-4}$ (by Proposition~\ref{sesqui}); \\[5pt]
&$\,\Gamma_{<4}\cong A_9$;  $\,\Gamma_{>1} \cap \Gamma_{<4}\cong A_5\not \cong \langle \rho_2, \rho_3 \rangle$.\\[5pt]

(10), (11), (11)*, (12): &$\Gamma_{>2}\cong A_{n-5}$; $\,\Gamma_{<5}\cong A_{10}$;  $\,\Gamma_{>2} \cap \Gamma_{<5}\cong A_5\not \cong \langle \rho_3, \rho_4 \rangle$.\\[5pt]

(14), (16):& $\Gamma_{>3}\cong A_{n-6}$; $\,\Gamma_{<6}\cong A_{11}$;  $\,\Gamma_{>3} \cap \Gamma_{<6}\cong A_5\not \cong \langle \rho_4, \rho_5 \rangle$.
\end{tabular}
\end{center}

\end{proof}

\begin{prop}\label{RI}
 Let $\Phi=\langle \alpha_0,\ldots,\alpha_{d-1}\rangle$ be a  transitive sggi  embedded into $S_a\wr S_b$  with $X$ being the set of generators of  $\Phi$ generating the block action independently.
If  $\Phi_j$  is intransitive for all $j\in\{0,\ldots,d-1\}$, then 
\begin{enumerate}
\item $d\leq |X|+a-1$;
\item $d\leq  \frac{ab-2}{2}$ for $a,b\neq 2$ and $ab\neq 9$;
\item if $(a,d)=(2,b)$ or $(b,d)=(2,a)$ or $(a,b)=(3,3)$ then $\Phi$
either contains an odd
permutation or has disconnected diagram.
\end{enumerate}
\end{prop}
\begin{proof}
Consider a partition $P$ of $\{1,\ldots,ab\}$ determined by the orbits of $\langle X\rangle$. 
Now let $\mathcal{X}$ be a graph with $|P|$ vertices corresponding the partitions of $P$ and with exactly one  $j$-edge for each generator of $\Phi$ that is not in $X$, connecting two partitions in different $\Phi_j$-orbits. This graph has no cycles and $|P|\leq a$, thus
$d-|X|\leq a-1$.  Hence, $d\leq |X|+a-1$. Moreover as $X$ generate the bock action independenty $|X|\leq b-1$. Hence  $d\leq \leq a+b-2 \frac{ab-2}{2}$ for $a,b\neq 2$ and $ab\neq 9$.

Now suppose that $(a,d)=(2,b)$ or  $(b,d)=(2,a)$ and $\Phi$ is even, up to duality, $\Phi$ has the  following permutation representation graph, and 

$$
   
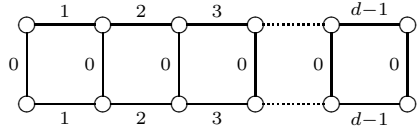
\xymatrix@-1.3pc{*+[o][F]{} \ar@{-}[rr]^1 \ar@{-}[dd]_0 &&
 *+[o][F]{}\ar@{-}[rr]^2\ar@{-}[dd]_0 &&
 *+[o][F]{} \ar@{-}[rr]^3\ar@{-}[dd]_0 &&
 *+[o][F]{} \ar@{.}[rr]\ar@{-}[dd]_0 &&
 *+[o][F]{}\ar@{-}[rr]^{d-1}\ar@{-}[dd]_0 && 
  *+[o][F]{}\ar@{-}[dd]_0 \\ 
&& && && && &&\\
*+[o][F]{} \ar@{-}[rr]_1 &&
 *+[o][F]{}\ar@{-}[rr]_2&&
 *+[o][F]{} \ar@{-}[rr]_3&&
 *+[o][F]{} \ar@{.}[rr]&&
 *+[o][F]{}\ar@{-}[rr]_{d-1}&& 
  *+[o][F]{}\\ }
  $$
 Thus $\Phi$ has a disconnected diagram.
\end{proof}

Using induction over $n$ and  Proposition~\ref{RI} we get the following result.
\begin{prop}\label{induction} 
Suppose that $\Phi$ is string C-group generated by involutions of rank $d$, with connected diagram, having all maximal parabolic subgroups intransitive. If  $\Phi$ is a transitive even group of degree $m$ with $12\leq m< n$  then,  $d\leq \frac{m-1}{2}$.  Moreover if $d= \frac{m-1}{2}$ then $\Phi$ is the alternating group $A_m$.
\end{prop}
\begin{proof}
If $\Phi$  is primitive and not the alternating group then, by Proposition~\ref{prim},  $d\leq \frac{m-2}{2}$.  If  $\Phi$ is imprimitive, then it is embedded into a group $S_a\wr S_b$ with $ab=m$ and with the block action being generated by at most $b-1$ elements.
As the maximal parabolic subgroups of $\Phi$ are intransitive we may use Proposition~\ref{RI} to get 
 \[d\leq \frac{m-2}{2}.\]
 for $a,b\neq 2$. If either $a=2$ or $b=2$ then, as  $\Phi$ is even and has a connected diagram, $d< \frac{m}{2}$ with $m$ even, hence $d\leq \frac{m}{2}$.
Finally if $\Phi\cong A_m$ with $12\leq m< n$, we then get the result by induction on $n$.

\end{proof}

\begin{prop}\label{small(1)} 
Suppose that $\Gamma_{<3}$ is not one of the string $C$-groups (2) or (3) Table~\ref{small}.
If $A$ is the string $C$-group (1) of Table~\ref{small}, then $r \leq \frac{n-1}{2}$. The same result holds for $B$.
\end{prop}
\begin{proof}
Consider first  $n-5 < 12$. If $r-3\geq \frac{(n-5)-1}{2}$ then, by Proposition~\ref{small123}, $B\,(=\Gamma_{>2})$ is one of the small examples (1), (2) or (3) of Table~\ref{small}, thus $n<12$, a contradiction. Hence $r-3<\frac{(n-5)-1}{2}$. Let $n-5 \geq 12$.  

Suppose first that $\rho_2$ has a trivial action on the first $\Gamma_i$- orbit. In this case $\Gamma_{>1}$ is an even transitive group on $n-4$ points.
By Proposition~\ref{sesqui} $\Gamma_1 \cong \langle \rho_0\rangle \times \Gamma_{>1}$. Now, by Proposition~\ref{induction}, we get that $r-2 \leq \frac{(n-4)-1}{2}$ and hence $r\leq \frac{n-1}{2}$.

Now suppose that  $\rho_2$ has a nontrivial action on the first block. As $\Gamma_{<3}$ cannot be one of the string C-groups (2) or (3) of Table~\ref{small}, it is a sesqui extension of it (with respecto to $\rho_2$). 
If $B$ is not the alternating group, then by Proposition~\ref{induction}, we get that $r-3 \leq \frac{(n-5)-2}{2}$. We need only to consider the case $B=\Gamma_{>2}\cong A_{n_2}$.
Let  $\Phi$ be either the string  C-group (2)  or (3)  of Table~\ref{small} and $\tau$ be the action of $\rho_2$ in the second $\Gamma_2$-orbit. As either $\tau=(\rho_1\rho_2)^3$ or $\tau=(\alpha_1\alpha_2)^5$ (according to each case  (1) or (2)), $\Gamma_{<3}$  is isomorphic to $\langle \tau\rangle\times \Phi$. But then, as $\Gamma_{>2}\cong A_{n_2}$, $\tau\in \Gamma_{>2}\cap  \Gamma_{<3}$, a contradiction.
Hence $r \leq \frac{n-1}{2}$.
\end{proof}

\begin{prop}\label{inot0} 
Suppose neither $\Gamma_{<3}$ nor $\Gamma_{>r-4}$  is one of the string $C$-groups (2) or (3) of Table~\ref{small}, or their duals.
If $A$ and $B$ are not both the alternating groups, then $r \leq \frac{n-1}{2}$.
\end{prop}
\begin{proof}
By Proposition~\ref{small(1)} it may be assumed that neither $A$ nor $B$ is the string $C$-group (1) of Table~\ref{small}.

If $A$ and $B$ are not  alternating groups then by Propositions~\ref{induction} and \ref{smalln}, $i  \leq  \frac{n_1-2}{2}$ and $r-1-i  \leq  \frac{n_2-2}{2}$. If $A$ is the alternating group, then either  $n\leq 12$ and $i\leq \frac{n_1-2}{2}$ (by Proposition~\ref{smalln}) or, $n_1\geq 12$ and then, by Proposition~\ref{induction} (induction) we conclude that $i\leq \frac{n_1-1}{2}$. Analogously  if $B$ is the alternating group then  $r-1-i\leq \frac{n_2-1}{2}$ . In any case if $A$ and $B$ are not both the alternating groups, $r \leq  1 + \frac{n_1+n_2-3}{2}= \frac{n-1}{2}$.
\end{proof}

\begin{prop}\label{bothAn} 
Suppose neither $\Gamma_{<3}$ nor $\Gamma_{>r-4}$  is one of the string $C$-groups (2) or (3) of Table~\ref{small}, or their duals.
If $A$ and $B$ are both alternating groups then $r \leq \frac{n-1}{2}$.
\end{prop}
\begin{proof}
In this case $\Gamma_{<i+1}$ is a sesqui extension of a sggi  $\Phi$ with respect to $\rho_i$, where $\Phi$ is a group of degree $n_1+1$.
By Proposition~\ref{sesqui} $\Gamma_{<i+1}$ is isomorphic either to $2\times \Phi$ or $\Phi$. 
Suppose that  $\Gamma_{<i+1}\cong 2\times \Phi$. In that case there is an even permutation $\tau$ on the second $\Gamma_i$-orbit that belongs to $\Gamma_{<i+1}$.
As $\Gamma_{>i}\cong A_{n_2}$, $\tau\in \Gamma_{>i}$ and therefore $\tau\in \Gamma_{<i+1}\cap \Gamma_{>i}$, a contradiction.
Hence $\Gamma_{<i+1}\cong \Phi$ and $\Phi$ is itself a string $C$-group. 
Using the same argument $\Gamma_{>i-1}$ is also isomorphic to a transitive group $\Psi$ of degree $n_2+1$.
Moreover either $\Phi$ or $\Psi$ is a even group. Suppose $\Phi$ is even. 
Then, as $A$ is not one of the string C-groups (1), (2) or (3), either $n_1<12$ and  $i\leq \frac{n_1-2}{2}$, or $n_1+1\geq 12$.
In latest case, by Proposition~\ref{induction} $i+1\leq \frac{(n_1+1)-1}{2}$.
In addition, $r-1-i\leq \frac{n_2-1}{2}$, hence $r \leq \frac{n-1}{2}$.
\end{proof}

\begin{prop}\label{Case2} 
Suppose neither $\Gamma_{<3}$ nor $\Gamma_{>r-4}$  is one of the string $C$-groups (2) or (3) of Table~\ref{small}, or their duals.
Let $i\notin\{0,r-1\}$. If $A\cong \Gamma_{<i}$ and $B=\Gamma_{>i}$ then $r \leq \frac{n-1}{2}$.
\end{prop}
\begin{proof}
This is a consequence of Propositions~\ref{inot0} and \ref{bothAn}.
\end{proof}

To complete this case  we still need to deal with $\Gamma_{<3}$, $\Gamma_{>r-4}$, or both, being one of the string $C$-groups (2) or (3) of Table~\ref{small}. 
This is included in Case 3, at the end.
\subsection{Case 3: The remaining cases}
Assume in this case that $J_B$ is either empty or an interval. As before let $\mathcal{G}$ be the permutation representation graph of $\Gamma$.

\begin{prop}\label{path}
If $e$ is an $f$-edge of $\mathcal{G}$ not in an alternating square, then any path (not containing another $f$-edge) from $e$ to an edge with label $l$, with $l<f$ (resp. $l>f$), contains all labels between $l$ and $f$. 
Moreover, there exists a path from $e$ to an $l$-edge, that is fixed by $\Gamma_{>l}$  (resp. $\Gamma_{<l}$).
\end{prop}
\begin{proof}
Consider a path starting in $e$ and containing an $l$-edge. Let $f<l$.
Suppose that none of the edges of the path has label $k$, for some $f<k<l$.
Then in this path there is an edge with label $<k$ meeting an edge with label $u>k$.
Suppose that this is the first time in the path that this happens.
Then there is an alternating square, containing $e$ and a $u$-edge, a contradiction, as shown in the following figure.
$$ 
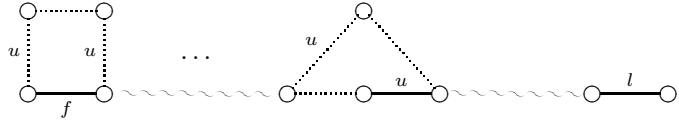
\xymatrix@-1.3pc{*+[o][F]{}\ar@{.}[rr] &&*+[o][F]{}&&&&&& *+[o][F]{}  \ar@{.}[ddll]_u\ar@{.}[ddrr]&&&&&&&&\\
&&&& \ldots&&&&&&\\
*+[o][F]{} \ar@{-}[rr]_f \ar@{.}[uu]^u&&*+[o][F]{} \ar@{.}[uu]^u\ar@{~}[rrrr]&&&&*+[o][F]{}\ar@{.}[rr]&&*+[o][F]{} \ar@{-}[rr]^u && *+[o][F]{} \ar@{~}[rrrr] &&&&*+[o][F]{}\ar@{-}[rr]^l&&*+[o][F]{}
 }
  $$
\end{proof}

\begin{prop}\label{i=0andh=r-1}
Let $i=0$.
If  $\rho_{r-1}$ acts non-trivially on both $\Gamma_i$-orbits, then $r\leq  \frac{n-1}{2}$.
\end{prop}
\begin{proof}
By Proposition~\ref{path} there are two paths, one in the first and the other in the second $\Gamma_i$-orbit, each containing  all labels from $1$ to $r-1$.
Thus $2(r-1)+1\leq n-1$. Hence $r\leq\frac{n}{2}$. Suppose we have the equality. In that case the paths have consecutive labels, as in the following figure and $n$ is precisely the number of vertices of the two paths.
$$\xymatrix@-1.3pc{*+[o][F]{}\ar@{-}[rr]^{r-1} && *+[o][F]{} \ar@{-}[rr]^{r-2}&&*+[o][F]{} \ar@{.}[rr] && *+[o][F]{}\ar@{-}[rr]^1 && *+[o][F]{}\ar@{-}[rr]^0 &&*+[o][F]{} \ar@{-}[rr]^1&& *+[o][F]{}\ar@{.}[rr]&& *+[o][F]{} \ar@{-}[rr]^{r-2}&& *+[o][F]{} \ar@{-}[rr]^{r-1}&&*+[o][F]{}  }
$$
But then there is no place for extra $i$-edges unless $r=3$ (which  gives the bound trivially as $n\geq 12$).
\end{proof}

\begin{prop}\label{i=0andh<>r-1}
Let $i=0$. Let $h\neq r-1$ be the maximal label such that $\rho_h$ acts non-trivially on both $\Gamma_i$-orbits.
There exists a set of vertices $X$, contained in the $\Gamma_i$-orbit fixed by $\rho_{r-1}$, such that $h\leq \frac{n-|X|-1}{2}$ and $\Gamma_{>h}$ fixes $\{1,\ldots, n\}\setminus X$. Moreover if $h= \frac{n-|X|-1}{2}$ then $\Gamma_{<h}$ has the following permutation representation graph, where the black dots represent the vertices of $\{1,\ldots, n\}\setminus X$. 
$$\xymatrix@-1.3pc{\bullet\ar@{-}[rr]^h && \bullet \ar@{-}[rr]^{h-1}&&\bullet \ar@{.}[rr] && \bullet\ar@{-}[rr]^1 && \bullet\ar@{-}[rr]^0 &&\bullet \ar@{-}[rr]^1&& \bullet \ar@{.}[rr]&& \bullet \ar@{-}[rr]^{h-1}&&  \bullet \ar@{-}[rr]^{h}&&*+[o][F]{}  }
$$
\end{prop}
\begin{proof}
By Proposition~\ref{path} there exist a path $\mathcal{P}_1$ from the $h$-edge in the first $\Gamma_i$-orbit to the vertex $a$, and a path $\mathcal{P}_2$ from the $h$-edge in the second $\Gamma_i$-orbit to the vertex $b$, each of them containing all labels from $1$ to $h-1$, and fixed by $\Gamma_{>h}$. 
In that case, the two paths give us $2h+1 \leq |\mathrm{Fix}(\Gamma_{>h})|=n-|X|$ with $X := \{1, \ldots, n\}\setminus \mathrm{Fix}(\Gamma_{>h})$ as required. When equality holds, the permutation representation graph is the one given in this proposition.
\end{proof}

\begin{prop}\label{<jon>jtran}
Let $X:= \{1,\ldots,n\} \setminus \mathrm{Fix}(\Gamma_{>j})$.
If $\Gamma_{>j}$ is transitive on $X$ and there exists a permutation $\rho_l$ with $l<j$ acting non trivially on $X$, then   $r-1-j\leq \frac{|X|}{2}-1$. Moreover if $r-1-j= \frac{|X|}{2}-1$ then,  $\Gamma_{>j}$ has one of the following permutation representation graphs for some $k\in\{j+1,\ldots, r-1\}$.
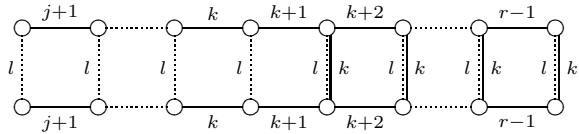
$$ \xymatrix@-1.3pc{*+[o][F]{} \ar@{-}[rr]^{j+1} \ar@{.}[dd]_l &&
 *+[o][F]{}\ar@{.}[rr] \ar@{.}[dd]_l&&
 *+[o][F]{} \ar@{-}[rr]^k\ar@{.}[dd]_l&&
 *+[o][F]{} \ar@{-}[rr]^{k+1}\ar@{.}[dd]_l&&
 *+[o][F]{}\ar@{-}[rr]^{k+2}\ar@{.}[dd]_l\ar@<.3ex>@{-}[dd]^k && 
 *+[o][F]{}\ar@{.}[dd]_l\ar@<.3ex>@{-}[dd]^k  \ar@{.}[rr] &&
 *+[o][F]{} \ar@{-}[rr]^{r-1} \ar@{.}[dd]_l \ar@<.3ex>@{-}[dd]^k && 
 *+[o][F]{}\ar@{.}[dd]_l\ar@<.3ex>@{-}[dd]^k \\
&& && && && && && &&\\
*+[o][F]{} \ar@{-}[rr]_{j+1}  && *+[o][F]{}\ar@{.}[rr] &&*+[o][F]{} \ar@{-}[rr]_k && *+[o][F]{}\ar@{-}[rr]_{k+1}&& *+[o][F]{}\ar@{-}[rr]_{k+2} &&*+[o][F]{} \ar@{.}[rr] &&*+[o][F]{}\ar@{-}[rr]_{r-1}  && *+[o][F]{}}
  
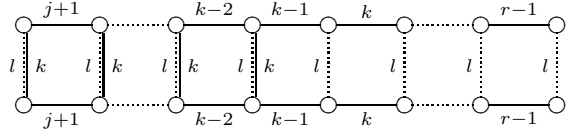
$$
  $$ \xymatrix@-1.3pc{*+[o][F]{} \ar@{-}[rr]^{j+1} \ar@{.}[dd]_l \ar@<.3ex>@{-}[dd]^k &&
 *+[o][F]{}\ar@{.}[rr] \ar@{.}[dd]_l\ar@<.3ex>@{-}[dd]^k &&
 *+[o][F]{} \ar@{-}[rr]^{k-2}\ar@{.}[dd]_l\ar@<.3ex>@{-}[dd]^k &&
 *+[o][F]{} \ar@{-}[rr]^{k-1}\ar@{.}[dd]_l\ar@<.3ex>@{-}[dd]^k&&
 *+[o][F]{}\ar@{-}[rr]^{k}\ar@{.}[dd]_l && 
 *+[o][F]{}\ar@{.}[dd]_l \ar@{.}[rr] &&
 *+[o][F]{} \ar@{-}[rr]^{r-1} \ar@{.}[dd]_l  && 
 *+[o][F]{}\ar@{.}[dd]_l\\
&& && && && && && &&\\
*+[o][F]{} \ar@{-}[rr]_{j+1}  && *+[o][F]{}\ar@{.}[rr] &&*+[o][F]{} \ar@{-}[rr]_{k-2} && *+[o][F]{}\ar@{-}[rr]_{k-1}&& *+[o][F]{}\ar@{-}[rr]_k &&*+[o][F]{} \ar@{.}[rr] &&*+[o][F]{}\ar@{-}[rr]_{r-1}  && *+[o][F]{}}
  $$
\end{prop}
\begin{proof}
If $\rho_l$ acts non trivially on $X$ then $\Gamma_{>j}$ must be imprimitive with blocks of size two, with $\rho_l$ swapping all pairs of vertices inside the blocks. 
Let $M$ an independent generating set for the block action of $\Gamma_{>j}$ and suppose that the group generated by $M$ is intransitive on $X$. Then there must be two orbits on $X$ under the action of $\langle M \rangle$ and some $k>j$ such that $\rho_k$ swaps points of distinct orbits. But then, as $\rho_l \in \Gamma_k$,  $\Gamma_k$ is transitive on $X$.
Moreover, $\rho_k$ only moves points in $X$ which is the union of two $\Gamma_k$-orbits. These two orbits are already fused by $\rho_l\in\Gamma_k$. Therefore, $\Gamma_k$ has to be transitive on $\{1,\ldots, n\}$, a contradiction.
So the action of $\langle M\rangle$ must be transitive on $X$.
If  $|M|=\frac{|X|}{2}-1$ there are only the given possibilities for $\Gamma_{>j}$. 
\end{proof}

In what follows we suppose that $i > 0$ and that all generators acting on the second  $\Gamma_i$-orbit have labels $> i$. As $J_A$ is not an interval, $\Gamma_{<i}$ can either transitive or intransitive on the first $\Gamma_i$-orbit. We consider these cases separately.

\begin{prop}\label{<itran}
Let $r>\frac{n-1}{2}$. Let $h>i$ be the maximal label of  a permutation acting non-trivially on both $\Gamma_i$-orbits.
If $\Gamma_{<i}$ is transitive on the first $\Gamma_i$-orbit then $h=i+1$, $h<r-1$ and there exists a set of vertices $X$, contained in the second $\Gamma_i$-orbit, such that $h\leq \frac{n-|X|-1}{2}$ and $\Gamma_{>h}$ fixes $\{1,\ldots, n\}\setminus X$. 

Moreover if $h= \frac{n-|X|-1}{2}$ then $\Gamma_{<h}$ (with $h=i+1$)  has the following permutation representation graph for some $k\in\{2,\ldots,i-1\}$ where the black dots represent the vertices of $\{1,\ldots, n\}\setminus X$. 
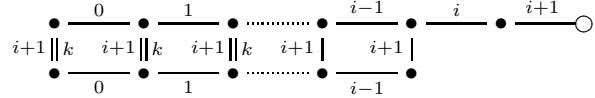
 $$
\xymatrix@-1.3pc{\bullet\ar@{-}[rr]^0 \ar@{=}[d]_{i+1}^k && \bullet \ar@{-}[rr]^1\ar@{=}[d]_{i+1}^k&&\bullet\ar@{=}[d]_{i+1}^k\ar@{.}[rr] && \bullet\ar@{-}[rr]^{i-1}\ar@{-}[d]_{i+1} && \bullet\ar@{-}[rr]^i \ar@{-}[d]_{i+1} &&\bullet \ar@{-}[rr]^{i+1}&&*+[o][F]{} \\
 \bullet\ar@{-}[rr]_0 && \bullet \ar@{-}[rr]_1&&\bullet \ar@{.}[rr] && \bullet\ar@{-}[rr]_{i-1}&& \bullet&&&&
 }
  $$
\end{prop}
\begin{proof}
In this case, $\Gamma_{<i}$ is transitive on the first $\Gamma_i$-orbit $O_1$, which has size $n_1$. Moreover there exists $h>i$ such that $\rho_h$ acts non-trivially on $O_1$.
 This action is fixed-point-free and hence $\rho_h$ moves $a$ and $n_1$ is even. But since the $i$-edge $\{a,b\}$ is not in a square by hypothesis, $h=i+1$ and $\Gamma_{>h}$ acts trivially on $O_1$.
Moreover by the dual of Proposition~\ref{<jon>jtran} we have that $i\leq \frac{n_1}{2}-1$.
Thus $h=i+1\leq \frac{n_1}{2}$. 
As $J_B$ is an interval, $\rho_{i+1}$ is the unique permutation acting nontrivially on $b$. If $h\neq r-1$,  $\Gamma_{>h}$ fixes $O_1$ and  a vertex in the second $\Gamma_i$-orbit, so at least $n_1+1$ vertices altogether.
Then $h\leq \frac{|\mathrm{Fix}(\Gamma_{>h})|-1}{2}\leq \frac{n-|X|-1}{2}$. The equality cannot occur as it would imply that $n-|X| = n_1+1$ is even, a contradiction. So $h\leq \frac{n-|X|-1}{2}$ as wanted.
Moreover, when the equality holds, by the dual of Proposition~\ref{<jon>jtran}, the only possibility is the permutation representation graph given in the statement of this proposition. 

Now suppose that  $h=r-1$. Then $r-1\leq \frac{n_1}{2}$. As $n\geq n_1+2$, $r\leq \frac{n}{2}$. As by hypothesis $r> \frac{n-1}{2}$, we have the equality $r= \frac{n}{2}$. Then $\Gamma$ has the permutation  representation graph given in the statement of this proposition with $h=i+1=r-1$, with some additional $i$-edges. As all $\Gamma_j$'s must be intransitive, the extra $i$-edges must be vertical edges. Hence one of $\rho_i$ and $\rho_{i+1}$ has to be an odd permutation, a contradiction.
\end{proof}


\begin{prop}\label{<iint}
Let $r>\frac{n-1}{2}$. Let $h>i$ be the maximal label of  a permutation acting non-trivially on both $\Gamma_i$-orbits.
If $\Gamma_{<i}$ is intransitive in the first $\Gamma_i$-orbit, then $h<r-1$ and there exists a set of vertices $X$, contained in the second $\Gamma_i$-orbit,  such that $h\leq \frac{n-|X|-1}{2}$ and $\Gamma_{>h}$ fixes $\{1,\ldots, n\}\setminus X$. 

Moreover if $h= \frac{n-|X|-1}{2}$ then $h=i+1$ and $\Gamma_{<h+1}$ has the following permutation representation graph, where the black dots represent the vertices of $X$.
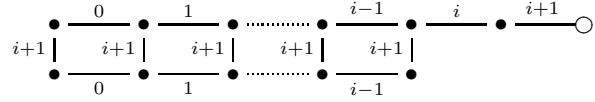
$$\xymatrix@-1.3pc{\bullet\ar@{-}[rr]^0 \ar@{-}[d]_{i+1} && \bullet \ar@{-}[rr]^1\ar@{-}[d]_{i+1}&&\bullet\ar@{-}[d]_{i+1} \ar@{.}[rr] && \bullet\ar@{-}[rr]^{i-1}\ar@{-}[d]_{i+1} && \bullet\ar@{-}[rr]^i \ar@{-}[d]_{i+1} &&\bullet \ar@{-}[rr]^{i+1}&& *+[o][F]{} \\
 \bullet\ar@{-}[rr]_0 && \bullet \ar@{-}[rr]_1&&\bullet \ar@{.}[rr] && \bullet\ar@{-}[rr]_{i-1}&& \bullet&&\\ 
 }
  $$
\end{prop}
\begin{proof}
In this case $J_A$ is not an interval thus, by Proposition~\ref{CCD},  $A$ is  imprimitive embedded into  $S_k\wr S_m$ and $\Gamma_{<i}$ is fixing all the blocks (with $k,m\geq 2$).
By Proposition~\ref{path} there exist a path $\mathcal{P}_1$ from the $h$-edge in the first $\Gamma_i$-orbit to the vertex $a$, and a path $\mathcal{P}_2$ from the $h$-edge in the second $\Gamma_i$-orbit to the vertex $b$, each of them containing all labels from $i+1$ to $h-1$, and fixed by $\Gamma_{>h}$. 
In addition there is a path $\mathcal{P}_3$ in the block $\beta$ containing the vertex $a$, and edges with all labels from $0$ to $i-1$.
Moreover, there is also a path $\mathcal{P}_4$ in the block $\beta\rho_{i+1}$ also containing edges with all labels from $0$ to $i-1$. 

Let us first assume that there is an edge with label $l>i$ inside one of the two blocks $\beta$ or $\beta\rho_{i+1}$. 
Then the generator $\rho_l$ is fixed-point-free on the block containing that edge.
Suppose that block is $\beta$. Then $l = i+1$, a contradiction. Suppose then that the edge is in $\beta\rho_{i+1}$. Then $l=i+2 = h$ and there are only two blocks in the block system. 
Moreover, $\rho_l$ acts inside $\beta\rho_{i+1}$ and provides an embedding of $\Gamma_{<i}$ into $S_2\wr S_{k/2}$.
Thus, since $\Gamma_j$ is intransitive for every $j$, a similar argument  to the one used in the proof of Proposition~\ref{<jon>jtran} shows that $i \leq k/2-1$.
Hence $h-2 \leq k/2-1$.
Now if $h=r-1$, as $n\geq 2k+2$ (the extra 2 coming from the $h$-edge in the second $\Gamma_i$-orbit),
we get $r-1\leq \frac{n+2}{4} < \frac{n-2}{2}$ (as $n$ must be at least 7 in this case)
giving a contradiction with the hypotheses of the proposition.
Thus  $h< r-1$ and  $\Gamma_{>h}$ fixes a set $P$ of size $2k+2$.
Therefore, $h\leq \frac{|P|+2}{4} \leq \frac{|P|-1}{2}$ for every $|P| \geq 4$ which is obviously the case here as $k\geq 2$.  

Let us now assume that all edges with label $l>i$ are not contained in $\beta$ nor in $\beta\rho_{i+1}$.
Then the paths  $\mathcal{P}_1$, $\mathcal{P}_2$, $\mathcal{P}_3$ and $\mathcal{P}_4$ have no edge in common, as shown in the following figure.
$$
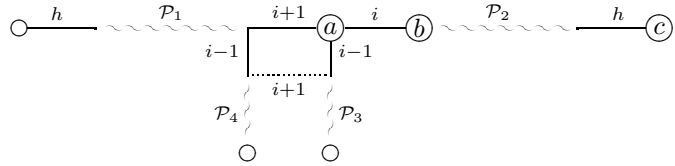
\xymatrix@-1.3pc{  *+[o][F]{} \ar@{-}[rr]^h&&*+[.][F]{} \ar@{~}[rrrr]^{\mathcal{P}_1}   &&&&*+[.][F]{}\ar@{-}[d]_{i-1}\ar@{-}[rr]^{i+1} && *+[o][F]{a}\ar@{-}[d]^{i-1} \ar@{-}[rr]^i&& *+[o][F]{b}\ar@{~}[rrrr]^{\mathcal{P}_2}&&&& *+[.][F]{}\ar@{-}[rr]^h&& *+[o][F]{c}\\
  &&&&&&*+[.][F]{}\ar@{~}[dd]_{\mathcal{P}_4}\ar@{.}[rr]_{i+1}&&*+[.][F]{}\ar@{~}[dd]^{\mathcal{P}_3}&&&&&&\\
  &&&&&&&&&&&&&&\\
  &&&&&&*+[o][F]{}&& *+[o][F]{} &&&&&&}
 $$

Suppose that $h\neq r-1$ and that  $\rho_{h+1}$ acts trivially on the first $\Gamma_i$-orbit.
Let $P$ be the set of vertices of $\mathcal{P}_1\cup\mathcal{P}_2\cup\mathcal{P}_3\cup \mathcal{P}_4$  excluding the vertex $c$ on the right hand side of the diagram above, and $X=\{1,\ldots,n\}\setminus P$. 
The set $P$ is fixed by $\Gamma_{>h}$ and $2h+1\leq |P|=n-|X|$. 
Moreover when $2h+1= |P|$ the  permutation representation of $\Gamma_{<h}$ is the one given in this proposition with $h=i+1$.  

Now if  $h=r-1$, we have  $2(r-1)+1\leq n-1$. Hence $r\leq \frac{n}{2}$, and by hypotheses we must have the equality. 
Therefore $\Gamma$ has precisely the  permutation representation graph given in the statement of this proposition with some additional $i$-edges. 
As all $\Gamma_j$'s must be intransitive, the extra $i$-edges must be vertical. Hence one of $\rho_i$ and $\rho_{i+1}$ has to be an odd permutation, a contradiction. 

It remains to prove that $\Gamma_{>h}$ fixes the first $\Gamma_i$-orbit. Suppose the contrary. Then $\rho_{r-1}$ acts nontrivially on the first $\Gamma_i$-orbit.
 In this case consider the path $\mathcal{P}_1$ as before, with the label of its first edge being $r-1$ instead of $h$, and consider $\mathcal{P}_3$ and $\mathcal{P}_4$ has before.
 Now let $\mathcal{P}_2$ be a copy of $\mathcal{P}_1$ whose  last vertex is $a\rho_{i-1}$, as in the following figure.
 $$
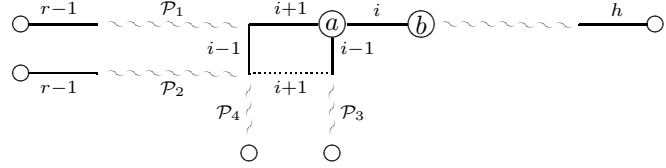
\xymatrix@-1.3pc{  *+[o][F]{} \ar@{-}[rr]^{r-1}&&*+[.][F]{} \ar@{~}[rrrr]^{\mathcal{P}_1}   &&&&*+[.][F]{}\ar@{-}[d]_{i-1}\ar@{-}[rr]^{i+1} && *+[o][F]{a}\ar@{-}[d]^{i-1} \ar@{-}[rr]^i&& *+[o][F]{b}\ar@{~}[rrrr]&&&& *+[.][F]{}\ar@{-}[rr]^h&& *+[o][F]{}\\
   *+[o][F]{} \ar@{-}[rr]_{r-1}&&*+[.][F]{} \ar@{~}[rrrr]_{\mathcal{P}_2}   &&&&*+[.][F]{}\ar@{~}[dd]_{\mathcal{P}_4}\ar@{.}[rr]_{i+1}&&*+[.][F]{}\ar@{~}[dd]^{\mathcal{P}_3}&&&&&&\\
  &&&&&&&&&&&&&&\\
  &&&&&&*+[o][F]{}&& *+[o][F]{} &&&&&&}
 $$
 In this case $2(r-1)+1\leq (n-1)-1$, a contradiction.
\end{proof}

\begin{prop}\label{h}
Let $r>\frac{n-1}{2}$. Let $i\neq 0$ and $J_B$ be an interval with labels $>i$.
If  $h>i$ is the maximal label of  a permutation acting non-trivially on both $\Gamma_i$-orbits, then $h< r-1$ and  there exists a set of vertices $X$, contained in the second $\Gamma_i$-orbit, such that $h\leq \frac{n-|X|-1}{2}$ and $\Gamma_{>h}$ acts trivially on $\{1,\ldots, n\}\setminus X$.
\end{prop}
\begin{proof}
This is consequence of Propositions~\ref{<itran} and \ref{<iint}.
\end{proof}

\begin{prop}\label{ineq0Btrivial}
Let $n\geq 8$.
If  $i\neq 0$ and $J_B=\emptyset$ then $r\leq\frac{n-1}{2}$.
\end{prop}
\begin{proof}
Similarly to the proof of  Proposition~\ref{<iint}, $A$ is embedded into a wreath product $S_k\wr S_m$ with $n=km+1$ and $\Gamma_{<i}$ fixing the blocks.
If there is a label $l>i$ such that $\rho_l$ acts nontrivially inside a block then $m=2$, $r-1=l=i+2$ and  $\Gamma_{<i}$ is embedded into $S_2\wr S_{k/2}$, giving the inequality $i=r-3\leq k/2-1\leq \frac{n-1}{4}-1$. Hence for $n\geq 8$, $r\leq\frac{n-1}{2}$. 

Suppose then that
every generator with label $>i$ acts trivially on the blocks it fixes.
Thus there are four paths  $\mathcal{P}_1$, $\mathcal{P}_2$, $\mathcal{P}_3$ and $\mathcal{P}_4$, as in the following graph, containing all but one label twice and one cycle that is an alternating square. 
$$
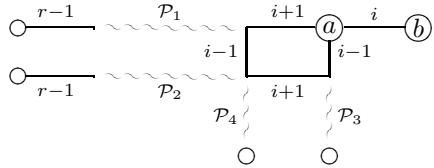
\xymatrix@-1.3pc{  *+[o][F]{} \ar@{-}[rr]^{r-1}&&*+[.][F]{} \ar@{~}[rrrr]^{\mathcal{P}_1}   &&&&*+[.][F]{}\ar@{-}[d]_{i-1}\ar@{-}[rr]^{i+1} && *+[o][F]{a}\ar@{-}[d]^{i-1} \ar@{-}[rr]^i&& *+[o][F]{b}\\
   *+[o][F]{} \ar@{-}[rr]_{r-1}&&*+[.][F]{} \ar@{~}[rrrr]_{\mathcal{P}_2}   &&&&*+[.][F]{}\ar@{~}[dd]_{\mathcal{P}_4}\ar@{-}[rr]_{i+1}&&*+[.][F]{}\ar@{~}[dd]^{\mathcal{P}_3}&&\\
  &&&&&&&&&&\\
  &&&&&&*+[o][F]{}&& *+[o][F]{} &&}
 $$
 Hence $2(r-1)+1\leq n$. 
If equality holds, the paths contain exactly one edge for each label, but then as all $\Gamma_j$'s are intransitive,  $\rho_i$ acts trivially on $\{1,\ldots,n\}\setminus \{a,b\}$, and thus is a transposition, a contradiction. Therefore $r\leq \frac{n}{2}$. If $n$ is odd then $r\leq \frac{n-1}{2}$. If $n$ is even then $n_1$ is odd, and neither $k=2$ nor $m=2$. Hence  there at least two more vertices not in the paths, and then $2(r-1)+1\leq n-2$. Therefore $r\leq\frac{n-1}{2}$. 
\end{proof}

\begin{prop}\label{>htran}
Let $n\geq 12$. Suppose that $h>i$ is the maximal label of  a permutation acting non-trivially on both $\Gamma_i$-orbits.
If $h\neq r-1$ and $\Gamma_{>h}$ is transitive on $X:= \{1,\ldots,n\} \setminus \mathrm{Fix}(\Gamma_{>h})$ then, either $r\leq \frac{n-1}{2}$, or one of the groups $\Gamma_{<3}$ or $\Gamma_{>r-4}$  is one the string C-groups (2) and (3) given in Table~\ref{small} or their duals.
\end{prop}
\begin{proof}
Suppose that $\Gamma_{>h}$ is neither the dual of the string C-group (2) nor the string C-group (3) given in Table~\ref{small}.  Then, by Propositions~\ref{small123}, \ref{induction} and \ref{small(1)}, $r-1-h\leq \frac{|X|-1}{2}$. Moreover $r-1-h= \frac{|X|-1}{2}$ if $\Gamma_{>h}$ is the alternating group $A_{|X|}$.

By Proposition~\ref{h},  $h\leq \frac{n-|X|-1}{2}$. Suppose that  $\Gamma_{>h}\cong A_{|X|}$ and $h= \frac{n-|X|-1}{2}$.
If $\rho_i$ has a nontrivial action on $X$ then,  by Proposition~\ref{<jon>jtran}, $\Gamma_{>h}$ must be imprimitive with blocks of size two, a contradiction.
Thus we may assume that $\rho_i$ fixes $X$ pointwise and swaps a pair of vertices in $ \{1,\ldots,n\} \setminus X$.

If $i=0$ then $\Gamma_{<h}$ has the permutation representation graph given in Proposition~\ref{i=0andh<>r-1}. The string condition implies that $h=2$ and $\Gamma_{<3}$ must be the string C-group (2) given in Table~\ref{small}, a contradiction.

When $i\neq 0$, either $r\leq \frac{n-1}{2}$ or $\Gamma_{<h}$ has either the permutation representation graph given in Proposition~\ref{<itran} or the one given in Proposition~\ref{<iint}. 
As all $\Gamma_j$'s must be intransitive, the extra $i$-edges must be vertical edges. Hence one of $\rho_i$ and $\rho_{i+1}$ has to be an odd permutation, a contradiction.

Hence either $h<\frac{n-|X|-1}{2}$ or $r-1-h<\frac{|X|-1}{2}$, which implies in both cases that $r\leq \frac{n-1}{2}$.
\end{proof}

In the previous proposition we considered the case where $\rho_h$ acts nontrivially in both $\Gamma_i$-orbits and $\Gamma_{>h}$ is transitive on the points it does not fix.
Let us now deal with the case where $\Gamma_{>h}$ is intransitive.


In what follows we use the results of the previous section on 2-fracture graphs.
Observe that in the proofs  of Propositions~\ref{max1squareforcomp} and \ref{dislc} none of the following three conditions on $\Gamma$ were needed: intersection condition, connectedness of the diagram and the group being even.
Indeed all we need is a string group $G$ generated by a set $\{\delta_j \,|\, j \in \{0,\ldots ,d-1\}\}$ of involutions
 where $\delta_j$ swaps two pairs of vertices in different $G_j$-orbits for each $j \in \{0,\ldots ,d-1\}$.
 
\begin{prop}\label{t}
Let $t\in\{0,\ldots,r-2\}$ and $U:=\{1,\ldots, n\}\setminus {\rm Fix}(\Gamma_{>t})$. If $t$ is such that
\begin{itemize}
\item  $t\leq \frac{n-|U|-1}{2}$,
\item $\Gamma_{>t}$ has a 2-fracture graph,
\item $\Gamma_{>t}$  acts intransitively on $U$,
\end{itemize}
then $r\leq \frac{n-1}{2}$.
\end{prop}
\begin{proof}
Assume first that $t=r-2$. In this case, $|U| \geq 4$ for $\rho_{r-2}$ to be an even permutation. Hence $r\leq \frac{n-1}{2}$.
Let $t<r-2$ be the maximal label satisfying the conditions of this proposition.

Suppose that $\Gamma_{>t}$ has $c$ nontrivial orbits  $U_1,\ldots, U_c$.
For each set $U_s$ with $s\in\{1,\ldots, c\}$ denote by $I_s$ the set of labels 
$l\,(>t)$ of edges in $U_s$. 
Consider the graph $\mathcal{C}$ 
whose vertices are the orbits $U_1, \ldots, U_c$, and two orbits $U_s$ and $U_q$ are joined by an $l$-edge if there exist a point in $U_s$ and a point in $U_q$ that are swapped by $\rho_l$.

Consider a (simple) fracture graph $\mathcal{F}$ of $\Phi:=\Gamma_{>t}$, that is, a graph with $|U|$ vertices and  $r-1-t$ edges. Such a graph exists by the second hypothesis of the proposition.
Each $l$-edge of $\mathcal{F}$ connects vertices in different $\Phi_l$-orbits. Let $s\in\{1,\ldots,c\}$ and $F_s$ be the set of labels of edges of $\mathcal{F}$ within $U_s$.
Clearly $F_s\subseteq I_s$.
Choose $\mathcal{F}$ such that it satisfies the following property:
\begin{enumerate}
 \item[P1] if $l\in F_s$ is the label of the unique $l$-edge in one component swapping vertices in different $\Gamma_l$-orbits, then no other component has more than one pair of vertices in different $\Gamma_l$-orbits.
\end{enumerate}
Let $G_s$ be the group action of $\Gamma_{>t}$ in $U_s$. We have that $G_s$ is generated by a set of involutions (not necessarily independent) with labels in $I_s$. The subset of involutions with labels in $F_s$ is independent since $F_s$ is the subset of labels of edges of $\mathcal{F}$.
We denote by $(G_s)_j$ the group generated by all  involutions of the generating set of $G_s$ except the one with label $j$.

If $G_s$ does not admit a 2-fracture graph with set of labels $F_s$, then there exists an edge $e$ with label $l\in F_s$ with $\rho_l$ swapping only one pair of vertices  of $U_s$, in different $\Gamma_l$-orbits. Let $m$ be the minimal label and $x$ be the maximal label of an edge inside $U_s$. 
By Proposition~\ref{path} there exists a path $\mathcal{P}_1$ containing all labels from $l-1$ to $m$ and another path $\mathcal{P}_2$ containing all labels from $l+1$ $x$ in $U_s$. Let $\mathcal{P}=\mathcal{P}_1\cup
\{e\} \cup\mathcal{P}_2$. Now we deal separately with the cases $m>t+1$ and $m=t+1$ and we conclude, in both cases, that $x \leq \frac{n-|X|-2}{2}$ where  $X := \{1,\ldots,n\}\setminus \mathrm{Fix}(\Gamma_{>x})$.\\

$\underline{m>t+1}$: If $m>t+1$ then a component $U_q$ adjacent to $U_s$ also contains all labels from $m$ to $x$. 
Let $U_z$ be a component containing an edge with label $m-1$.
We can reach it from $U_s$ with a shortest path in $\mathcal{C}$. 
The last component before $U_z$ in this path, say $U_w$, contains a copy $\mathcal{P}'$ of $\mathcal{P}$.
Thus there is a $t$-edge from $U_w$ to $U_z$.
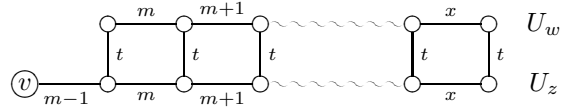
$$\xymatrix@-1.3pc{  &&*+[o][F]{} \ar@{-}[rr]^m\ar@{-}[d]^t&&*+[o][F]{} \ar@{-}[rr]^{m+1}\ar@{-}[d]^t&&*+[o][F]{} \ar@{~}[rrrr]\ar@{-}[d]^t&&&&*+[o][F]{}\ar@{-}[rr]^x\ar@{-}[d]^t&&*+[o][F]{}\ar@{-}[d]^t& U_w\\
 *+[o][F]{v}\ar@{-}[rr]_{m-1}&&*+[o][F]{} \ar@{-}[rr]_m&&*+[o][F]{} \ar@{-}[rr]_{m+1}&&*+[o][F]{} \ar@{~}[rrrr]&&&&*+[o][F]{}\ar@{-}[rr]_x&&*+[o][F]{}&U_z}
 $$
As $m$ is the minimal label in $U_{w}$,   $m-1\notin U_{w}$, thus $m-1=t+1$. Let $P$ be the set of vertices of $\mathcal{P}\cup\mathcal{P}'$. 
We have $2(x-(t+1))\leq |P|-2$, hence $x\leq t+\frac{|P|}{2}\leq \frac{n-|U|+|P|-1}{2}$.
As $\Gamma_{>t}$ has a 2-fracture graph by hypothesis, there is at least one more edge with label $t+1$, so $|U|>|P|+1$ and $x\leq \frac{n-3}{2}$.
If $x=r-1$ then $r\leq\frac{n-1}{2}$. 
Let $x\neq r-1$.  The vertex $v$ is fixed by  $\Gamma_{>x}$ as well as all vertices of $P$. Let $X := \{1,\ldots,n\}\setminus \mathrm{Fix}(\Gamma_{>x})$.
Then we have $x \leq \frac{n-|X|-2}{2}$.\\

$\underline{m=t+1}$: Suppose that $m=t+1$. Furthermore assume that any component containing a unique $l$-edge between vertices in different $\Gamma_l$-orbits, has minimal label $t+1$. 
Let $x$ be the maximal label in $U_s$ and let $\mathcal{P}$ be as before.
We now use the fact that $\Gamma_{>t}$ has a 2-fracture graph, and thus there exists another component $U_q$ having one $l$-edge $e'$ between vertices in different components. Moreover by Property (P1) this $l$-edge cannot be in an alternating square inside $U_q$. By assumption the minimal label in $U_q$ is also $t+1$. 
Let $y$ be the maximal label in $U_q$. 
Assume that $y\geq x$. Then, there exists a path $\mathcal{P}'_1$ containing all labels from $l-1$ to $t+1$ and another path $\mathcal{P}'_2$ containing all labels from $l+1$ to $x$, both in $U_q$. Consider $\mathcal{P}'=\mathcal{P}'_1\cup\{e'\} \cup\mathcal{P}'_2$.
There is at most one vertex of $\mathcal{P}'$ that is not fixed by $\Gamma_{>x}$.
Let $Q$ be the set of vertices of $\mathcal{P}\cup\mathcal{P}'$.
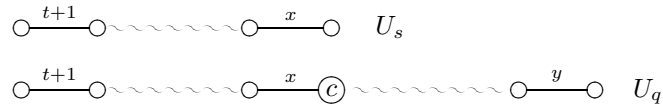
$$\xymatrix@-1.3pc{  *+[o][F]{} \ar@{-}[rr]^{t+1}&&*+[o][F]{} \ar@{~}[rrrr]&&&&*+[o][F]{}\ar@{-}[rr]^x&&*+[o][F]{}&U_s&&&&\\
*+[o][F]{} \ar@{-}[rr]^{t+1}&&*+[o][F]{} \ar@{~}[rrrr]&&&&*+[o][F]{}\ar@{-}[rr]^x&&*+[o][F]{c} \ar@{~}[rrrr]&&&&*+[o][F]{}\ar@{-}[rr]^y&&*+[o][F]{}&U_q}$$
In this case $\Gamma_{>x}$ fixes the set $Q\setminus\{c\}$.
Hence $2(x-t)\leq |Q|-2$ and we get 
$$x\leq \frac{n-|X|-2}{2}$$
where $X := \{1,\ldots,n\}\setminus \mathrm{Fix}(\Gamma_{>x})$.\\

In both cases, thanks to the maximality of $t$ we conclude that either $x=r-1$ or  $\Gamma_{>x}$ is transitive on $X$. If $x=r-1$, $r\leq \frac{n-2}{2}$.
Suppose $\Gamma_{>x}$ is transitive on $X$. In this case $X$ is a subset of  $U_s$ for some $s\in\{1,\ldots,c\}$.
If $\Gamma_{>x}$ is fix-point-free on $U_s$, then $\rho_t$ centralizes $\Gamma_{>x}$, thus, by Proposition~\ref{<jon>jtran}, $r-1-x\leq \frac{|X|}{2}-1$.
Hence $r\leq \frac{n-1}{2}$. If $\Gamma_{>x}$ has fixed points in $U_s$, then $x\neq  \frac{n-|X|-2}{2}$ and, as by Proposition~\ref{induction} $r-1-x\leq \frac{X}{2}$, we get $r\leq \frac{n-1}{2}$.
This concludes the case where $G_s$ has no 2-fracture graph.

Therefore we may assume that $G_s$ admits a 2-fracture graph with set of labels $F_s$ and therefore $F_s\leq \frac{|U_s|}{2}$. We can use the results of Section~\ref{s:2frac}.

Suppose that $\Gamma_{>t+1}$ is transitive on $U_s$.
Take the closest orbit $U_q$ to $U_s$ in $\mathcal{C}$ such that $\Gamma_{> t+1}$ is intransitive on $U_q$ and let $P$ be a shortest path from $U_s$ to $U_q$.
We can use $\rho_t$ modify $\mathcal{F}$ along this path, to concentrate all edges of the fracture graph in $U_q$. We obtain another fracture graph also satisfying P1.
Hence, we may assume that $F_x$ is empty for every orbit $U_x$ of $P$ except $U_q$.
Therefore $|F_x| = 0 \leq \frac{|U_x|-2}{2}$.
Thus in every component with $F_s\neq \emptyset$ there exists an edge with label $t+1$ between vertices in different  $\Gamma_{t+1}$-orbits.
Consequently there is at most one component having a connected 2-fracture graph, and $r-1-t\leq \frac{|U|-(c-1)}{2}$
 where $c$ is the number of $\Gamma_{>t}$-orbits of size at least 2.
 Whenever $c>2$ or if $|F_x| = 0$ for some $x$, we get $r \leq \frac{n-1}{2}$. 
Let us now assume $c=2$. In this case, if $r = \frac{n}{2}$, the 2-fracture graph on $U_s$ is connected and has an alternating square (by Proposition~\ref{max1squareforcomp}), and the 2-fracture graph on $U_q$ is disconnected, and has two components, one having an alternating square and the other being a tree. These two components in $U_q$ must moreover be connected by a $t$-edge.
These components satisfy the following  equalities: $|F_s|=\frac{|U_s|}{2}$, $|F_q|=\frac{|U_q|-1}{2}$, $I_s=F_s$ and $I_q=F_q\cup\{t+1\}$.
Let $G_s$ be the group action in $U_s$ and let $G_q$ be the group action in $U_q$.

Recall that $t < r-2$. If  $(G_q)_{t+2}$ is transitive on $U_q$ then $\rho_{t+1}$ centralizes $G_q$, thus the two components of the 2-fracture graph of $G_q$ have the same shape, so they must both be trees, a contradiction. Thus $(G_q)_{t+2}$ is intransitive on $U_q$. Moreover $t+2\in F_q$ as $I_q=F_q\cup\{t+1\}$. Therefore $t+2\notin F_s=I_s$, hence  $\rho_{t+1}$ centralizes $G_s$. But then there exists an edge with label $l\neq t+1$ in $U_s$ incident to the $t$-edge that connects $U_s$ to $U_q$, thus $l\in I_s\cap I_q$, a contradiction.
\end{proof}


\begin{prop}\label{case3}
Let $n\geq 12$. Suppose that $i$ is the maximal label such that 
\begin{itemize}
\item $\Gamma_i$ has two orbits with a single $i$-edge joining them;
\item there exists $h>i$ such that $\rho_h$ acts non-trivially on both $\Gamma_i$-orbits.
\end{itemize}
If $r> \frac{n-1}{2}$, then one of the following occurs.
\begin{enumerate}
\item $\Gamma_{>r-4}$  is the dual of the string C-group (2), or the string C-group (3) given in Table~\ref{small};
\item $\Gamma_{r-1}$ has exactly two orbits, one being trivial.
\end{enumerate}
\end{prop}
\begin{proof}
Let $h$ be maximal.
Let us choose the vertices $a$ and $b$, and consequently the groups $A$ and $B$,  such that $J_B$ is an interval with all labels in $J_B$ being $>i$. If $J_A$ is also an interval, then, as $h\in J_A\cap J_B$,  $i=0$; we assume without loss of generality that $r-1\in J_A\setminus J_B$ (Recall that the case $h=r-1$ has been dealt with in Proposition~\ref{i=0andh=r-1}).

Denote by $O_1$ and $O_2$ the orbits of $A$ and $B$ respectively.
By Proposition~\ref{i=0andh<>r-1} when $i=0$, and Proposition~\ref{h} when $i\neq 0$, we have that  $h< r-1$ and $h\leq \frac{n-|X|-1}{2}$  with $X:=\{1,\ldots, n\}\setminus {\rm Fix}(\Gamma_{>h})\subseteq O_2$. If $\Gamma_{>h}$ is transitive, by Proposition~\ref{>htran}, one of the groups $\Gamma_{<3}$ or $\Gamma_{>r-4}$  is one of the string C-groups (2) or (3)  given in Table~\ref{small} or their duals. Thus we may assume that  $\Gamma_{>h}$ is intransitive.
Now if $\Gamma_{>h}$ has a 2-fracture graph,  by Proposition~\ref{t},  $r\leq \frac{n-1}{2}$, a contradiction.
Hence $\Gamma_{>h}$ does not admit a 2-fracture graph.
Then there exists $j>h$ such that $\rho_j$ only swaps one pair of vertices in different $\Gamma_j$-orbits. 
Choose $j$ minimal with this property. Then $\Gamma_{\{h+1,\ldots j-1\}}$ has a 2-fracture graph.
Let $C$ and $D$ be the group actions of $\Gamma_j$ in each $\Gamma_j$-orbit with $C$ acting on the orbit $L_1$ containing the $i$-edge $\{a,b\}$. Let $L_2$ denote the other $\Gamma_j$-orbit.
If one of the groups is trivial then either we get case (b) of the statement of this proposition or $r\leq \frac{n-1}{2}$ by the dual of Proposition~\ref{ineq0Btrivial}, a contradiction.
Thus assume both $C$ and $D$ are nontrivial sggi's. Indeed let $\rho_j=\gamma_j\delta_j$ with $\gamma_j$ and $\delta_j$ being the permutations in each $\Gamma_j$-orbit.
Then $C = \langle \gamma_i\, |\, i\in \{0,\ldots, r-1\}\rangle$ and $D := \langle \delta_i \,| \,i\in \{0,\ldots, r-1\}\rangle$.
Let $J_C:= \{  i \in \{0,\ldots, r-1\}\,|\,\gamma_i \neq 1_C\}$ and  $J_D:= \{  i \in \{0,\ldots, r-1\}\,|\,\delta_i \neq 1_D\}$.
By Propositions~\ref{CCD} and \ref{bothimp} either $J_C$ or $J_D$ is an interval. 
If $C=\Gamma_{<j}$ and $D=\Gamma_{>j}$ then by Proposition~\ref{Case2}, $r\leq \frac{n-1}{2}$ or $\Gamma_{>r-4}$  is, up to duality, one of the  string C-groups (2) or (3)  given in Table~\ref{small}. This gives case (a) of the statement.

It remains to consider the case where there exists a permutation $\rho_g$ acting non-trivially on both $\Gamma_j$-orbits. 
Choose $g$ minimal.
Thanks to the maximality of $i$,  $g<j$. 
As $g\in J_D$ we have that $g>i$. We now consider four cases:

(1) If $j=r-1$ and $g=0$, then by the dual of Proposition~\ref{i=0andh=r-1}, we get a contradiction.

(2) If $j=r-1$ and $g\neq 0$, then both $J_C$ and $J_D$ are intervals and $g$ is the minimal label of a permutation acting nontrivially on $L_2$. 
Then we can use the dual of Proposition~\ref{i=0andh<>r-1} to conclude that 
$r-g\leq \frac{n-|Y|-1}{2}$ with $Y:= \{1,\ldots, n\}\setminus {\rm Fix}(\Gamma_{<g})\subseteq L_1$. 
If $g\leq h$ then, by Proposition~\ref{path}, there is a path in $X$ containing all labels from $h$ to $r-2$ twice. Let $V$ be the set of vertices of this path, then $r-h\leq\frac{|V|}{2}$.  Hence $r\leq \frac{n-|X|-|V|-1}{2}$, giving a contradiction.
Hence $g> h$ and $X\cap Y\neq \emptyset$.

(3) Let $j\neq r-1$ and $J_D$  be an interval. In this case both $\rho_0$ and $\rho_{r-1}$ act nontrivially on the first $\Gamma_j$-orbit. Hence we can use the dual argument to the one used in the last paragraph of the proof of Proposition~\ref{<iint}, to conclude that $r\leq \frac{n-1}{2}$, a contradiction.

(4) Let $j\neq r-1$ and $J_C$ be an interval. By the dual of Proposition~\ref{h}, $g<0$ and  $r-g\leq \frac{n-|Y|-1}{2}$ with $Y:= \{1,\ldots, n\}\setminus {\rm Fix}(\Gamma_{<g})\subseteq L_1$. By the same reasoning as in (2), $h>g$.

Only cases (2) and (4) are possible and both give the inequality  $r-g\leq \frac{n-|Y|-1}{2}$ with $Y:= \{1,\ldots, n\}\setminus {\rm Fix}(\Gamma_{<g})\subseteq L_1$ and  $h>g$.

By choice of $j$, there is no other label $l$ between $h$ and $g$ having only one pair of vertices in different $\Gamma_l$-orbits. 
Then $\Gamma_{\{h+1,\ldots, g-1\}}$ has a 2-fracture graph. Moreover $\Gamma_{>h} \cap \Gamma_{<g}$
is intransitive on  $X\cap Y$
for if it were transitive, then $\Gamma_{>h}$ would be transitive on $X$, contradicting the assumptions of the proposition.
Hence the 2-fracture graph for $\Gamma_{\{h+1,\ldots, g-1\}}$ is disconnected and $g-h-1\leq \frac{|X\cap Y|-1}{2}$. 
We have $h\leq \frac{n-|X|-1}{2}$, $g-h\leq \frac{|X\cap Y|-1}{2}+1$,  $r-g \leq \frac{n-|Y|-1}{2}$ and $n= |X| + |Y| - |X \cap Y|$, thus $r\leq \frac{n-1}{2}$, a contradiction.
\end{proof}

We now consider that $A$ is trivial or has the permutation representation graph (2) or (3) of Table~\ref{small}.

\begin{prop}\label{0trivialtran}
Let $i=0$. If $A$ is trivial and $\Gamma_{>1}$ is transitive on $\{1,\ldots,n\}\setminus\{a,b\}$, then $r\leq \frac{n-1}{2}$.
\end{prop}
\begin{proof}
As $\rho_0$ acts nontrivially on $\{1,\ldots,n\}\setminus\{a,b\}$ and  centralizes $\Gamma_{>1}$, by Proposition~\ref{<jon>jtran}, 
$r-2\leq \frac{n-2}{2}-1$ and therefore $r\leq n/2$. Suppose that we have the equality. Then $\Gamma_{>1}$ is as given in  Proposition~\ref{<jon>jtran}, and either $\rho_0$ or $\rho_1$ is odd, a contradiction. Hence $r\leq \frac{n-1}{2}$.
\end{proof}

\begin{prop}\label{12Tran}
Let  $A$ have the permutation representation $F_x=\emptyset$  or (3) of Table~\ref{small} (this implies that $n\geq 7$). 
If $\Gamma_{>4}$ is transitive on $n-7$ vertices then  $r\leq \frac{n-1}{2}$. 
\end{prop}
\begin{proof}
In this case $i=3$ and $\rho_i$ has full support on the $n-7$ vertices that are not fixed by $\Gamma_{>4}$.
So $n-7$ is even.  By Proposition~\ref{<jon>jtran}, $r-5\leq \frac{n-7}{2}-1$.
Moreover if we have the equality then either $\rho_3$ or $\rho_4$ is odd, a contradiction.
Hence $r-5< \frac{n-7}{2}-1$ and $r\leq \frac{n-1}{2}$. 
\end{proof}
\begin{prop}\label{specialbubbles}
Suppose that $A$ is trivial or has permutation representation graph (2) or (3) of Table~\ref{small} and $B$ is the alternating group.
If $\Gamma_{>i+1}$ is intransitive on the second $\Gamma_i$-orbit and  $\Gamma_{>i}$ has a 2-fracture graph, then $r\leq \frac{n-1}{2}$. 
Moreover if $r> \frac{n-3}{2}$ then one of the following possibilities must occur.
\begin{enumerate}
\item there exists $x\in\{0,\ldots, r-1\}$ such that $x\leq\frac{n-|X|-1}{2}$ with $X=\{1,\ldots, n\}\setminus {\rm Fix}(\Gamma_{>x})$.
\item $\Gamma$ has the following permutation representation graph.
$$
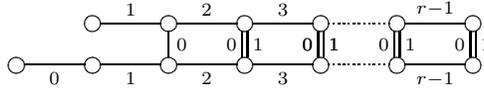
\xymatrix@-1.3pc{ &&*+[o][F]{} \ar@{-}[rr]^1&&*+[o][F]{} \ar@{-}[rr]^2\ar@{-}[d]^0&&*+[o][F]{} \ar@{-}[rr]^3\ar@{=}[d]^1_0&&*+[o][F]{}\ar@{=}[d]^1_0 \ar@{.}[rr]\ar@{=}[d]^1_0&&*+[o][F]{}\ar@{-}[rr]^{r-1}\ar@{=}[d]^1_0&&*+[o][F]{}\ar@{=}[d]^1_0\\
   *+[o][F]{}\ar@{-}[rr]_0&&*+[o][F]{}\ar@{-}[rr]_1&&*+[o][F]{}\ar@{-}[rr]_2&&*+[o][F]{} \ar@{-}[rr]_3&&*+[o][F]{} \ar@{.}[rr]&&*+[o][F]{}\ar@{-}[rr]_{r-1}&&*+[o][F]{}}$$
\end{enumerate}
\end{prop}
\begin{proof}
Consider the graph $\mathcal{C}$ as in the proof of Proposition~\ref{t}, with $t=i+1$. 
Let  $U$, $U_s$, $G_s$, $F_s$ and $I_s$ be as in Proposition~\ref{t}.

Suppose there is a component $U_s$ that does not have a 2-fracture graph. 
Let $m$ and $x$ be the minimal and the maximal  label  of that component respectively. 
We proved in Proposition~\ref{t} that $m\in\{t+1, t+2\}$ and accordantly to these possibilities for $m$ the permutation representation graph of $\Gamma$ contains one of the following graphs.

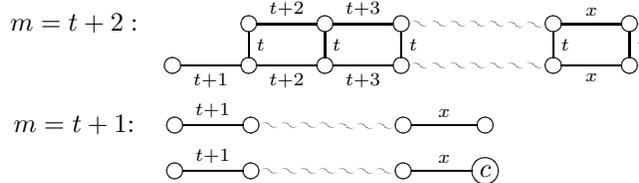
\begin{tabular}{cc}
$m=t+2:$&$\xymatrix@-1.3pc{ &&*+[o][F]{} \ar@{-}[rr]^{t+2}\ar@{-}[d]^t&&*+[o][F]{} \ar@{-}[rr]^{t+3}\ar@{-}[d]^t&&*+[o][F]{} \ar@{~}[rrrr]\ar@{-}[d]^t&&&&*+[o][F]{}\ar@{-}[rr]^x\ar@{-}[d]^t&&*+[o][F]{}\ar@{-}[d]^t\\
 *+[o][F]{}\ar@{-}[rr]_{t+1}&&*+[o][F]{} \ar@{-}[rr]_{t+2}&&*+[o][F]{} \ar@{-}[rr]_{t+3}&&*+[o][F]{} \ar@{~}[rrrr]&&&&*+[o][F]{}\ar@{-}[rr]_x&&*+[o][F]{}}$\\

$m=t+1$:& $\xymatrix@-1.3pc{  *+[o][F]{} \ar@{-}[rr]^{t+1}&&*+[o][F]{} \ar@{~}[rrrr]&&&&*+[o][F]{}\ar@{-}[rr]^x&&*+[o][F]{}&&&&\\
*+[o][F]{} \ar@{-}[rr]^{t+1}&&*+[o][F]{} \ar@{~}[rrrr]&&&&*+[o][F]{}\ar@{-}[rr]^x&&*+[o][F]{c} }$
\end{tabular}

Let $P$ be the set of vertices of the first graph  and $Q$ be the set of vertices of the second graph. 
If $x\neq r-1$, in the first case $\Gamma_{>x}$ fixes $P$, and in the second case  $\Gamma_{>x}$ fixes  $Q\setminus\{c\}$. If $i=0$ and $t=1$ we have that $x\leq \frac{|P|+1}{2}\mbox{ or }x\leq \frac{|Q|}{2}.$
If $i=3$ (and $t=4$), $x\leq \frac{|P|+7}{2}\mbox{ or }x\leq \frac{|Q|+6}{2}.$

Suppose that $x=r-1$ and $r\geq \frac{n-2}{2}$.
Consider first that $A$ is trivial and $m=t+2$. In this case $n\geq |P|+2$ and $\frac{|P|+1}{2}\geq \frac{n-4}{2}$, thus $|P|\in\{n-5,n-4,n-3,n-2\}$.  Then it is possible to determined the permutation representation graph of $\Gamma$ according to the value of  $|P|$.  If $|P|=n-2$ then $\Gamma$ is the graph (1) of Table~\ref{GammaWhenATrivial1}  which is not an even group. It is also not possible to get an even group when $|P|=n-3$, as there must exist a $\Gamma_{>1}$-orbit with a 2-edge. For  $|P|=n-4$ it is possible to create such a component and we get the graph (2) of Table~\ref{GammaWhenATrivial1}.  But as $\Gamma_{>1}$ does not have a 2-fracture graph we get a contradiction. 
If $|P|=n-5$, it is not possible to create a third $\Gamma_{>1}$-orbit, thus $\Gamma_{>1}$ has exactly two components, one containing $P$
 and the other having an even number of vertices swapped pairwise by $\rho_0$, a contradiction.

\begin{table}
\begin{small}
\[\begin{tabular}{|c|c|c|}
\hline
\#&$ |P|$& \textbf{Possibilities for $\Gamma$ when $A$ is trivial }\\
\hline
\textbf{1}&$n-2$&$\xymatrix@-1.5pc{ &&&&&&*+[o][F]{} \ar@{-}[rr]^3\ar@{-}[d]^1&&*+[o][F]{}\ar@{-}[d]^1 \ar@{.}[rr]\ar@{-}[d]^1&&*+[o][F]{}\ar@{-}[rr]^{r-1}\ar@{-}[d]^1&&*+[o][F]{}\ar@{-}[d]^1\\
   *+[o][F]{}\ar@{-}[rr]_0&&*+[o][F]{}\ar@{-}[rr]_1&&*+[o][F]{}\ar@{-}[rr]_2&&*+[o][F]{} \ar@{-}[rr]_3&&*+[o][F]{} \ar@{.}[rr]&&*+[o][F]{}\ar@{-}[rr]_{r-1}&&*+[o][F]{}}$
\\
\hline
\textbf{2}&$n-4$&$\xymatrix@-1.5pc{&&&& &&&&&&*+[o][F]{} \ar@{-}[rr]^3\ar@{-}[d]^1&&*+[o][F]{}\ar@{-}[d]^1 \ar@{.}[rr]\ar@{-}[d]^1&&*+[o][F]{}\ar@{-}[rr]^{r-1}\ar@{-}[d]^1&&*+[o][F]{}\ar@{-}[d]^1\\
   *+[o][F]{}\ar@{-}[rr]_0&&*+[o][F]{}\ar@{-}[rr]_1&&*+[o][F]{}\ar@{=}[rr]_2^0&&*+[o][F]{}\ar@{-}[rr]_1&&*+[o][F]{}\ar@{-}[rr]_2&&*+[o][F]{} \ar@{-}[rr]_3&&*+[o][F]{} \ar@{.}[rr]&&*+[o][F]{}\ar@{-}[rr]_{r-1}&&*+[o][F]{}}$
\\
\hline
\end{tabular}\]
\end{small}
\caption{Possibilities depending on $|P|$ for $\Gamma$ when $A$ is trivial}\label{GammaWhenATrivial1} 
\end{table}

Now let $A$ be trivial and  $m=t+1$.  In this case $|Q|\in\{n-4,n-3,n-2\}$. In Table~\ref{GammaWhenATrivial2} we list all possibilities for the permutation representation graph of $\Gamma$ for each value of $|Q|$.  If the permutation representation graph of  $\Gamma$ is one of the graphs  (3), (5a), (5b) or (5c), then $\Gamma$ is odd, a contradiction. Thus the only possibility is the permutation representation graph (4), giving the graph of the statement of this proposition.

\begin{table}
\begin{small}
\[\begin{tabular}{|c|c|c|}
\hline
\#&$|Q|$ & \textbf{Possibilities for $\Gamma$ when $A$ is trivial }\\
\hline
\textbf{3}&$n-2$&$\xymatrix@-1.5pc{ &&&&*+[o][F]{} \ar@{-}[rr]^2\ar@{-}[d]^0&&*+[o][F]{} \ar@{-}[rr]^3\ar@{=}[d]^1_0&&*+[o][F]{}\ar@{=}[d]^1_0 \ar@{.}[rr]\ar@{=}[d]^1_0&&*+[o][F]{}\ar@{-}[rr]^{r-1}\ar@{=}[d]^1_0&&*+[o][F]{}\ar@{=}[d]^1_0\\
   *+[o][F]{}\ar@{-}[rr]_0&&*+[o][F]{}\ar@{-}[rr]_1&&*+[o][F]{}\ar@{-}[rr]_2&&*+[o][F]{} \ar@{-}[rr]_3&&*+[o][F]{} \ar@{.}[rr]&&*+[o][F]{}\ar@{-}[rr]_{r-1}&&*+[o][F]{}}$
\\
\hline
\textbf{4}&$n-3$&$\xymatrix@-1.5pc{ &&*+[o][F]{} \ar@{-}[rr]^1&&*+[o][F]{} \ar@{-}[rr]^2\ar@{-}[d]^0&&*+[o][F]{} \ar@{-}[rr]^3\ar@{=}[d]^1_0&&*+[o][F]{}\ar@{=}[d]^1_0 \ar@{.}[rr]\ar@{=}[d]^1_0&&*+[o][F]{}\ar@{-}[rr]^{r-1}\ar@{=}[d]^1_0&&*+[o][F]{}\ar@{=}[d]^1_0\\
   *+[o][F]{}\ar@{-}[rr]_0&&*+[o][F]{}\ar@{-}[rr]_1&&*+[o][F]{}\ar@{-}[rr]_2&&*+[o][F]{} \ar@{-}[rr]_3&&*+[o][F]{} \ar@{.}[rr]&&*+[o][F]{}\ar@{-}[rr]_{r-1}&&*+[o][F]{}}$
\\
\hline 
\textbf{5a}&$n-4$&$\xymatrix@-1.5pc{ *+[o][F]{} \ar@{-}[rr]^2&&*+[o][F]{} \ar@{-}[rr]^1&&*+[o][F]{} \ar@{-}[rr]^2\ar@{-}[d]^0&&*+[o][F]{} \ar@{-}[rr]^3\ar@{=}[d]^1_0&&*+[o][F]{}\ar@{=}[d]^1_0 \ar@{.}[rr]\ar@{=}[d]^1_0&&*+[o][F]{}\ar@{-}[rr]^{r-1}\ar@{=}[d]^1_0&&*+[o][F]{}\ar@{=}[d]^1_0\\
   *+[o][F]{}\ar@{-}[rr]_0&&*+[o][F]{}\ar@{-}[rr]_1&&*+[o][F]{}\ar@{-}[rr]_2&&*+[o][F]{} \ar@{-}[rr]_3&&*+[o][F]{} \ar@{.}[rr]&&*+[o][F]{}\ar@{-}[rr]_{r-1}&&*+[o][F]{}}$
\\
  && or \\
&&$\xymatrix@-1.5pc{ *+[o][F]{} \ar@{=}[rr]^2_0&&*+[o][F]{} \ar@{-}[rr]^1&&*+[o][F]{} \ar@{-}[rr]^2\ar@{-}[d]^0&&*+[o][F]{} \ar@{-}[rr]^3\ar@{=}[d]^1_0&&*+[o][F]{}\ar@{=}[d]^1_0 \ar@{.}[rr]\ar@{=}[d]^1_0&&*+[o][F]{}\ar@{-}[rr]^{r-1}\ar@{=}[d]^1_0&&*+[o][F]{}\ar@{=}[d]^1_0\\
   *+[o][F]{}\ar@{-}[rr]_0&&*+[o][F]{}\ar@{-}[rr]_1&&*+[o][F]{}\ar@{-}[rr]_2&&*+[o][F]{} \ar@{-}[rr]_3&&*+[o][F]{} \ar@{.}[rr]&&*+[o][F]{}\ar@{-}[rr]_{r-1}&&*+[o][F]{}}$\\
\hline 
\textbf{5b} &$n-4$ &$\xymatrix@-1.5pc{ &&&&*+[o][F]{} \ar@{-}[rr]^2\ar@{-}[d]^0&&*+[o][F]{} \ar@{-}[rr]^3\ar@{=}[d]^1_0&&*+[o][F]{}\ar@{=}[d]^1_0 \ar@{.}[rr]\ar@{=}[d]^1_0&&*+[o][F]{}\ar@{-}[rr]^{r-2}\ar@{=}[d]^1_0&&*+[o][F]{}\ar@{=}[d]^1_0\ar@{-}[rr]^{r-1}\ar@{=}[d]^1_0&&*+[o][F]{}\ar@{=}[d]^1_0\ar@{-}[rr]^{r-2}\ar@{=}[d]^1_0&&*+[o][F]{}\ar@{=}[d]^1_0\\
   *+[o][F]{}\ar@{-}[rr]_0&&*+[o][F]{}\ar@{-}[rr]_1&&*+[o][F]{}\ar@{-}[rr]_2&&*+[o][F]{} \ar@{-}[rr]_3&&*+[o][F]{} \ar@{.}[rr]&&*+[o][F]{}\ar@{-}[rr]_{r-2}&&*+[o][F]{}\ar@{-}[rr]_{r-1}&&*+[o][F]{}\ar@{-}[rr]_{r-2}&&*+[o][F]{}}$
\\
\hline 
\textbf{5c} &$n-4$ &$\xymatrix@-1.5pc{ &&&&&&&&*+[o][F]{} \ar@{-}[rr]^2\ar@{-}[d]^0&&*+[o][F]{} \ar@{-}[rr]^3\ar@{=}[d]^1_0&&*+[o][F]{}\ar@{=}[d]^1_0 \ar@{.}[rr]\ar@{=}[d]^1_0&&*+[o][F]{}\ar@{-}[rr]^{r-1}\ar@{=}[d]^1_0&&*+[o][F]{}\ar@{=}[d]^1_0\\
   *+[o][F]{}\ar@{-}[rr]_0&&*+[o][F]{}\ar@{-}[rr]_1&&*+[o][F]{}\ar@{-}[rr]_2&&*+[o][F]{}\ar@{-}[rr]_1&&*+[o][F]{}\ar@{-}[rr]_2&&*+[o][F]{} \ar@{-}[rr]_3&&*+[o][F]{} \ar@{.}[rr]&&*+[o][F]{}\ar@{-}[rr]_{r-1}&&*+[o][F]{}}$\\
   
   && or \\
   &&$\xymatrix@-1.5pc{ &&&&&&&&*+[o][F]{} \ar@{-}[rr]^2\ar@{-}[d]^0&&*+[o][F]{} \ar@{-}[rr]^3\ar@{=}[d]^1_0&&*+[o][F]{}\ar@{=}[d]^1_0 \ar@{.}[rr]\ar@{=}[d]^1_0&&*+[o][F]{}\ar@{-}[rr]^{r-1}\ar@{=}[d]^1_0&&*+[o][F]{}\ar@{=}[d]^1_0\\
   *+[o][F]{}\ar@{-}[rr]_0&&*+[o][F]{}\ar@{-}[rr]_1&&*+[o][F]{}\ar@{=}[rr]_2^0&&*+[o][F]{}\ar@{-}[rr]_1&&*+[o][F]{}\ar@{-}[rr]_2&&*+[o][F]{} \ar@{-}[rr]_3&&*+[o][F]{} \ar@{.}[rr]&&*+[o][F]{}\ar@{-}[rr]_{r-1}&&*+[o][F]{}}$\\
\hline
\end{tabular}\]
\end{small}
\caption{Possibilities depending on $|Q|$ for $\Gamma$ when $A$ is trivial}\label{GammaWhenATrivial2} 
\end{table}

Now let $A$ be the permutation representation graph (2) of Table~\ref{small} and  $m=t+2$. If $r\geq \frac{n-2}{2}$  and $x=r-1$ then $|P|\in\{n-11,\, n-10,\,n-9,\,n-8,\,n-7\}$.
In Table~\ref{GammaWhenA(1)} we list all possibilities for $\Gamma$ according to $|P|$. For $|P|=n-7$ we get the permutation representation graph (1) and $\Gamma$ is odd, a contradiction. Then $|P|<n-8$,  as there must exist a $\Gamma_{>4}$-orbit containing a $5$-edge and a $3$-edge.  Thus for $|P|=n-9$ we get the permutation representation graph (2), but  $\Gamma_{>4}$ doesn't have a 2-fracture graph, a contradiction. Now suppose $|P|\leq n-10$.
Either there are two $\Gamma_{>4}$-orbits, one having a set of vertices $P$ 
and  another having $\rho_3$ swapping all its vertices pairwise, or
there exist a third  $\Gamma_{>4}$-orbit containing a $5$-edge. For this to happen at least two additional vertices are needed, thus $|P|= n-11$. This gives the permutation representation graph (3), which again corresponds to an odd group, a contradiction.

\begin{table}
\begin{small}
\[\begin{tabular}{|c|c|c|}
\hline
\# &$|P|$& \begin{tabular}{c} \textbf{Possibilities for $\Gamma$ when $A$}\\\textbf{ has the permutation representation graph (2)} \end{tabular}\\
\hline
\textbf{1} & $n-7$ &$\xymatrix@-1.5pc{&& && && && && && && &&*+[o][F]{}\ar@{-}[rr]^6\ar@{-}[d]^4&& *+[o][F]{}\ar@{.}[rr]\ar@{-}[d]^4&&*+[o][F]{}\ar@{-}[rr]^{r-1}\ar@{-}[d]^4&&*+[o][F]{}\ar@{-}[d]^4\\
   *+[o][F]{}\ar@{=}[rr]_0^2&&*+[o][F]{}\ar@{-}[rr]_1&&*+[o][F]{}\ar@{-}[rr]_0&&*+[o][F]{}\ar@{-}[rr]_1&&*+[o][F]{}\ar@{-}[rr]_2&&*+[o][F]{}\ar@{-}[rr]_3&&*+[o][F]{}\ar@{-}[rr]_4&&*+[o][F]{}\ar@{-}[rr]_5&&*+[o][F]{}\ar@{-}[rr]_6&&*+[o][F]{} \ar@{.}[rr]&&*+[o][F]{}\ar@{-}[rr]_{r-1}&&*+[o][F]{}}$\\
\hline 
\textbf{2} & $n-9$ &$\xymatrix@-1.6pc{&& && && && && && && && && &&*+[o][F]{}\ar@{-}[rr]^6\ar@{-}[d]^4&& *+[o][F]{}\ar@{.}[rr]\ar@{-}[d]^4&&*+[o][F]{}\ar@{-}[rr]^{r-1}\ar@{-}[d]^4&&*+[o][F]{}\ar@{-}[d]^4\\
   *+[o][F]{}\ar@{=}[rr]_0^2&&*+[o][F]{}\ar@{-}[rr]_1&&*+[o][F]{}\ar@{-}[rr]_0&&*+[o][F]{}\ar@{-}[rr]_1&&*+[o][F]{}\ar@{-}[rr]_2&&*+[o][F]{}\ar@{-}[rr]_3&&*+[o][F]{}\ar@{-}[rr]_4&&*+[o][F]{}\ar@{=}[rr]_5^3&&*+[o][F]{}\ar@{-}[rr]_4&&*+[o][F]{}\ar@{-}[rr]_5&&*+[o][F]{}\ar@{-}[rr]_6&&*+[o][F]{} \ar@{.}[rr]&&*+[o][F]{}\ar@{-}[rr]_{r-1}&&*+[o][F]{}}$\\
\hline 
\textbf{3} & $n-11$ &$\xymatrix@-1.6pc{&& && && && && && && *+[o][F]{}\ar@{=}[rr]^3_4&&*+[o][F]{}  && &&*+[o][F]{}\ar@{-}[rr]^6\ar@{-}[d]^4&& *+[o][F]{}\ar@{.}[rr]\ar@{-}[d]^4&&*+[o][F]{}\ar@{-}[rr]^{r-1}\ar@{-}[d]^4&&*+[o][F]{}\ar@{-}[d]^4\\
   *+[o][F]{}\ar@{=}[rr]_0^2&&*+[o][F]{}\ar@{-}[rr]_1&&*+[o][F]{}\ar@{-}[rr]_0&&*+[o][F]{}\ar@{-}[rr]_1&&*+[o][F]{}\ar@{-}[rr]_2&&*+[o][F]{}\ar@{-}[rr]_3&&*+[o][F]{}\ar@{-}[rr]_4&&*+[o][F]{}\ar@{-}[u]^5\ar@{-}[rr]_3&&*+[o][F]{}\ar@{-}[u]_5\ar@{-}[rr]_4&&*+[o][F]{}\ar@{-}[rr]_5&&*+[o][F]{}\ar@{-}[rr]_6&&*+[o][F]{} \ar@{.}[rr]&&*+[o][F]{}\ar@{-}[rr]_{r-1}&&*+[o][F]{}}$\\
\hline 
\end{tabular}\]
\end{small} 
\caption{Possibilities depending on $|P|$ for $\Gamma$ when $A$ has permutation representation graph (2)}\label{GammaWhenA(1)} 
\end{table}

In Table~\ref{GammaWhenA(1)2}, we list all possibilities for $\Gamma$ when $A$ has the permutation representation graph (2) of Table~\ref{small} and  $m=t+1$. As before $r\geq \frac{n-2}{2}$ and $x=r-1$, hence $|Q|\in\{n-10,\,n-9,\,n-8,\,n-7\}$. In (4), (6a), (6b), (6c), (6d), (7b), (7c), (7d), (7e) and (7f) $\Gamma$ is odd. In the remaining case the intersection condition fails. 

\begin{table}
\begin{small}
\[\begin{tabular}{|c|c|c|}
\hline
\# &$|Q|$& \begin{tabular}{c} \textbf{Possibilities for $\Gamma$ when $A$}\\\textbf{ has the permutation representation graph (2)} \end{tabular}\\
\hline

\textbf{4} & $n-7$ &$\xymatrix@-1.5pc{&& && && && && && &&*+[o][F]{}\ar@{-}[rr]^5\ar@{-}[d]^3 &&*+[o][F]{}\ar@{-}[rr]^6\ar@{=}[d]^4_3&& *+[o][F]{}\ar@{.}[rr]\ar@{=}[d]^4_3&&*+[o][F]{}\ar@{-}[rr]^{r-1}\ar@{=}[d]^4_3&&*+[o][F]{}\ar@{=}[d]^4_3\\
   *+[o][F]{}\ar@{=}[rr]_0^2&&*+[o][F]{}\ar@{-}[rr]_1&&*+[o][F]{}\ar@{-}[rr]_0&&*+[o][F]{}\ar@{-}[rr]_1&&*+[o][F]{}\ar@{-}[rr]_2&&*+[o][F]{}\ar@{-}[rr]_3&&*+[o][F]{}\ar@{-}[rr]_4&&*+[o][F]{}\ar@{-}[rr]_5&&*+[o][F]{}\ar@{-}[rr]_6&&*+[o][F]{} \ar@{.}[rr]&&*+[o][F]{}\ar@{-}[rr]_{r-1}&&*+[o][F]{}}$\\
\hline 

\textbf{5} & $n-8$ &$\xymatrix@-1.5pc{&& && && && && &&*+[o][F]{}\ar@{-}[rr]^4 &&*+[o][F]{}\ar@{-}[rr]^5\ar@{-}[d]^3 &&*+[o][F]{}\ar@{-}[rr]^6\ar@{=}[d]^4_3&& *+[o][F]{}\ar@{.}[rr]\ar@{=}[d]^4_3&&*+[o][F]{}\ar@{-}[rr]^{r-1}\ar@{=}[d]^4_3&&*+[o][F]{}\ar@{=}[d]^4_3\\
   *+[o][F]{}\ar@{=}[rr]_0^2&&*+[o][F]{}\ar@{-}[rr]_1&&*+[o][F]{}\ar@{-}[rr]_0&&*+[o][F]{}\ar@{-}[rr]_1&&*+[o][F]{}\ar@{-}[rr]_2&&*+[o][F]{}\ar@{-}[rr]_3&&*+[o][F]{}\ar@{-}[rr]_4&&*+[o][F]{}\ar@{-}[rr]_5&&*+[o][F]{}\ar@{-}[rr]_6&&*+[o][F]{} \ar@{.}[rr]&&*+[o][F]{}\ar@{-}[rr]_{r-1}&&*+[o][F]{}}$\\
\hline

\textbf{6a} & $n-9$ &$\xymatrix@-1.5pc{&& && && && &&*+[o][F]{}\ar@{-}[rr]^5 &&*+[o][F]{}\ar@{-}[rr]^4 &&*+[o][F]{}\ar@{-}[rr]^5\ar@{-}[d]^3 &&*+[o][F]{}\ar@{-}[rr]^6\ar@{=}[d]^4_3&& *+[o][F]{}\ar@{.}[rr]\ar@{=}[d]^4_3&&*+[o][F]{}\ar@{-}[rr]^{r-1}\ar@{=}[d]^4_3&&*+[o][F]{}\ar@{=}[d]^4_3\\
   *+[o][F]{}\ar@{=}[rr]_0^2&&*+[o][F]{}\ar@{-}[rr]_1&&*+[o][F]{}\ar@{-}[rr]_0&&*+[o][F]{}\ar@{-}[rr]_1&&*+[o][F]{}\ar@{-}[rr]_2&&*+[o][F]{}\ar@{-}[rr]_3&&*+[o][F]{}\ar@{-}[rr]_4&&*+[o][F]{}\ar@{-}[rr]_5&&*+[o][F]{}\ar@{-}[rr]_6&&*+[o][F]{} \ar@{.}[rr]&&*+[o][F]{}\ar@{-}[rr]_{r-1}&&*+[o][F]{}}$\\
   \hline
  \textbf{6b} & $n-9$ &$\xymatrix@-1.5pc{&&&& && && && && && && &&*+[o][F]{}\ar@{-}[rr]^5\ar@{-}[d]^3 &&*+[o][F]{}\ar@{-}[rr]^6\ar@{=}[d]^4_3&& *+[o][F]{}\ar@{.}[rr]\ar@{=}[d]^4_3&&*+[o][F]{}\ar@{-}[rr]^{r-1}\ar@{=}[d]^4_3&&*+[o][F]{}\ar@{=}[d]^4_3\\
   *+[o][F]{}\ar@{=}[rr]_0^2&&*+[o][F]{}\ar@{-}[rr]_1&&*+[o][F]{}\ar@{-}[rr]_0&&*+[o][F]{}\ar@{-}[rr]_1&&*+[o][F]{}\ar@{-}[rr]_2&&*+[o][F]{}\ar@{-}[rr]_3&&
   *+[o][F]{}\ar@{-}[rr]_4&&*+[o][F]{}\ar@{-}[rr]_5&&*+[o][F]{}\ar@{-}[rr]_4&&*+[o][F]{}\ar@{-}[rr]_5&&*+[o][F]{}\ar@{-}[rr]_6&&*+[o][F]{} \ar@{.}[rr]&&*+[o][F]{}\ar@{-}[rr]_{r-1}&&*+[o][F]{}}$\\
\hline 
\textbf{6c} & $n-9$ & $\xymatrix@-1.6pc{&& && && && && && &&*+[o][F]{}\ar@{-}[rr]^5\ar@{-}[d]^3 &&*+[o][F]{}\ar@{-}[rr]^6\ar@{=}[d]^4_3&& *+[o][F]{}\ar@{.}[rr]\ar@{=}[d]^4_3&&*+[o][F]{}\ar@{-}[rr]^{r-2}\ar@{=}[d]^4_3&&*+[o][F]{}\ar@{-}[rr]^{r-1}\ar@{=}[d]^4_3&&*+[o][F]{}\ar@{-}[rr]^{r-2}\ar@{=}[d]^4_3&&*+[o][F]{}\ar@{=}[d]^4_3\\
   *+[o][F]{}\ar@{=}[rr]_0^2&&*+[o][F]{}\ar@{-}[rr]_1&&*+[o][F]{}\ar@{-}[rr]_0&&*+[o][F]{}\ar@{-}[rr]_1&&*+[o][F]{}\ar@{-}[rr]_2&&*+[o][F]{}\ar@{-}[rr]_3&&*+[o][F]{}\ar@{-}[rr]_4&&*+[o][F]{}\ar@{-}[rr]_5&&*+[o][F]{}\ar@{-}[rr]_6&&*+[o][F]{} \ar@{.}[rr]&&*+[o][F]{}\ar@{-}[rr]_{r-2}&&*+[o][F]{}\ar@{-}[rr]_{r-1}&&*+[o][F]{}\ar@{-}[rr]_{r-2}&&*+[o][F]{}}$\\
\hline
\textbf{6d} & $n-9$ &$\xymatrix@-1.5pc{&& && && && &&*+[o][F]{}\ar@{=}[rr]^5_3 &&*+[o][F]{}\ar@{-}[rr]^4 &&*+[o][F]{}\ar@{-}[rr]^5\ar@{-}[d]^3 &&*+[o][F]{}\ar@{-}[rr]^6\ar@{=}[d]^4_3&& *+[o][F]{}\ar@{.}[rr]\ar@{=}[d]^4_3&&*+[o][F]{}\ar@{-}[rr]^{r-1}\ar@{=}[d]^4_3&&*+[o][F]{}\ar@{=}[d]^4_3\\
   *+[o][F]{}\ar@{=}[rr]_0^2&&*+[o][F]{}\ar@{-}[rr]_1&&*+[o][F]{}\ar@{-}[rr]_0&&*+[o][F]{}\ar@{-}[rr]_1&&*+[o][F]{}\ar@{-}[rr]_2&&*+[o][F]{}\ar@{-}[rr]_3&&*+[o][F]{}\ar@{-}[rr]_4&&*+[o][F]{}\ar@{-}[rr]_5&&*+[o][F]{}\ar@{-}[rr]_6&&*+[o][F]{} \ar@{.}[rr]&&*+[o][F]{}\ar@{-}[rr]_{r-1}&&*+[o][F]{}}$\\
   \hline
\textbf{7a} & $n-10$ & $\xymatrix@-1.6pc{&& && && && && &&*+[o][F]{}\ar@{-}[rr]^4 &&*+[o][F]{}\ar@{-}[rr]^5\ar@{-}[d]^3 &&*+[o][F]{}\ar@{-}[rr]^6\ar@{=}[d]^4_3&& *+[o][F]{}\ar@{.}[rr]\ar@{=}[d]^4_3&&*+[o][F]{}\ar@{-}[rr]^{r-2}\ar@{=}[d]^4_3&&*+[o][F]{}\ar@{-}[rr]^{r-1}\ar@{=}[d]^4_3&&*+[o][F]{}\ar@{-}[rr]^{r-2}\ar@{=}[d]^4_3&&*+[o][F]{}\ar@{=}[d]^4_3\\
   *+[o][F]{}\ar@{=}[rr]_0^2&&*+[o][F]{}\ar@{-}[rr]_1&&*+[o][F]{}\ar@{-}[rr]_0&&*+[o][F]{}\ar@{-}[rr]_1&&*+[o][F]{}\ar@{-}[rr]_2&&*+[o][F]{}\ar@{-}[rr]_3&&*+[o][F]{}\ar@{-}[rr]_4&&*+[o][F]{}\ar@{-}[rr]_5&&*+[o][F]{}\ar@{-}[rr]_6&&*+[o][F]{} \ar@{.}[rr]&&*+[o][F]{}\ar@{-}[rr]_{r-2}&&*+[o][F]{}\ar@{-}[rr]_{r-1}&&*+[o][F]{}\ar@{-}[rr]_{r-2}&&*+[o][F]{}}$\\
 \hline
 \textbf{7b} & $n-10$ &$\xymatrix@-1.5pc{&& && && && *+[o][F]{}\ar@{-}[rr]^4&&*+[o][F]{}\ar@{-}[rr]^5 &&*+[o][F]{}\ar@{-}[rr]^4 &&*+[o][F]{}\ar@{-}[rr]^5\ar@{-}[d]^3 &&*+[o][F]{}\ar@{-}[rr]^6\ar@{=}[d]^4_3&& *+[o][F]{}\ar@{.}[rr]\ar@{=}[d]^4_3&&*+[o][F]{}\ar@{-}[rr]^{r-1}\ar@{=}[d]^4_3&&*+[o][F]{}\ar@{=}[d]^4_3\\
   *+[o][F]{}\ar@{=}[rr]_0^2&&*+[o][F]{}\ar@{-}[rr]_1&&*+[o][F]{}\ar@{-}[rr]_0&&*+[o][F]{}\ar@{-}[rr]_1&&*+[o][F]{}\ar@{-}[rr]_2&&*+[o][F]{}\ar@{-}[rr]_3&&*+[o][F]{}\ar@{-}[rr]_4&&*+[o][F]{}\ar@{-}[rr]_5&&*+[o][F]{}\ar@{-}[rr]_6&&*+[o][F]{} \ar@{.}[rr]&&*+[o][F]{}\ar@{-}[rr]_{r-1}&&*+[o][F]{}}$\\
   \hline
  \textbf{7c} & $n-10$ &$\xymatrix@-1.5pc{&&&& && && && && && &&  *+[o][F]{}\ar@{-}[rr]^4&&*+[o][F]{}\ar@{-}[rr]^5\ar@{-}[d]^3 &&*+[o][F]{}\ar@{-}[rr]^6\ar@{=}[d]^4_3&& *+[o][F]{}\ar@{.}[rr]\ar@{=}[d]^4_3&&*+[o][F]{}\ar@{-}[rr]^{r-1}\ar@{=}[d]^4_3&&*+[o][F]{}\ar@{=}[d]^4_3\\
   *+[o][F]{}\ar@{=}[rr]_0^2&&*+[o][F]{}\ar@{-}[rr]_1&&*+[o][F]{}\ar@{-}[rr]_0&&*+[o][F]{}\ar@{-}[rr]_1&&*+[o][F]{}\ar@{-}[rr]_2&&*+[o][F]{}\ar@{-}[rr]_3&&
   *+[o][F]{}\ar@{-}[rr]_4&&*+[o][F]{}\ar@{-}[rr]_5&&*+[o][F]{}\ar@{-}[rr]_4&&*+[o][F]{}\ar@{-}[rr]_5&&*+[o][F]{}\ar@{-}[rr]_6&&*+[o][F]{} \ar@{.}[rr]&&*+[o][F]{}\ar@{-}[rr]_{r-1}&&*+[o][F]{}}$\\
\hline 
\textbf{7d} & $n-10$ & $\xymatrix@-1.6pc{&& && && && && && *+[o][F]{}\ar@{-}[rr]^4 &&*+[o][F]{}\ar@{-}[rr]^5\ar@{-}[d]^3 &&*+[o][F]{}\ar@{-}[rr]^6\ar@{=}[d]^4_3&& *+[o][F]{}\ar@{.}[rr]\ar@{=}[d]^4_3&&*+[o][F]{}\ar@{-}[rr]^{r-2}\ar@{=}[d]^4_3&&*+[o][F]{}\ar@{-}[rr]^{r-1}\ar@{=}[d]^4_3&&*+[o][F]{}\ar@{-}[rr]^{r-2}\ar@{=}[d]^4_3&&*+[o][F]{}\ar@{=}[d]^4_3\\
   *+[o][F]{}\ar@{=}[rr]_0^2&&*+[o][F]{}\ar@{-}[rr]_1&&*+[o][F]{}\ar@{-}[rr]_0&&*+[o][F]{}\ar@{-}[rr]_1&&*+[o][F]{}\ar@{-}[rr]_2&&*+[o][F]{}\ar@{-}[rr]_3&&*+[o][F]{}\ar@{-}[rr]_4&&*+[o][F]{}\ar@{-}[rr]_5&&*+[o][F]{}\ar@{-}[rr]_6&&*+[o][F]{} \ar@{.}[rr]&&*+[o][F]{}\ar@{-}[rr]_{r-2}&&*+[o][F]{}\ar@{-}[rr]_{r-1}&&*+[o][F]{}\ar@{-}[rr]_{r-2}&&*+[o][F]{}}$\\
\hline
\textbf{7e} & $n-10$ &$\xymatrix@-1.5pc{&& && && && *+[o][F]{}\ar@{-}[rr]^6&&*+[o][F]{}\ar@{-}[rr]^5 &&*+[o][F]{}\ar@{-}[rr]^4 &&*+[o][F]{}\ar@{-}[rr]^5\ar@{-}[d]^3 &&*+[o][F]{}\ar@{-}[rr]^6\ar@{=}[d]^4_3&& *+[o][F]{}\ar@{.}[rr]\ar@{=}[d]^4_3&&*+[o][F]{}\ar@{-}[rr]^{r-1}\ar@{=}[d]^4_3&&*+[o][F]{}\ar@{=}[d]^4_3\\
   *+[o][F]{}\ar@{=}[rr]_0^2&&*+[o][F]{}\ar@{-}[rr]_1&&*+[o][F]{}\ar@{-}[rr]_0&&*+[o][F]{}\ar@{-}[rr]_1&&*+[o][F]{}\ar@{-}[rr]_2&&*+[o][F]{}\ar@{-}[rr]_3&&*+[o][F]{}\ar@{-}[rr]_4&&*+[o][F]{}\ar@{-}[rr]_5&&*+[o][F]{}\ar@{-}[rr]_6&&*+[o][F]{} \ar@{.}[rr]&&*+[o][F]{}\ar@{-}[rr]_{r-1}&&*+[o][F]{}}$\\
    \hline
 \textbf{7f} & $n-10$ &$\xymatrix@-1.5pc{&& && && && *+[o][F]{}\ar@{-}[rr]^4&&*+[o][F]{}\ar@{=}[rr]^5_3 &&*+[o][F]{}\ar@{-}[rr]^4 &&*+[o][F]{}\ar@{-}[rr]^5\ar@{-}[d]^3 &&*+[o][F]{}\ar@{-}[rr]^6\ar@{=}[d]^4_3&& *+[o][F]{}\ar@{.}[rr]\ar@{=}[d]^4_3&&*+[o][F]{}\ar@{-}[rr]^{r-1}\ar@{=}[d]^4_3&&*+[o][F]{}\ar@{=}[d]^4_3\\
   *+[o][F]{}\ar@{=}[rr]_0^2&&*+[o][F]{}\ar@{-}[rr]_1&&*+[o][F]{}\ar@{-}[rr]_0&&*+[o][F]{}\ar@{-}[rr]_1&&*+[o][F]{}\ar@{-}[rr]_2&&*+[o][F]{}\ar@{-}[rr]_3&&*+[o][F]{}\ar@{-}[rr]_4&&*+[o][F]{}\ar@{-}[rr]_5&&*+[o][F]{}\ar@{-}[rr]_6&&*+[o][F]{} \ar@{.}[rr]&&*+[o][F]{}\ar@{-}[rr]_{r-1}&&*+[o][F]{}}$\\
    \hline
\end{tabular}\]
\end{small}
\caption{Possibilities depending on $|Q|$ for $\Gamma$ when $A$ has permutation representation graph (2)}\label{GammaWhenA(1)2} 
\end{table}

If $A$ has the permutation representation graph (3) of Table~\ref{small} we get the same contradictions as in Tables~\ref{GammaWhenA(1)} and~\ref{GammaWhenA(1)2}.

Now suppose that $x\neq r-1$. The group $\Gamma_{>x}$ fixes the first $\Gamma_i$-orbit and the vertex $b$. 
Let $X:=\{1,\ldots, n\}\setminus {\rm Fix}(\Gamma_{>x})$.
When $i=0$ and $t=1$ we have $|X|\leq n-(|P|+2)$ and $x\leq \frac{|P|+1}{2}$, or $|X|\leq n-(|Q\setminus\{c\}|+2)$  and $x\leq\frac{|Q|}{2} $, giving in any case $x\leq\frac{n-|X|-1}{2}$ .
When $i=3$ and $t=4$, we have $|X|\leq n-(|P|+7)$ and $x\leq\frac{|P|+7}{2} $, or  $|X|\leq n-(|Q\setminus\{c\}|+7)$ and  $x\leq\frac{|Q|+6}{2}$. Hence $x\leq\frac{n-|X|}{2}$. Suppose we have the equality. Then $\Gamma$ contains one of following two graphs, having all vertices fixed by $\Gamma_{>x}$, except one of vertices $c$ or $d$ of graph (2).
$$(1)\; 
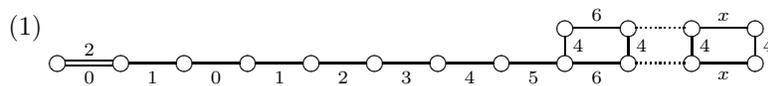
\xymatrix@-1.5pc{&& && && && && && && &&*+[o][F]{}\ar@{-}[rr]^6\ar@{-}[d]^4&& *+[o][F]{}\ar@{.}[rr]\ar@{-}[d]^4&&*+[o][F]{}\ar@{-}[rr]^{x}\ar@{-}[d]^4&&*+[o][F]{}\ar@{-}[d]^4\\
   *+[o][F]{}\ar@{=}[rr]_0^2&&*+[o][F]{}\ar@{-}[rr]_1&&*+[o][F]{}\ar@{-}[rr]_0&&*+[o][F]{}\ar@{-}[rr]_1&&*+[o][F]{}\ar@{-}[rr]_2&&*+[o][F]{}\ar@{-}[rr]_3&&*+[o][F]{}\ar@{-}[rr]_4&&*+[o][F]{}\ar@{-}[rr]_5&&*+[o][F]{}\ar@{-}[rr]_6&&*+[o][F]{} \ar@{.}[rr]&&*+[o][F]{}\ar@{-}[rr]_{x}&&*+[o][F]{}}$$    
or   
$$(2)\; 
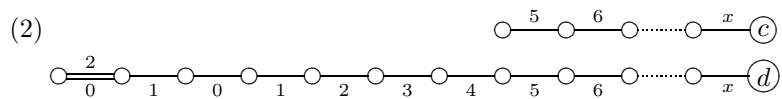
\xymatrix@-1.5pc{&& && && && && && &&*+[o][F]{}\ar@{-}[rr]^5&&*+[o][F]{}\ar@{-}[rr]^6&& *+[o][F]{}\ar@{.}[rr]&&*+[o][F]{}\ar@{-}[rr]^x&&*+[o][F]{c}\\
   *+[o][F]{}\ar@{=}[rr]_0^2&&*+[o][F]{}\ar@{-}[rr]_1&&*+[o][F]{}\ar@{-}[rr]_0&&*+[o][F]{}\ar@{-}[rr]_1&&*+[o][F]{}\ar@{-}[rr]_2&&*+[o][F]{}\ar@{-}[rr]_3&&*+[o][F]{}\ar@{-}[rr]_4&&*+[o][F]{}\ar@{-}[rr]_5&&*+[o][F]{}\ar@{-}[rr]_6&&*+[o][F]{} \ar@{.}[rr]&&*+[o][F]{}\ar@{-}[rr]_x&&*+[o][F]{d}}$$

Then as $x\neq r-1$ there is another vertex $v$, not fixed by $\Gamma_{>x}$, incident to one of the vertices of one of the two graphs above.
If $\Gamma$ contains the  graph (1), $v$ must be attached to it by a $4$-edge, which is not possible. This rules out graph (1).
In the second graph, the two components of the figure are not adjacent, since otherwise $c$ and $d$ are both fixed by $\Gamma_{>x}$. 
Suppose first that $c$ is  the vertex not fixed by $\Gamma_{>x}$.
Then there is another component $U_l$ adjacent to the $\Gamma_{>4}$-orbit containing the vertex $d$, as in the following figure. 
$$
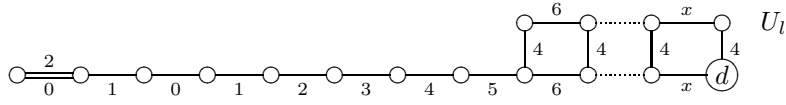
\xymatrix@-1.5pc{&& && && && && && && &&*+[o][F]{}\ar@{-}[rr]^6\ar@{-}[d]^4&& *+[o][F]{}\ar@{.}[rr]\ar@{-}[d]^4&&*+[o][F]{}\ar@{-}[rr]^{x}\ar@{-}[d]^4&&*+[o][F]{}\ar@{-}[d]^4&U_l\\
   *+[o][F]{}\ar@{=}[rr]_0^2&&*+[o][F]{}\ar@{-}[rr]_1&&*+[o][F]{}\ar@{-}[rr]_0&&*+[o][F]{}\ar@{-}[rr]_1&&*+[o][F]{}\ar@{-}[rr]_2&&*+[o][F]{}\ar@{-}[rr]_3&&*+[o][F]{}\ar@{-}[rr]_4&&*+[o][F]{}\ar@{-}[rr]_5&&*+[o][F]{}\ar@{-}[rr]_6&&*+[o][F]{} \ar@{.}[rr]&&*+[o][F]{}\ar@{-}[rr]_{x}&&*+[o][F]{d}&}$$
But $\Gamma_{>x}$ is fix-point-free in $U_l$, a contradiction. This shows that $x\leq\frac{n-|X|-1}{2}$. 
If $d$ is the vertex not fixed by $\Gamma_{>x}$, then the component adjacent to  the component having the vertex $c$ must be fixed by $\Gamma_{>x}$, giving a contradiction as before.

This finishes the case where some component does not have a 2-fracture graph.

We now consider that each group $G_s$ has a 2-fracture graph for $F_s$. Particularly each component $U_s$ with $F_s \ne \emptyset$ has at least four vertices. 
In addition let  $\mathcal{F}$ be a fracture graph satisfying the property (P1) and such that $E:=\{F_s\,|\, F_s=\emptyset, s\in\{1,\ldots,c\}\}$ has maximal size.
Denote by $S$ and $\delta_s$ the following numbers.
$$\delta_s= |F_s|-\frac{|U_s|}{2}\mbox{ and }  S:=\displaystyle\sum_{s=1}^c\delta_s.$$ 
As $G_s$ has a 2-fracture graph for $F_s$, $\delta_s\leq 0$ for all $s\in\{1,\ldots, c\}$.
Let $U_p$ be the component such that $t+1\in F_p$. 
In Proposition~\ref{t} we proved that $(G_s)_{t+1}$ cannot be transitive in $U_s$. Therefore a fracture graph for $G_s$ with $s\neq p$
is disconnected. Hence $\delta_s\leq -0.5$ for $s\neq p$.
If $S\leq -3$, then $r-1-t\leq \frac{|U|}{2}-3$. For $(i,t)=(0,1)$, $|U|\leq n-2$ hence $r-2\leq \frac{n-2}{2}-3$. For $(i,t)=(3,4)$, $|U|\leq n-7$ hence $r-5\leq \frac{n-7}{2}-3$. In any case $r\leq \frac{n-3}{2}$. In what follows we prove that either $S\leq -3$ or we have (a) of the statement of this proposition. 

Note that if $F_s\neq \emptyset$, then as $G_s$ has a 2-fracture graph for $F_s$, $|U_s|$ has at least four vertices.

As $\rho_i$ is an even permutation and fixes the first $\Gamma_i$-orbit except the vertex $a$, it must act nontrivially as an odd permutation in $U$.
In what follows we consider separately the following cases: (1) $\rho_i$ swaps an odd number of pairs of vertices $(v,w )$ with $v\in U_s$ and $w\in U_x$ with $s\neq x$;
(2) $\rho_i$ acts as an odd permutation inside a component $U_s$ with  $F_s\neq \emptyset$ and $s\neq p$; (3) $\rho_i$ acts as an odd permutation inside $U_p$; (4)  $\rho_i$ acts as an odd permutation inside a component $U_s$ with $F_s= \emptyset$.

(1) If $\rho_i$ swaps an odd number of pairs of vertices $(v,w )$ with $v\in U_s$ and $w\in U_x$, then $|U_s|=|U_x|$ is odd. If $|U_x|=|U_s|=3$ then $F_x=F_s=\emptyset$, hence $S\leq -3$. Consider $|U_x|=|U_s|\geq 5$.
As $|E|$ is maximal, either $F_x=\emptyset$ or $F_s=\emptyset$.   We may assume that $|F_x|=0$. Then $|F_s|< \frac{5}{2}$.
Thus  $\delta_x\leq -2,5$ and $\delta_s\leq -0,5$. Hence $S\leq -3$.

(2) Consider now that  $\rho_i$ acts as an odd permutation inside a $\Gamma_{>t}$-orbit $U_s$ with $F_s\neq \emptyset$ and $s\neq p$.
In this case $\rho_i$ centralizes $G_s$, therefore  $|U_s|$ is even and $|U_s|\geq 6$. 
Moreover there is a  2-fracture graph for $G_s$ with labels in $S:=F_s\cup\{t+1\}$ being disconnected and having no cycles.
Hence $|F_s|+1\leq \frac{|U_s|-2}{2}$, thus $\delta_s\leq -2$.
Suppose  $\delta_s=-2$. In that case the permutation representation graph of $G_s$ is as follows, where $x$ is the maximal label in $U_s$ and $k\in\{t+1,\ldots, x\}$. 

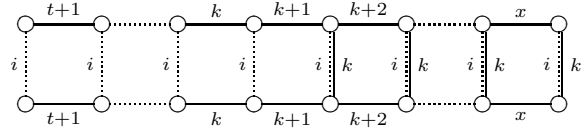
$$ \xymatrix@-1.3pc{*+[o][F]{} \ar@{-}[rr]^{t+1} \ar@{.}[dd]_i &&
 *+[o][F]{}\ar@{.}[rr] \ar@{.}[dd]_i&&
 *+[o][F]{} \ar@{-}[rr]^k\ar@{.}[dd]_i&&
 *+[o][F]{} \ar@{-}[rr]^{k+1}\ar@{.}[dd]_i&&
 *+[o][F]{}\ar@{-}[rr]^{k+2}\ar@{.}[dd]_i\ar@<.3ex>@{-}[dd]^k && 
 *+[o][F]{}\ar@{.}[dd]_i\ar@<.3ex>@{-}[dd]^k  \ar@{.}[rr] &&
 *+[o][F]{} \ar@{-}[rr]^x \ar@{.}[dd]_i \ar@<.3ex>@{-}[dd]^k && 
 *+[o][F]{}\ar@{.}[dd]_i\ar@<.3ex>@{-}[dd]^k \\
&& && && && && && &&\\
*+[o][F]{} \ar@{-}[rr]_{t+1}  && *+[o][F]{}\ar@{.}[rr] &&*+[o][F]{} \ar@{-}[rr]_k && *+[o][F]{}\ar@{-}[rr]_{k+1}&& *+[o][F]{}\ar@{-}[rr]_{k+2} &&*+[o][F]{} \ar@{.}[rr] &&*+[o][F]{}\ar@{-}[rr]_x && *+[o][F]{}}
  $$
Suppose that $x\neq r-1$. Then $\Gamma_{>x}$ fixes $U_s$, the vertex $b$ and the first $\Gamma_i$-orbit. Hence there exists $x\in\{0,\ldots, r-1\}$ as in statement (a) of this proposition. Thus we may consider that $\delta_s< -2$. As in this case $|U_s|$ is even, $\delta_s$ is an integer thus $\delta_s\leq -3$ and $S\leq -3$.

If $x=r-1$, we have $r-1-t\leq \frac{|U_s|}{2}-1$. For $(i,t)=(0,1)$, $|U_s|\leq (n-2)-4$ where $4$ is the minimal size of $U_p$. For $(i,t)=(3,4)$,  $|U_s|\leq (n-7)-4$.
In any case we get $r\leq \frac{n-3}{2}$.

(3) If  $\rho_i$ acts as an odd permutation inside $U_p$, then a 2-fracture graph of $G_p$ is disconnected  without cycles. 
Hence $|F_p|\leq \frac{|U_p|-2}{2}$, that is $\delta_p\leq -1$.  Suppose we have the equality $|F_p|=\frac{|U_p|-2}{2}$. Then the permutation representation graph of $G_s$ is as the permutation representation graph given in case (2).
The only difference is that in this case  $t+1\in F_p$. In this case we get exactly the same result as before, that is statement (a) of this proposition.
Assume now that $\delta_p<-1$. As in case (2), $|U_p|$ is even,  thus $\delta_p$ is an integer. Then we may assume that $\delta_p\leq -2$. 

Suppose that $t+2\notin F_p$. If $t+2\notin  I_p$ or $(G_p)_{t+2}$ is transitive in $U_p$, then $\rho_{t+1}$ also centralizes $G_p$. But then there exists $l\in I_p\cap I_x$ with $I_x$  being the set of labels of a component $U_x$ adjacent to $U_p$. Then a 2-fracture graph of $G_x$ has at least three components, hence $\delta_x\leq -1$. Thus $\delta_p+\delta_s\leq -3$ and $|S|\leq -3$. If $t+2\in I_p$ and $(G_p)_{t+2}$ is intransitive in $U_p$, then $G_p$ has a disconnected 2-fracture graph without cycles for $F_s\cup \{\rho_{t+2}\}$. Hence $2(|F_p|+1)\leq |U_p|-2$. Moreover if we have the equality then $G_s$ is as the permutation representation graph given in case (2). Hence $2(|F_p|+1)< |U_p|-2$ and $\delta_p\leq -3$. 

We now consider that $t+2\in F_p$.
Suppose $|S|=-2.5$. Then $c=2$  and the second component $U_s$ is such that $\delta_s= -0.5$. In particular, $F_s\neq \emptyset$, for otherwise $\delta_s\leq -1$.
First suppose that $(G_s)_{t+2}$ is transitive in $U_s$. Then $\rho_{t+1}$ centralizes $G_s$. Hence $2|F_s|\leq |U_s|-2$ and $\delta_s\leq -1$, a contradiction.
Thus $G_{t+2}$ is intransitive in $U_s$. As $t+2\notin F_s$, a 2-fracture graph for $G_s$, with  labels in $F_s$, has at least 3 components with at most one cycle, hence $2|F_s|\leq |U_s|-2$, as before, a contradiction.

(4) Suppose that $\rho_i$ acts as an odd permutation inside a component $U_s$ with $F_s=\emptyset$. Either  $|U_s|=2$ or $|U_s|\geq 6$. But in the second case clearly  $\delta_s\leq -3$, thus we consider $U_s=\{u,v\}$.  Let $l$  be the number of components of size two, whose vertices are swapped by $\rho_i$. We need to consider the case $l$ odd. If $l\geq 3$ then $S\leq -3$, hence we assume $l=1$. Moreover let $\rho_i$ be a 2-transposition, $\rho_i=(a\,b)(u\,v)$. Then $I_s=\{t+1\}$.
$$ \xymatrix@-.1pc{
    *+[o][F]{}   \ar@{-}[r]^t     & *+[o][F]{u}  \ar@{=}[r]^{t+1}_{i=t-1}   & *+[o][F]{v}   \ar@{-}[r]^t      & *+[o][F]{}   }$$
    
If $t+1=r-1$ then $\Gamma$ has one of the following permutation representation graphs. 

$(i,t)=(0,1):$
$$ 
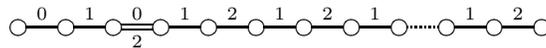
\xymatrix@-1pc{
 *+[o][F]{}   \ar@{-}[r]^0     & *+[o][F]{}  \ar@{-}[r]^1   & *+[o][F]{}   \ar@{=}[r]^0_2      & *+[o][F]{}  \ar@{-}[r]^1   & *+[o][F]{}  \ar@{-}[r]^2 & *+[o][F]{}  \ar@{-}[r]^1 & *+[o][F]{}  \ar@{-}[r]^2& *+[o][F]{}  \ar@{-}[r]^1  & *+[o][F]{} \ar@{.}[r] &*+[o][F]{} \ar@{-}[r]^1 & *+[o][F]{}  \ar@{-}[r]^2&*+[o][F]{} } $$

$(i,t)=(3,4):$
$$ \xymatrix@-1pc{
    *+[o][F]{}   \ar@{-}[r]^0     & *+[o][F]{}  \ar@{-}[r]^1   & *+[o][F]{}   \ar@{=}[r]^0_2      & *+[o][F]{}  \ar@{-}[r]^1   & *+[o][F]{}  \ar@{-}[r]^2 & *+[o][F]{}  \ar@{-}[r]^3 & *+[o][F]{}  \ar@{-}[r]^4& *+[o][F]{}  \ar@{=}[r]^3_5 &*+[o][F]{}  \ar@{-}[r]^4 & *+[o][F]{}   \ar@{-}[r]^5 & *+[o][F]{} \ar@{-}[r]^4 & *+[o][F]{}   \ar@{-}[r]^5 & *+[o][F]{} \ar@{-}[r]^4&*+[o][F]{} \ar@{-}[r]^5  & *+[o][F]{} \ar@{.}[r] &*+[o][F]{} \ar@{-}[r]^5 & *+[o][F]{}   \ar@{-}[r]^4 & *+[o][F]{} }$$
As by hypotheses $\Gamma_{>i}$ has a 2-fracture graph, for $(i,t)=(0,1)$, we have $n\geq 9 $, and for $(i,t)=(3,4)$, we have $n\geq 15$. In both cases $r\leq \frac{n-3}{2}$.

Assume $t+1\neq r-1$.
In this case $\Gamma_{>t+1}$ fixes the first $\Gamma_i$-orbit and $\{b,u,v,u\rho_t, v\rho_t\}$. 

If $|\{b,u,v,u\rho_t, v\rho_t\}|=5$, consider $x=t+1$. Then $x$ satisfies (a) of this proposition.
Otherwise, $|\{b,u,v,u\rho_t, v\rho_t\}|=4$. Then $\Gamma$ contains one of the following graphs.

\vspace{5pt}

For $(i,t)=(0,1)$:
$$ \xymatrix@-1pc{
    *+[o][F]{}   \ar@{-}[r]^0     & *+[o][F]{}  \ar@{-}[r]^1   & *+[o][F]{}   \ar@{=}[r]^0_2      & *+[o][F]{}  \ar@{-}[r]^1   & *+[o][F]{}  \ar@{-}[r]^2 & *+[o][F]{}    \ar@{.}[r] &  *+[o][F]{}  \ar@{-}[r]^2 & *+[o][F]{}\ar@{.}[r]& *+[.][F]{} }$$
    
For $(i,t)=(3,4)$:
$$ \xymatrix@-1pc{
    *+[o][F]{}   \ar@{-}[r]^0     & *+[o][F]{}  \ar@{-}[r]^1   & *+[o][F]{}   \ar@{=}[r]^0_2      & *+[o][F]{}  \ar@{-}[r]^1   & *+[o][F]{}  \ar@{-}[r]^2 & *+[o][F]{}  \ar@{-}[r]^3 & *+[o][F]{}  \ar@{-}[r]^4& *+[o][F]{}  \ar@{=}[r]^3_5 &*+[o][F]{}  \ar@{-}[r]^4 & *+[o][F]{}   \ar@{-}[r]^5 & *+[o][F]{}  \ar@{.}[r] & *+[o][F]{}  \ar@{-}[r]^5 & *+[o][F]{}\ar@{.}[r]& *+[.][F]{}}$$
In the first case, $\Gamma_{>2}$ has at least five fixed points. Therefore $x=2$ satisfies the statement (a) of this proposition. 
Consider the second case. As $B=\Gamma_{>3}\cong A_{n-6}$,  $\Gamma_{>2}\cong A_{n-5}$. In addition $\Gamma_{<5}\cong A_{10}$ or $\Gamma_{<5}\cong A_{10}\times \langle \tau\rangle $ where $\tau$ is an even involution. Thus $\Gamma_{>2}\cap\Gamma_{<5}$ is either $A_5$ or $A_5\times \langle \tau\rangle$. But $\langle \rho_3,\rho_4\rangle\cong D_5\times \langle \tau\rangle$, contradicting the intersection condition.
\end{proof}

\begin{prop}\label{specials}
Let $n\geq 12$. If $A$ is trivial or $A$ has the permutation representation graph (2) or (3) of Table~\ref{small} and $B$ is an alternating group, then $r\leq \frac{n-1}{2}$.
\end{prop}
\begin{proof}
Assume for contradiction that $r > \frac{n-1}{2}$.
By the dual of Proposition~\ref{ineq0Btrivial} we may consider, up to duality, that when $A$ is trivial, $i=0$.
By Propositions~\ref{0trivialtran}, \ref{12Tran} and  \ref{specialbubbles} we may assume that $\Gamma_{>i+1}$ is intransitive on $\{1,\ldots,n\}\setminus  {\rm Fix}(\Gamma_{>i+1})$, but  $\Gamma_{>i}$ does not have a 2-fracture graph. In addition, by Propositions~\ref{Case2} and \ref{case3} we can consider that if $\Gamma_j$ has exactly two orbits and a single $j$-edge connecting them, then the group orbits are either $C=\Gamma_{r-1}\cong A_{n-1}$ and $D$ trivial, or $C=\Gamma_{<j}\cong A_{n-6}$ and $D=\Gamma_{>j}\cong A_5$, $D$  having the permutation representation graph dual of (2) or the graph (3). 

Consider the group $C$. Let $r_C$ be the rank of $C$ and $n_C$ the degree of $C$. Thanks to the intersection condition $\Gamma_{i+1,\ldots,j-1}$ is the alternating group.
In addition $C_{>i+1}$ is intransitive on the second $C_i$-orbit (the orbit containing the vertex $b$), for otherwise $\Gamma_{>i+1}$ is transitive on $\{1,\ldots,n\}\setminus  {\rm Fix}(\Gamma_{>i+1})$. Thus $C$ satisfies the condition of Proposition~\ref{specialbubbles}.
Accordantly with that proposition there are three possibilities. 
The first one is $r_C\leq \frac{n_C-3}{2}$ which implies $r\leq \frac{n-1}{2}$, a contradiction.  
The second one gives  the following permutation representation graph for $C$.
$$\xymatrix@-1.3pc{ &&*+[o][F]{} \ar@{-}[rr]^1&&*+[o][F]{} \ar@{-}[rr]^2\ar@{-}[d]^0&&*+[o][F]{} \ar@{-}[rr]^3\ar@{=}[d]^1_0&&*+[o][F]{}\ar@{=}[d]^1_0 \ar@{.}[rr]\ar@{=}[d]^1_0&&*+[o][F]{}\ar@{-}[rr]^{r_C-1}\ar@{=}[d]^1_0&&*+[o][F]{}\ar@{=}[d]^1_0\\
   *+[o][F]{}\ar@{-}[rr]_0&&*+[o][F]{}\ar@{-}[rr]_1&&*+[o][F]{}\ar@{-}[rr]_2&&*+[o][F]{} \ar@{-}[rr]_3&&*+[o][F]{} \ar@{.}[rr]&&*+[o][F]{}\ar@{-}[rr]_{r_C-1}&&*+[o][F]{}}$$ 
But then, it is not possible to attach the permutation representation graph of $D$ by a single $j$-edge, a contradiction. 
The last establishes that 
there exists $x>i$ such that $x\leq \frac{n-|X|-1}{2}$ with $X:=\{1,\ldots, n\}\setminus {\rm Fix}(\Gamma_{>x})$.
On the other hand the dual of Proposition~\ref{specialbubbles} gives the same three possibilities for $B$ and, as before, one of them gives $r\leq \frac{n-1}{2}$, and the second gives a contradiction. Thus it may be assumed that there exists $y>i$ such that $r-y\leq \frac{n-|Y|-1}{2}$ with $X:=\{1,\ldots, n\}\setminus {\rm Fix}(\Gamma_{<y})$. If $\Gamma_{\{x+1,\ldots,y-1\}}$ is intransitive on $X\cap Y=\{1,\ldots, n\}\setminus {\rm Fix}(\Gamma_{\{x+1,\ldots,y-1\}})$ then it has a disconnected 2-fracture graph. Hence $y-1-x\leq \frac{|X\cap Y|-1}{2}$. Otherwise, as $\Gamma_{\{x+1,\ldots,y-1\}}$ has a 2-fracture graph it cannot be one of the graphs (2) or (3) of Table~\ref{small}. Thus by Propositions~\ref{small123} and \ref{induction}, we also have  $y-1-x\leq \frac{|X\cap Y|-1}{2}$. As $n= |X| + |Y| - |X \cap Y|$, we get $r\leq \frac{n-1}{2}$, a contradiction.
\end{proof}

The cases we have covered, that some $\Gamma_i$ is primitive, or transitive
imprimitive, or all $\Gamma_i$ are intransitive and $2$-fracture graphs do or do
not exist, exhaust all possibilities; so Theorem~\ref{maintheorem} is proved.

\section{Acknowledgements}
This research was supported by a Marsden grant (UOA1218) of the Royal Society of New Zealand, and by the Portuguese Foundation for Science and Technology (FCT-Funda\c{c}\~{a}o para a Ci\^{e}ncia e a Tecnologia), through CIDMA - Center for Research and Development in Mathematics and Applications, within project UID/MAT/04106/2013.

\bibliographystyle{amsplain}

\end{document}